\documentclass[12pt, reqno]{amsart}
\makeatletter
\DeclareMathAlphabet{\mathpzc}{OT1}{pzc}{m}{it}

\usepackage[margin=2cm]{geometry}
\usepackage{amsmath, amssymb, amsfonts, amstext, verbatim, amsthm, mathrsfs}
\usepackage{microtype}
\usepackage[all,arc,knot,poly]{xy}
\usepackage{mathtools}
\usepackage{aliascnt} 
\usepackage[colorlinks=true,linkcolor=blue,citecolor=blue,urlcolor=blue,citebordercolor={0 0 1},urlbordercolor={0 0 1},linkbordercolor={0 0 1}]{hyperref} 
\usepackage{enumerate}
\usepackage{stmaryrd}
\usepackage{tikz}
\usepackage{pgfplots}
\usepgfplotslibrary{polar}

\theoremstyle{plain}

\newcommand{\refnewtheoremn}[4]{
\newaliascnt{#1}{#2}
\newtheorem{#1}[#1]{#3}
\aliascntresetthe{#1}
\expandafter\providecommand\csname #1autorefname\endcsname{#4}}

\newcommand{\refnewtheorem}[3]{\refnewtheoremn{#1}{#2}{#3}{#3}}

\newtheorem{theorem}{Theorem}[section]
\refnewtheorem{lemma}{thm}{Lemma}
\refnewtheorem{claim}{thm}{Claim}
\refnewtheorem{corollary}{thm}{Corollary}
\refnewtheorem{conj}{thm}{Conjecture}
\refnewtheorem{prop}{thm}{Proposition}
\refnewtheorem{proposition}{thm}{Proposition}
\refnewtheorem{quest}{thm}{Question}
\refnewtheorem{assumption}{thm}{Assumption}
\refnewtheorem{problem}{thm}{Problem}

\theoremstyle{definition}
\refnewtheorem{remark}{thm}{Remark}
\refnewtheorem{remarks}{thm}{Remarks}
\refnewtheorem{defn}{thm}{Definition}
\refnewtheorem{definition}{thm}{Definition}
\refnewtheorem{notn}{thm}{Notation}
\refnewtheorem{const}{thm}{Construction}
\refnewtheorem{example}{thm}{Example}
\refnewtheorem{examples}{thm}{Examples}
\refnewtheorem{fact}{thm}{Fact}

\DeclareMathOperator{\Crit}{\operatorname{Crit}}
\DeclareMathOperator{\MCG}{\operatorname{MCG}}
\DeclareMathOperator{\Res}{\operatorname{Res}}
\DeclareMathOperator{\Hom}{\operatorname{Hom}}
\DeclareMathOperator{\PSL}{\operatorname{PSL}}
\DeclareMathOperator{\Cr}{\operatorname{Cr}}
\DeclareMathOperator{\Q}{\operatorname{Q}}
\DeclareMathOperator{\dist}{\operatorname{dist}}
\DeclareMathOperator{\sgn}{\operatorname{sgn}}
\DeclareMathOperator{\Tilt}{\operatorname{Tilt}}
\DeclareMathOperator{\Aut}{\operatorname{Aut}}
\DeclareMathOperator{\Nil}{\operatorname{Nil}}
\DeclareMathOperator{\Tw}{\operatorname{Tw}}
\DeclareMathOperator{\Sph}{\operatorname{Sph}}
\DeclareMathOperator{\Exch}{\operatorname{Exch}}
\DeclareMathOperator{\Tri}{\operatorname{Tri}}
\DeclareMathOperator{\Stab}{\operatorname{Stab}}

\newcommand{\cAut}{\mathpzc{Aut}}
\newcommand{\cSph}{\mathpzc{Sph}}

\renewcommand{\Re}{\operatorname{Re}}
\renewcommand{\Im}{\operatorname{Im}}

\begin{document}

\title{Stability conditions and Teichm\"uller space}
\author{Dylan G.L. Allegretti}

\date{}

\maketitle

\begin{abstract}
We consider a 3-Calabi-Yau triangulated category associated to an ideal triangulation of a marked bordered surface. Using the theory of harmonic maps between Riemann surfaces, we construct a natural map from a component of the space of Bridgeland stability conditions on this category to the enhanced Teichm\"uller space of the surface. We describe a relationship between the central charges of objects in the category and shear coordinates on the Teichm\"uller space.
\end{abstract}

\section{Introduction}

Recently, a remarkable connection has emerged between low-dimensional topology and the theory of triangulated categories. A series of works by Bridgeland and Smith~\cite{BridgelandSmith15}, Haiden, Katzarkov, and Kontsevich~\cite{HKK17}, and Haiden~\cite{Haiden21} has revealed that spaces of stability conditions on various kinds of triangulated categories can be identified with moduli spaces of quadratic differentials on surfaces. Inspired by these developments, a number of authors have proposed categorical analogs of familiar notions from dynamics and Teichm\"uller theory~\cite{BapatDeopurkarLicata20,DHKK14,Fan21,FanFilip20,FFHKL21,FanFuOuchi21}.

In the present paper, we aim to clarify the relationship between the space of stability conditions on a triangulated category and the Teichm\"uller space of a surface. We consider a large class of 3-Calabi-Yau triangulated categories associated to triangulated surfaces in the work of Bridgeland and Smith~\cite{BridgelandSmith15}. Using the theory of harmonic maps between Riemann surfaces, we construct a map from a component of the space of stability conditions on such a category to a version of Teichm\"uller~space known in the literature as the enhanced Teichm\"uller space. In addition, we describe an asymptotic relationship between the central charges of objects in the category and shear coordinates on the enhanced Teichm\"uller space.

The starting point for our construction is a result of Hitchin~\cite{Hitchin87} and Wolf~\cite{Wolf89}, which uses holomorphic quadratic differentials to parametrize the Teichm\"uller space of a compact surface. In the course of proving our main theorems, we obtain a meromorphic generalization of their result. The present paper is also strongly influenced by the work of Gaiotto, Moore, and Neitzke in physics~\cite{GaiottoMooreNeitzke13}. In an earlier series of papers~\cite{Allegretti18,AllegrettiBridgeland20,Allegretti19,Allegretti20,Allegretti21}, we gave a mathematical formalization of their ideas in a particular limit called the conformal limit. In contrast, the present paper can be seen as a first step toward understanding the ideas of Gaiotto, Moore, and Neitzke away from the conformal limit.

By the work of Fock and Goncharov~\cite{FockGoncharov1,FockGoncharov2}, the enhanced Teichm\"uller space arises as the set of $\mathbb{R}_{>0}$-valued points of a certain cluster variety. The latter can be constructed starting from the data of a quiver. Using these facts, we can give a formulation of our main results entirely in terms of the 3-Calabi-Yau triangulated category associated to a quiver with potential. We conjecture that the construction presented here is a special case of a far more general construction valid for such categories.

\subsection{Compact surfaces and holomorphic differentials}

In the following discussion, $S$ will denote a fixed compact Riemann~surface satisfying $\chi(S)<0$. If $C$ is a surface equipped with a complete finite area hyperbolic metric, then a marking of~$C$ by~$S$ is defined to be a diffeomorphism $\psi:S\rightarrow C$. In this case, the pair~$(C,\psi)$ is called a marked hyperbolic surface. Two marked hyperbolic surfaces $(C_1,\psi_1)$ and~$(C_2,\psi_2)$ are considered to be equivalent if there is an isometry $g:C_1\rightarrow C_2$ such that $\psi_2$ is homotopic to~$g\circ\psi_1$, and the Teichm\"uller space $\mathcal{T}(S)$ can be defined as the space of equivalence classes of marked hyperbolic surfaces of this type. It is a real manifold homeomorphic to a Euclidean space of dimension $-3\cdot\chi(S)$.

We will be interested in marked hyperbolic surfaces for which the marking is harmonic. As we will review in Section~\ref{sec:HarmonicMaps}, a map $f:M\rightarrow N$ of Riemannian manifolds is harmonic if it satisfies a certain Euler-Lagrange equation. In particular, we can consider the case where $M$ and~$N$ are Riemann~surfaces and the metrics are given locally by expressions of the form $g|dz|^2$ and $h|dw|^2$ for some positive functions $g$ and $h$, where $z$ and $w$ are complex local coordinates on~$M$ and~$N$, respectively. In this case, a map $f:M\rightarrow N$ is harmonic if and only if it satisfies the partial differential equation 
\[
\frac{\partial^2f}{\partial z\partial\bar{z}}+\frac{1}{h}\frac{\partial h}{\partial w}\frac{\partial f}{\partial z}\frac{\partial f}{\partial\bar{z}}=0.
\]
Note that while this condition depends on the exact form of the metric on~$N$, it only depends on the conformal structure of~$M$. If $M$ and~$N$ are any compact Riemannian manifolds and the metric on~$N$ has negative curvature, then the work of Eells and Sampson~\cite{EellsSampson64} implies that there exists a harmonic map in the homotopy class of any diffeomorphism $M\rightarrow N$. Hartman~\cite{Hartman67} proved that such a map is unique, and Schoen and Yau~\cite{SchoenYau78} and Sampson~\cite{Sampson78} independently proved that it is a diffeomorphism in the case where $M$ and~$N$ are surfaces.

Recall that a holomorphic quadratic differential on a Riemann surface~$S$ is defined as a holomorphic section of~$\omega_S^{\otimes2}$, where $\omega_S$ denotes the holomorphic cotangent bundle of~$S$. If $\psi:S\rightarrow C$ is a harmonic map and $h$ is the metric on~$C$, then the $(2,0)$-part of the pullback $\psi^*h$ defines a holomorphic quadratic differential on~$S$ called the Hopf~differential of~$\psi$. Our goal in this paper is to generalize the following theorem of Wolf.

\begin{theorem}[\cite{Wolf89}]
\label{thm:Wolfcompact}
There is a homeomorphism 
\[
\Phi^S:\mathcal{T}(S)\stackrel{\sim}{\longrightarrow}H^0(S,\omega_S^{\otimes2})
\]
taking a point $(C,\psi)\in\mathcal{T}(S)$ to the Hopf differential of the unique harmonic diffeomorphism homotopic to~$\psi$.
\end{theorem}

Theorem~\ref{thm:Wolfcompact} is equivalent to a result of Hitchin from~\cite{Hitchin87}. In modern terminology, a quadratic differential determines a Higgs bundle on the Hitchin section. By applying the nonabelian Hodge~correspondence to this Higgs bundle, one obtains a point in the $SL_2(\mathbb{C})$-character variety which can be shown to lie in the Teichm\"uller space. This defines a homeomorphism $H^0(S,\omega_S^{\otimes2})\stackrel{\sim}{\rightarrow}\mathcal{T}(S)$ which is the inverse of~$\Phi^S$.

\subsection{Noncompact surfaces and meromorphic differentials}

In this paper, we generalize Theorem~\ref{thm:Wolfcompact} by considering surfaces of the form 
\begin{equation}
\label{eqn:cuspsboundary}
C=\bar{C}\setminus\left(\bigcup_{i=1}^sD_i\cup\bigcup_{j=1}^t\{p_j\}\right)
\end{equation}
where $\bar{C}$ is a closed oriented surface, $D_1,\dots,D_s\subset\bar{C}$ are open disks whose closures are disjoint, and $p_1,\dots,p_t\in\bar{C}$ are points disjoint from the interiors of the~$D_i$. We equip $C$ with a complete, finite area hyperbolic metric with totally geodesic boundary. By completeness of the metric, there is a cusp neighborhood around each of the deleted points~$p_j$. Note that if a point $p_j$ lies on the boundary of some disk~$D_i$, then the boundary of~$C$ will contain at least one component which is a bi-infinite totally geodesic arc (see Figure~\ref{fig:surface}).

\begin{figure}[ht]
\begin{center}
\begin{tikzpicture}[scale=0.55]
\begin{scope}
\clip[rotate=20](4,-0.5) rectangle (4.25,0.5);
\draw[black, thin, rotate=20] (4,0) ellipse (0.25 and 0.5);
\end{scope}
\begin{scope}
\clip[rotate=20](3.75,-0.5) rectangle (4,0.5);
\draw[black, thin, dotted, rotate=20] (4,0) ellipse (0.25 and 0.5);
\end{scope}
\draw[thin] (-1,0.1) .. controls (-0.75,-0.15) and (-0.5,-0.25) .. (0,-0.25) .. controls (0.5,-0.25) and (0.75,-0.15) .. (1,0.1);
\draw[thin] (-0.875,0) .. controls (-0.75,0.15) and (-0.5,0.25) .. (0,0.25) .. controls (0.5,0.25) and (0.75,0.15) .. (0.875,0);
\draw[thin] (-5,3) .. controls (-2,0.3) and (2,0.5) .. (3.6,1.85);
\draw[thin] (-5,-3) .. controls (-2,-0.3) and (2,-0.5) .. (4,-1.5);
\draw[thin] (3.95,0.9) .. controls (2.75,0.25) and (2,-0.15) .. (4,-1.5);
\draw[black, thin] (-5,-3) .. controls (-2.6,-1) and (-2.6,1) .. (-5,3);
\draw[black, thin] (-5.2,-1) .. controls (-4,-0.6) and (-2.6,0.8) .. (-5,3);
\draw[black, thin] (-5,-3) .. controls (-3.8,-1.6) and (-3,-0.8) .. (-5.2,-1);
\end{tikzpicture}
\end{center}
\caption{A hyperbolic surface of the form~\eqref{eqn:cuspsboundary}.\label{fig:surface}}
\end{figure}
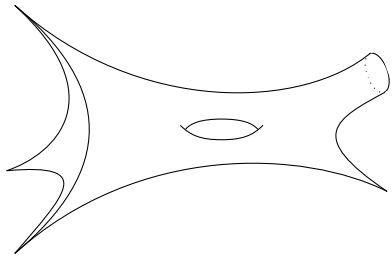

We define a marked bordered surface to be a pair $(\mathbb{S},\mathbb{M})$ consisting of a compact oriented surface~$\mathbb{S}$ with (possibly empty) boundary together with a nonempty finite subset $\mathbb{M}\subset\mathbb{S}$ of marked points such that each component of~$\partial\mathbb{S}$ contains at least one marked point. We will write $\mathbb{P}\subset\mathbb{M}$ for the set of punctures, defined as marked points that lie in the interior of~$\mathbb{S}$. If $C$ is a hyperbolic surface of the form~\eqref{eqn:cuspsboundary}, then a marking of~$C$ by~$(\mathbb{S},\mathbb{M})$ is defined to be a diffeomorphism $\psi:\mathbb{S}\setminus\mathbb{M}\rightarrow C^\circ$ where $C^\circ$ denotes the complement in~$C$ of all boundary components homeomorphic to~$S^1$. Then the pair $(C,\psi)$ is called a marked hyperbolic surface. We can define equivalence of marked hyperbolic surfaces exactly as we did for closed surfaces, and we write $\mathcal{T}(\mathbb{S},\mathbb{M})$ for the space of equivalence classes of marked hyperbolic surfaces of this type.

If $(C,\psi)\in\mathcal{T}(\mathbb{S},\mathbb{M})$ is a marked hyperbolic surface, then we will define a tuple $L=(L_i)$ of numbers, called length data, which are indexed by the punctures and boundary components of~$(\mathbb{S},\mathbb{M})$. If $\partial\mathbb{S}=\emptyset$ then these numbers simply give the lengths of the geodesic boundary components of~$C$, while if $\partial\mathbb{S}\neq\emptyset$ then there is an additional number, called a metric residue, associated to each boundary component of~$\mathbb{S}$. We write 
\[
\mathcal{T}_L(\mathbb{S},\mathbb{M})\subset\mathcal{T}(\mathbb{S},\mathbb{M})
\]
for the set of points with fixed length data given by~$L$.

We also consider a space parametrizing meromorphic quadratic differentials. To define it, let $S$ be a compact Riemann surface and $M=\{p_1,\dots,p_d\}\subset S$ a nonempty finite subset. Fix a local coordinate~$z_i$ defined in a neighborhood of $p_i\in S$ so that $z_i(p_i)=0$. Then for any meromorphic quadratic differential~$\phi$ on~$S$ with poles at the points of~$M$, we define a tuple $P=(P_i(z_i))$ of expressions, called the principal parts of~$\phi$, which are indexed by the points of~$M$. The principal part $P_i(z_i)$ comprises the terms of degree $\leq-1$ in the expansion of a square root $\sqrt{\phi}$ in the variable~$z_i$. The tuple $P=(P_i(z_i))$ determines a corresponding divisor $D=\sum_im_ip_i$ which is the polar divisor for any quadratic differential having principal part~$P_i(z_i)$ at each~$p_i$. We write  
\[
\mathcal{Q}_P(S,M)\subset H^0(S,\omega_S^{\otimes2}(D))
\]
for the space of meromorphic quadratic differentials with fixed principal parts given by~$P$.

As we explain in Section~\ref{sec:CompatibilityConditions}, the principal part $P_i(z_i)$ determines a collection of distinguished tangent directions at~$p_i$ whenever $m_i\geq3$. Using this fact, we can define an associated marked bordered surface $(S_P,M_P)$ in a natural way. The surface $S_P$ is obtained by taking an oriented real blowup of~$S$ at each $p_i\in M$ for which $m_i\geq3$. The distinguished tangent directions determine a collection of marked points on the boundary of the resulting surface, and we define $M_P$ to consist of these points together with all points $p_i\in M$ for which $m_i\leq2$. The tuple $P$ of principal parts also determines a corresponding tuple $L$ of length data, and we write $\mathcal{T}_P(S,M)$ for the space $\mathcal{T}_L(S_P,M_P)$. By construction, $S\setminus M$ can be viewed as a subsurface of $S_P\setminus M_P$, and so if $(C,\psi)\in\mathcal{T}_P(S,M)$, then $\psi$ restricts to a map $S\setminus M\rightarrow C\setminus\partial C$. In Section~\ref{sec:AMapBetweenModuliSpaces}, we prove the following extension of Theorem~\ref{thm:Wolfcompact}.

\begin{theorem}
\label{thm:introgeneralizeWolf}
If $(C,\psi)\in\mathcal{T}_P(S,M)$ then there is a unique harmonic diffeomorphism $F_\psi:S\setminus M\rightarrow C\setminus\partial C$ homotopic to~$\psi$ whose Hopf differential lies in~$\mathcal{Q}_P(S,M)$. There is a homeomorphism 
\[
\Phi_P^{S,M}:\mathcal{T}_P(S,M)\stackrel{\sim}{\longrightarrow}\mathcal{Q}_P(S,M)
\]
taking $(C,\psi)$ to the Hopf differential of this map~$F_\psi$.
\end{theorem}

Theorem~\ref{thm:introgeneralizeWolf} builds on earlier results of a number of different authors. In particular, Wolf~\cite{Wolf91} proved related results concerning harmonic maps between nodal Riemann surfaces. Lohkamp~\cite{Lohkamp91} proved Theorem~\ref{thm:introgeneralizeWolf} in the special case where the Hopf differentials have at most simple poles. Gupta~\cite{Gupta19} proved the case where the Hopf differentials have only poles of order~$\geq3$, and  Sagman~\cite{Sagman19} further studied the case where the Hopf differentials have double poles. From the perspective of the theory of Higgs bundles, Theorem~\ref{thm:introgeneralizeWolf} is also closely related to the work of Simpson~\cite{Simpson90}, Biswas, Ad\'es-Gastesi, and Govindarajan~\cite{BAGG97}, Sabbah~\cite{Sabbah99}, and Biquard and Boalch~\cite{BiquardBoalch04}.

\subsection{Signed differentials and enhanced Teichm\"uller space}

To formulate our other results, we will need to define various objects naturally associated to a marked bordered surface. Firstly, if $(\mathbb{S}_1,\mathbb{M}_1)$ and $(\mathbb{S}_2,\mathbb{M}_2)$ are marked bordered surfaces, then an isomorphism $(\mathbb{S}_1,\mathbb{M}_1)\rightarrow(\mathbb{S}_2,\mathbb{M}_2)$ is defined to be an orientation preserving diffeomorphism $\mathbb{S}_1\rightarrow\mathbb{S}_2$ that induces a bijection $\mathbb{M}_1\cong\mathbb{M}_2$ of the marked points. Two isomorphisms are called isotopic if they are homotopic through isomorphisms, and the mapping class group $\MCG(\mathbb{S},\mathbb{M})$ is defined as the group of all isotopy classes of isomorphisms $(\mathbb{S},\mathbb{M})\rightarrow(\mathbb{S},\mathbb{M})$. This group acts on the set~$\mathbb{P}\subset\mathbb{M}$ of punctures, and we define the signed mapping class group to be the corresponding semidirect product $\MCG^\pm(\mathbb{S},\mathbb{M})=\MCG(\mathbb{S},\mathbb{M})\ltimes\mathbb{Z}_2^{\mathbb{P}}$.

If $\phi$ is a meromorphic quadratic differential with at least one pole on a compact Riemann surface~$S$, then the pair $(S,\phi)$ determines an associated marked bordered surface $(\mathbb{S},\mathbb{M})$. It is the surface $(S_P,M_P)$ considered above where $P$ is the tuple of principal parts of~$\phi$. This marked bordered surface is in fact independent of the local coordinates used to define the principal parts. For any marked bordered surface $(\mathbb{S},\mathbb{M})$, we can define a moduli space $\mathcal{Q}(\mathbb{S},\mathbb{M})$ parametrizing triples $(S,\phi,\theta)$ where $\phi$ is a meromorphic quadratic differential on a compact Riemann surface~$S$, and $\theta$ is an isomorphism from~$(\mathbb{S},\mathbb{M})$ to the marked bordered surface determined by~$(S,\phi)$. Following Bridgeland and Smith~\cite{BridgelandSmith15}, we require the differential~$\phi$ to satisfy some mild conditions making it into what is known as a Gaiotto-Moore-Neitzke (GMN) differential (see Definition~\ref{def:GMNdifferential} below). If we assume that $|\mathbb{M}|\geq3$ whenever $\mathbb{S}$ has genus zero, then $\mathcal{Q}(\mathbb{S},\mathbb{M})$ has the natural structure of a complex manifold.

Suppose $\phi$ is a quadratic differential on a Riemann surface~$S$ with a double pole at some point $p\in S$. If $z$ is a local coordinate defined in a neighborhood of~$p$ with $z(p)=0$, then the principal part of $\phi$ with respect to the local coordinate is equivalent to the residue of~$\phi$ at~$p$. The latter is defined as $\Res_p(\phi)=\pm4\pi\mathrm{i}\sqrt{a_p}$ where $a_p$ is the leading coefficient in the Laurent expansion of~$\phi$ in the variable~$z$. The residue is independent of the choice of local coordinate and is defined only up to a sign. We define a signing for~$\phi$ to be a choice of one of the two values for the residue at each pole of order two, and we define a signed differential to be a quadratic differential equipped with a signing. There is a branched cover 
\[
\mathcal{Q}^\pm(\mathbb{S},\mathbb{M})\rightarrow\mathcal{Q}(\mathbb{S},\mathbb{M})
\]
where the fiber over a point $(S,\phi,\theta)$ parametrizes the choices of signing for~$\phi$. This cover has degree $2^{|\mathbb{P}|}$ and is branched over the locus of points $(S,\phi,\theta)$ such that $\phi$ has at least one simple pole. There is a natural action of the group $\MCG^\pm(\mathbb{S},\mathbb{M})$ on this space $\mathcal{Q}^\pm(\mathbb{S},\mathbb{M})$ where the $\mathbb{Z}_2^{\mathbb{P}}$ factor acts by changing the signing.

In addition to the space $\mathcal{Q}(\mathbb{S},\mathbb{M})$, we consider the Teichm\"uller space $\mathcal{T}(\mathbb{S},\mathbb{M})$ defined as before. We will see that it has the structure of an orbifold whose real dimension is half that of $\mathcal{Q}(\mathbb{S},\mathbb{M})$. There is a branched cover 
\[
\mathcal{T}^\pm(\mathbb{S},\mathbb{M})\rightarrow\mathcal{T}(\mathbb{S},\mathbb{M})
\]
where the fiber over a marked hyperbolic surface $(C,\psi)$ corresponds to a choice of orientation for every boundary component of~$C$ that is homeomorphic to~$S^1$. This cover has degree $2^{|\mathbb{P}|}$ and is branched over the locus of points $(C,\psi)$ where the surface $C$ has at least one cusp. The space $\mathcal{T}^\pm(\mathbb{S},\mathbb{M})$ is a real manifold known in the literature as the enhanced Teichm\"uller space. There is a natural action of $\MCG^\pm(\mathbb{S},\mathbb{M})$ on the enhanced Teichm\"uller space where the $\mathbb{Z}_2^{\mathbb{P}}$ factor acts by changing the choice of orientations for boundary components. In Section~\ref{sec:AMapBetweenModuliSpaces}, we use Theorem~\ref{thm:introgeneralizeWolf} to prove the following result.

\begin{theorem}
\label{thm:introPsipm}
Let $(\mathbb{S},\mathbb{M})$ be a marked bordered surface, and if $\mathbb{S}$ has genus zero, assume $|\mathbb{M}|\geq3$. Then there is an $\MCG^\pm(\mathbb{S},\mathbb{M})$-equivariant continuous map 
\[
\Psi^\pm:\mathcal{Q}^\pm(\mathbb{S},\mathbb{M})\rightarrow\mathcal{T}^\pm(\mathbb{S},\mathbb{M})
\]
from the space of signed differentials to the enhanced Teichm\"uller space.
\end{theorem}

We conjecture that the map $\Psi^\pm$ of Theorem~\ref{thm:introPsipm} is actually real analytic. It should be possible to prove this by extending the methods from Section~5 of~\cite{Wolf91} using the analytic implicit function theorem. We leave this as a problem for future research.

\subsection{Periods and shear coordinates}

If $\phi$ is a GMN~differential on a compact Riemann surface~$S$, then $\phi$ determines a double cover $\pi:\Sigma_\phi\rightarrow S$ on which the square root $\sqrt{\phi}$ is a well defined meromorphic 1-form. This double cover is known as the spectral cover for~$\phi$. Let $\Sigma_\phi^\circ$ be the surface obtained from~$\Sigma_\phi$ by deleting the preimages of all poles of~$\phi$ of order~$\geq2$, and let $\tau:\Sigma_\phi^\circ\rightarrow\Sigma_\phi^\circ$ be the natural involution interchanging the two sheets of the double cover. Then the set 
\[
\widehat{H}(\phi)=\{\gamma\in H_1(\Sigma_\phi^\circ,\mathbb{Z}):\tau(\gamma)=-\gamma\}
\]
is a lattice called the hat homology of~$\phi$, and the map 
\[
Z_\phi:\widehat{H}(\phi)\rightarrow\mathbb{C}, \quad Z_\phi(\gamma)=\int_\gamma\sqrt{\phi}
\]
is a homomorphism called the period map. If we choose $\phi$ to be sufficiently generic (more precisely, if it is complete and saddle-free in the sense defined below), then $\phi$ determines an associated basis $\{\gamma_i\}$ of the hat homology lattice whose elements $\gamma_i$ are called standard saddle classes. In particular, we can evaluate the periods $Z_\phi(\gamma_i)\in\mathbb{C}$ of these basis elements. As shown in Section~4 of~\cite{BridgelandSmith15}, these periods can be used to define local coordinates on moduli spaces of quadratic differentials.

One can also construct interesting coordinates on the enhanced Teichm\"uller space. We first define an ideal triangulation of a marked bordered surface $(\mathbb{S},\mathbb{M})$ to be a triangulation of~$\mathbb{S}$ whose vertices are the points of~$\mathbb{M}$. We will also consider the slightly more refined concept of a tagged triangulation introduced by Fomin, Shapiro, and Thurston~\cite{FominShapiroThurston08}. To each arc~$\alpha$ of a tagged triangulation, we associate a function $X_\alpha:\mathcal{T}^\pm(\mathbb{S},\mathbb{M})\rightarrow\mathbb{R}_{>0}$ called a \emph{cluster coordinate} or \emph{Fock-Goncharov coordinate}. These functions are closely related to Thurston's shear coordinates and provide a global parametrization of $\mathcal{T}^\pm(\mathbb{S},\mathbb{M})$. As observed by Fock and Goncharov~\cite{FockGoncharov1,FockGoncharov2}, they transform by cluster Poisson transformations when we change the triangulation.

Now suppose we are given a quadratic differential $\phi$ in the space $\mathcal{Q}^\pm(\mathbb{S},\mathbb{M})$. If we choose~$\phi$ generically, then there is an associated tagged triangulation $\tau(\phi)$ of $(\mathbb{S},\mathbb{M})$ called the tagged WKB~triangulation. Its arcs are in bijection with the standard saddle classes, and so we can use the same symbol $\gamma$ to denote a standard saddle class and the corresponding arc of~$\tau(\phi)$. For each $R>0$, we will write $X_{\phi,\gamma}(R)$ for the cluster coordinate of $\Psi^\pm(R^2\cdot\phi)$ with respect to this arc~$\gamma$. Then in Section~\ref{sec:AsymptoticProperty}, we prove the following result.

\begin{theorem}
\label{thm:introclusterasymptotics}
Taking notation as above, we have 
\[
X_{\phi,\gamma}(R)\cdot\exp(R\cdot\Re Z_\phi(\gamma))\rightarrow1 \quad \text{as $R\rightarrow\infty$}.
\]
\end{theorem}

Theorem~\ref{thm:introclusterasymptotics} confirms a conjecture of Gaiotto, Moore, and Neitzke~\cite{GaiottoMooreNeitzke13}, formulated precisely in equation~(3.19) of~\cite{DumasNeitzke20}. It is an analog of Theorem~1.5 in~\cite{Allegretti19}. In physical terms, the construction in~\cite{Allegretti19} is the conformal limit of the construction presented here.

\subsection{Categorical interpretation}

In the final two sections of this paper, we interpret our results in categorical terms. As we will review in Section~\ref{sec:TriangulatedCategories}, one can associate to any 2-acyclic quiver with potential $(Q,W)$ a corresponding 3-Calabi-Yau triangulated category $\mathcal{D}(Q,W)$. Explicitly, it is defined as the full subcategory of the derived category of modules over the Ginzburg algebra of~$(Q,W)$ consisting of modules with finite-dimensional cohomology~\cite{KellerYang11}. The resulting category is equipped with a canonical bounded t-structure whose heart is a full abelian subcategory $\mathcal{A}(Q,W)\subset\mathcal{D}(Q,W)$ encoding the quiver~$Q$.

To see other abelian subcategories of~$\mathcal{D}=\mathcal{D}(Q,W)$, we use the operation of tilting. Given a finite length heart $\mathcal{A}\subset\mathcal{D}$ and a simple object $S\in\mathcal{A}$, there is a notion of tilting with respect to~$S$ to get a new heart~$\mathcal{A}'\subset\mathcal{D}$. We define $\Tilt(\mathcal{D})$ to be the graph whose vertices are in bijection with the finite length hearts in~$\mathcal{D}$, where two vertices are connected by an edge if the associated hearts are related by tilting at a simple object. The group $\Aut(\mathcal{D})$ of autoequivalences of~$\mathcal{D}$ acts naturally on $\Tilt(\mathcal{D})$. This graph has a distinguished connected component $\Tilt_\Delta(\mathcal{D})\subset\Tilt(\mathcal{D})$ containing the distinguished heart $\mathcal{A}(Q,W)$, and we will write $\Aut_\Delta(\mathcal{D})\subset\Aut(\mathcal{D})$ for the subgroup that preserves this distinguished component. We will write $\cAut_\Delta(\mathcal{D})$ for the quotient of $\Aut_\Delta(\mathcal{D})$ by the subgroup of autoequivalences that act trivially on $\Tilt_\Delta(\mathcal{D})$. The spherical twist functors introduced by Seidel and Thomas~\cite{SeidelThomas01} generate a subgroup which we denote $\cSph_\Delta(\mathcal{D})\subset\cAut_\Delta(\mathcal{D})$. The quotient 
\[
\Exch_\Delta(\mathcal{D})=\Tilt_\Delta(\mathcal{D})/\cSph_\Delta(\mathcal{D})
\]
is known in cluster theory as the exchange graph, while 
\[
\mathcal{G}_\Delta(\mathcal{D})=\cAut_\Delta(\mathcal{D})/\cSph_\Delta(\mathcal{D})
\]
is known as the cluster modular group.

The space of stability conditions on a triangulated category was introduced by Bridgeland in~\cite{Bridgeland07}. It is a complex manifold which carries a natural action of the group of autoequivalences. A point in the stability manifold for a category~$\mathcal{D}$ is specified by a heart $\mathcal{A}\subset\mathcal{D}$ together with a group homomorphism $Z:K(\mathcal{A})\rightarrow\mathbb{C}$ satisfying certain requirements. For the categories $\mathcal{D}=\mathcal{D}(Q,W)$ that we consider, the space of stability conditions contains a distinguished connected component which we denote $\Stab_\Delta(\mathcal{D})$. The group $\cSph_\Delta(\mathcal{D})$ acts on $\Stab_\Delta(\mathcal{D})$, and we consider the quotient 
\[
\Sigma(Q,W)=\Stab_\Delta(\mathcal{D})/\cSph_\Delta(\mathcal{D}).
\]
We define a second space $\mathcal{T}(Q)$ by associating a copy of the real manifold $\mathbb{R}_{>0}^n$ to each vertex $t\in\Exch_\Delta(\mathcal{D})$ and then identifying two manifolds by a cluster transformation whenever the associated vertices are connected by an edge. Here $n$ is the rank of the Grothendieck group $K(\mathcal{D})$. The space $\mathcal{T}(Q)$ obtained in this way is the set of $\mathbb{R}_{>0}$-valued points of a cluster variety associated to the quiver~$Q$. We refer to it as the enhanced Teichm\"uller space of~$Q$. For every $t\in\Exch_\Delta(\mathcal{D})$ and $\gamma\in K(\mathcal{D})$, one has a function $X_\gamma:\mathcal{T}(Q)\rightarrow\mathbb{R}_{>0}$ defined using cluster coordinates.

We will apply the above constructions in a class of examples arising from triangulated surfaces. If $(\mathbb{S},\mathbb{M})$ is a marked bordered surface satisfying mild conditions and $\tau$ is a tagged triangulation of~$(\mathbb{S},\mathbb{M})$, then the work of Labardini-Fragoso~\cite{LabardiniFragoso08,LabardiniFragoso16} provides an associated quiver with potential $(Q,W)=(Q(\tau),W(\tau))$, well defined up to a suitable notion of equivalence. We can then consider the associated 3-Calabi-Yau triangulated category $\mathcal{D}=\mathcal{D}(Q,W)$, which was studied in the work of Bridgeland and Smith~\cite{BridgelandSmith15}. Assuming $(\mathbb{S},\mathbb{M})$ satisfies an amenability condition from~\cite{BridgelandSmith15}, one has isomorphisms $\Sigma(Q,W)\cong\mathcal{Q}^\pm(\mathbb{S},\mathbb{M})$ and $\mathcal{T}(Q)\cong\mathcal{T}^\pm(\mathbb{S},\mathbb{M})$. By combining these facts with Theorem~\ref{thm:introPsipm}, we deduce the following result in Section~\ref{sec:TheMainResults}.

\begin{theorem}
\label{thm:intromain1}
Let $(Q,W)$ be the quiver with potential associated to a tagged triangulation of an amenable marked bordered surface, and let $\mathcal{D}=\mathcal{D}(Q,W)$ be the associated 3-Calabi-Yau triangulated category. Then there is a $\mathcal{G}_\Delta(\mathcal{D})$-equivariant continuous map 
\[
\widehat{\Psi}:\Sigma(Q,W)\rightarrow\mathcal{T}(Q)
\]
from the space of stability conditions to the enhanced Teichm\"uller space.
\end{theorem}

Finally, for a generic point $\sigma=(\mathcal{A},Z)$ of~$\Stab_\Delta(\mathcal{D})$, we consider the 1-parameter family of stability conditions $\sigma_R=(\mathcal{A},R\cdot Z)$ for $R>0$. The heart $\mathcal{A}$ determines a vertex $t\in\Exch_\Delta(\mathcal{D})$, and hence for any $\gamma\in K(\mathcal{D})$, we have a function $X_\gamma:\mathcal{T}(Q)\rightarrow\mathbb{R}_{>0}$. We will write $X_{\sigma,\gamma}(R)=X_\gamma(\widehat{\Psi}(\sigma_R))$. Using Theorem~\ref{thm:introclusterasymptotics}, we prove the following in Section~\ref{sec:TheMainResults}.

\begin{theorem}
\label{thm:intromain2}
Take notation as in the last paragraph. Then 
\[
X_{\sigma,\gamma}(R)\cdot\exp(R\cdot\Re Z(\gamma))\rightarrow1 \quad \text{as $R\rightarrow\infty$}.
\]
\end{theorem}

The statements in Theorems~\ref{thm:intromain1} and~\ref{thm:intromain2} make sense for the 3-Calabi-Yau triangulated category constructed from any nondegenerate quiver with potential, even when the quiver with potential does not arise from a triangulated surface. We conjecture that Theorems~\ref{thm:intromain1} and~\ref{thm:intromain2} are special cases of more general results valid for such categories.

\subsection*{Acknowledgements.}
I am grateful to Subhojoy~Gupta, Andrew~Neitzke, and Michael~Wolf for answering questions related to the content of this paper.

\section{Meromorphic quadratic differentials}
\label{sec:MeromorphicQuadraticDifferentials}

In this section, we define various moduli spaces parametrizing meromorphic quadratic differentials on Riemann surfaces.

\subsection{Basic definitions}

Let $S$ be a Riemann surface. Then a meromorphic \emph{quadratic differential} on~$S$ is defined to be a meromorphic section of~$\omega_S^{\otimes2}$ where $\omega_S$ is the holomorphic cotangent bundle of~$S$. In terms of a local coordinate $z$ on~$S$, a quadratic differential $\phi$ can be expressed as 
\[
\phi(z)=\varphi(z)dz^{\otimes2}
\]
where $\varphi(z)$ is a meromorphic function in the local coordinate.

By a \emph{critical point} of a quadratic differential~$\phi$, we mean either a zero or a pole of~$\phi$. We will denote by $\Crit(\phi)$ the set of all critical points of~$\phi$. It is useful to classify the critical points into two types. The first type is a \emph{finite critical point}, which is defined as a zero or simple pole. The second type is an \emph{infinite critical point}, defined as a pole of order $\geq2$. We will denote by $\Crit_{<\infty}(\phi)$ and $\Crit_\infty(\phi)$ the sets of finite and infinite critical points, respectively. In this paper, we sometimes restrict attention to quadratic differentials satisfying additional conditions:

\begin{definition}[\cite{BridgelandSmith15}, Definition~2.1]
\label{def:GMNdifferential}
A quadratic differential $\phi$ defined on a compact, connected Riemann surface is called a \emph{Gaiotto-Moore-Neitzke (GMN) differential} if 
\begin{enumerate}
\item $\phi$ has no zero of order~$>1$.
\item $\phi$ has at least one pole.
\item $\phi$ has at least one finite critical point.
\end{enumerate}
A GMN differential is said to be \emph{complete} if it has no simple poles.
\end{definition}

Let $\phi$ be a quadratic differential on~$S$. In a neighborhood of any point of~$S\setminus\Crit(\phi)$, there is a local coordinate~$w$, well defined up to transformations of the form $w\mapsto\pm w+\text{constant}$, with respect to which the quadratic differential can be written $\phi(w)=dw^{\otimes2}$. Indeed, if we have $\phi(z)=\varphi(z)dz^{\otimes2}$ for some local coordinate~$z$, then we can take $w=\int\sqrt{\varphi(z)}dz$. These distinguished local coordinates determine two important geometric structures on the surface. On the one hand, by pulling back the standard Euclidean metric on~$\mathbb{C}$ using the distinguished local coordinates, we get a flat metric defined on $S\setminus\Crit(\phi)$. On the other hand, by pulling back the horizontal lines $\Im(w)=\text{constant}$ using the distinguished local coordinates, we get a foliation of $S\setminus\Crit(\phi)$ called the \emph{horizontal~foliation}.

\subsection{Marked bordered surfaces}
\label{sec:MarkedBorderedSurfaces}

In the following, we would like to consider differentials on surfaces of a fixed topological type with varying complex structure. With this in mind, we define a \emph{marked bordered surface} to be a pair~$(\mathbb{S},\mathbb{M})$ where $\mathbb{S}$ is a compact connected oriented smooth surface with (possibly empty) boundary, and $\mathbb{M}\subset\mathbb{S}$ is a nonempty finite set of marked points such that each component of~$\partial\mathbb{S}$ contains at least one marked point. We will refer to marked points in the interior of~$\mathbb{S}$ as \emph{punctures} and denote the set of all punctures by~$\mathbb{P}\subset\mathbb{M}$.

Given a marked bordered surface $(\mathbb{S},\mathbb{M})$, it is sometimes convenient to consider the surface $\mathbb{S}'$ obtained by taking the oriented real blowup of~$\mathbb{S}$ at each point of~$\mathbb{P}$. The resulting surface has an additional boundary component with no marked points corresponding to each puncture of~$\mathbb{S}$. If we write $k_i$ for the number of marked points on the $i$th boundary component of this modified surface~$\mathbb{S}'$ and write $g=g(\mathbb{S})$ for the genus of~$\mathbb{S}$, then we can define 
\begin{equation}
\label{eqn:dimension}
n=6-6g+\sum_i(k_i+3).
\end{equation}
As we will see below, this number appears as the dimension of various natural moduli spaces associated to the marked bordered surface $(\mathbb{S},\mathbb{M})$.

If $(\mathbb{S}_1,\mathbb{M}_1)$ and $(\mathbb{S}_2,\mathbb{M}_2)$ are marked bordered surfaces, then an \emph{isomorphism} $(\mathbb{S}_1,\mathbb{M}_1)\rightarrow(\mathbb{S}_2,\mathbb{M}_2)$ is defined to be an orientation preserving diffeomorphism $\mathbb{S}_1\rightarrow\mathbb{S}_2$ which induces a bijection $\mathbb{M}_1\cong\mathbb{M}_2$. Two such isomorphisms are said to be \emph{isotopic} if they are homotopic through maps $\mathbb{S}_1\rightarrow\mathbb{S}_2$ which also induce bijections $\mathbb{M}_1\cong\mathbb{M}_2$. The group of all isotopy classes of isomorphisms from a marked bordered surface $(\mathbb{S},\mathbb{M})$ to itself is called the \emph{mapping class group} of~$(\mathbb{S},\mathbb{M})$ and denoted $\MCG(\mathbb{S},\mathbb{M})$.

\subsection{Marked quadratic differentials}

Let $\phi$ be a quadratic differential on~$S$ and let $p\in S$ be a pole of~$\phi$ of order $m\geq3$. If $z$ is a local coordinate defined in a neighborhood of~$p$ such that $z(p)=0$, then we can write 
\[
\phi(z)=\left(a_0z^{-m}+a_1z^{-m+1}+\dots\right)dz^{\otimes2}.
\]
Then the \emph{asymptotic horizontal directions} of~$\phi$ at~$p$ are the $m-2$ tangent vectors at~$p$ which are tangent to the rays defined by the condition $a_0\cdot z^{2-m}\in\mathbb{R}_{>0}$. The name comes from the fact that there is a neighborhood $p\in U\subset S$ such that any leaf of the horizontal foliation that enters~$U$ eventually tends to~$p$ and is asymptotic to one of the asymptotic horizontal directions. In particular, it follows that these directions are independent of the choice of local coordinate.

If we are given a quadratic differential $\phi$ on a compact Riemann surface $S$ with at least one pole, then we can construct an associated marked bordered surface $(\mathbb{S},\mathbb{M})$ by the following procedure. To define the underlying smooth surface~$\mathbb{S}$, we take an oriented real blowup of~$S$ at each pole of~$\phi$ of order $\geq3$. The asymptotic horizontal directions of~$\phi$ determine a collection of points on the boundary of the resulting smooth surface, and we define the set $\mathbb{M}$ to consist of these points together with the poles of~$\phi$ of order~$\leq2$, considered as punctures.

Now suppose we fix a marked bordered surface $(\mathbb{S},\mathbb{M})$. If $\phi$ is a quadratic differential on~$S$, then a \emph{marking} of~$(S,\phi)$ by~$(\mathbb{S},\mathbb{M})$ is defined as an isotopy class of isomorphisms from $(\mathbb{S},\mathbb{M})$ to the marked bordered surface determined by~$\phi$.  A \emph{marked quadratic differential} $(S,\phi,\theta)$ is defined as a quadratic differential~$\phi$ on $S$ together with a marking~$\theta$ of~$(S,\phi)$ by~$(\mathbb{S},\mathbb{M})$. Two marked quadratic differentials $(S_1,\phi_1,\theta_1)$ and $(S_2,\phi_2,\theta_2)$ are considered to be equivalent if there exists a biholomorphism $f:S_1\rightarrow S_2$ that preserves the differentials $\phi_i$ and commutes with the markings~$\theta_i$ in the obvious way.

\subsection{Moduli spaces}
\label{sec:ModuliSpaces}

For any marked bordered surface $(\mathbb{S},\mathbb{M})$, we will write $\mathcal{Q}(\mathbb{S},\mathbb{M})$ for the set of equivalence classes of GMN~differentials together with a marking by~$(\mathbb{S},\mathbb{M})$. It is important to note that if $(S,\phi,\theta)\in\mathcal{Q}(\mathbb{S},\mathbb{M})$ then a puncture of~$(\mathbb{S},\mathbb{M})$ may correspond via~$\theta$ to a pole of~$\phi$ of order one or two, and therefore $(\mathbb{S},\mathbb{M})$ does not uniquely determine the orders of the poles of~$\phi$. The mapping class group $\MCG(\mathbb{S},\mathbb{M})$ acts naturally on~$\mathcal{Q}(\mathbb{S},\mathbb{M})$ by changing the marking of a differential in this set. We will now describe a natural complex manifold structure on this set $\mathcal{Q}(\mathbb{S},\mathbb{M})$ such that the mapping class group acts by biholomorphisms.

We begin by defining our notation. Let us write $g=g(\mathbb{S})$ for the genus of~$\mathbb{S}$ and $d$ for the number of boundary components of the surface~$\mathbb{S}'$ defined in Section~\ref{sec:MarkedBorderedSurfaces}. We will also write $k_1,\dots,k_d$ for the integers encoding the number of marked points on the boundary components of this surface~$\mathbb{S}'$ and set $m_i=k_i+2$. Note that these integers $m_i$ can be viewed as the pole orders for a quadratic differential with associated marked bordered surface $(\mathbb{S},\mathbb{M})$.

If we write $B=\mathcal{T}(g,d)$ for the Teichm\"uller space parametrizing marked Riemann surfaces of genus~$g$ with $d$ marked points, then there is a universal curve $\pi:X\rightarrow B$ whose fiber over $b\in B$ is the corresponding Riemann surface~$X(b)$, and this universal curve comes with $d$ disjoint sections $p_1,\dots,p_d$. We form the effective divisor 
\[
D=\sum_im_i\cdot D_i, \quad D_i=p_i(B)\subset X
\]
and write $D(b)=\sum_im_i\cdot p_i(b)$ for its restriction to~$X(b)$. Then there is a vector bundle $q:E\rightarrow B$ whose fiber over $b\in B$ is the vector space 
\[
E_b=H^0(X(b),\omega_{X(b)}^{\otimes2}(D(b))).
\]
There is a bundle over $E$ with discrete fibers parametrizing the possible markings of a differential by~$(\mathbb{S},\mathbb{M})$, and $\mathcal{Q}(\mathbb{S},\mathbb{M})$ is identified with the subset of the total space of this bundle consisting of GMN~differentials where the marking agrees, after blowing down all boundary components of~$\mathbb{S}$, with the marking of the corresponding point in the Teichm\"uller space~$B$. In this way, the set $\mathcal{Q}(\mathbb{S},\mathbb{M})$ acquires a topology.

\begin{proposition}[\cite{Allegretti21}, Proposition~6.2]
\label{prop:differentialsmanifold}
Let $(\mathbb{S},\mathbb{M})$ be a marked bordered surface, and if $g(\mathbb{S})=0$, assume that $|\mathbb{M}|\geq3$. Then the space $\mathcal{Q}(\mathbb{S},\mathbb{M})$ has the structure of a complex manifold of dimension~$n$ given by~\eqref{eqn:dimension}.
\end{proposition}

\subsection{Principal parts}

Let $\phi$ be a quadratic differential on~$S$ and suppose $p\in S$ is a pole of~$\phi$ of order~$m\geq3$. We will also fix a local coordinate $z$ defined in an open neighborhood $p\in U\subset S$ such that $z(p)=0$. If $m$ is even, then we can write down a square root of~$\phi$ given at any point of~$U$ by an expression of the form 
\[
\frac{1}{z^{m/2}}\left(f(z)+z^{\frac{m}{2}}g(z)\right)dz
\]
where $f(z)$ is a polynomial of degree $\leq\frac{m-2}{2}$ and $g(z)$ is a holomorphic function. On the other hand, if $m$ is odd, then after choosing a square root $z^{1/2}$ of~$z$, we can write down a square root of~$\phi$ given at a point of~$U$ by an expression 
\[
\frac{1}{z^{m/2}}\left(f(z)+z^{\frac{m-2}{2}}g(z)\right)dz
\]
where $f(z)$ is a polynomial of degree $\leq\frac{m-3}{2}$ and $g(z)$ is a holomorphic function. In either case, the expression $\pm\frac{1}{z^{m/2}}f(z)dz$ will be called the \emph{principal part} of~$\phi$ at~$p$. Note that this is well defined only up to a sign, and in general it depends on the choice of local coordinate~$z$.

In the case where the pole $p$ has order $m\leq2$, the principal part of~$\phi$ at~$p$ is simply defined as $\pm\sqrt{a_p}\,dz/z$ where $a_p$ is the coefficient of $z^{-2}$ in the Laurent expansion of~$\phi$ near~$p$. In this case, the coefficient $a_p$ is called the \emph{leading coefficient} of~$\phi$ at~$p$ and can be shown to be independent of the choice of local coordinate. The associated quantity 
\[
\Res_p(\phi)=\pm4\pi\mathrm{i}\sqrt{a_p}
\]
is known as the \emph{residue} of~$\phi$ at~$p$. Like the principal part, it is well defined only up to a sign.

\subsection{Prescribing the principal parts}
\label{sec:PrescribingThePrincipalParts}

Let $z$ be a local coordinate on a compact Riemann surface~$S$. In view of the above discussion, it is natural to define a \emph{principal differential} in~$z$ to be an expression of the form 
\[
z^{-\varepsilon}\left(c_rz^{-r}+c_{r-1}z^{r-1}+\dots+c_1z^{-1}\right)dz
\]
where $\varepsilon\in\{0,\frac{1}{2}\}$ and the $c_i$ are complex numbers. Suppose we are given a finite subset $M=\{p_1,\dots,p_d\}\subset S$ and a choice of local coordinate $z_i$ in a neighborhood of each~$p_i$ so that $z_i(p_i)=0$. If $P=(P_i(z_i))$ is a tuple of principal differentials in these local coordinates, then we get a divisor $D=\sum_im_ip_i$ where $m_i$ is the order of the pole of the quadratic differential~$\phi_i(z_i)=P_i(z_i)^{\otimes2}$ at the point $z_i=0$. We define 
\[
\mathcal{Q}_P(S,M)\subset H^0(S,\omega_S^{\otimes2}(D))
\]
to be the set of quadratic differentials having principal part $P_i(z_i)$ at the point~$p_i$ for each index~$i$. Note that if $m_i\leq2$, then this condition is equivalent to fixing the leading coefficient at~$p_i$.

\begin{lemma}
\label{lem:dimfixedprincipalpart}
Take notation as in the last paragraph. Suppose 
\[
q=6g-6+\sum_is_i\geq 0
\]
where $g=g(S)$ and we write $s_i=m_i+1$ if $m_i$ is odd and $s_i=m_i$ if $m_i$ even. Then there is a homeomorphism $\mathcal{Q}_P(S,M)\cong\mathbb{R}^q$.
\end{lemma}

\begin{proof}
If $\phi\in\mathcal{Q}_P(S,M)$ then in a neighborhood of any~$p_i\in M$ we can write $\phi(z_i)=\varphi(z_i)dz_i^{\otimes2}$ where $\varphi(z_i)=\frac{1}{z_i^{m_i}}\cdot\sum_{j\geq0}a_jz_i^j$. One can check that the choice of the principal part $P_i(z_i)$ is equivalent to the choice of the coefficients $a_0,\dots,a_{\lfloor m_i/2\rfloor-1}$. In particular, if $\phi_1$,~$\phi_2\in\mathcal{Q}_P(S,M)$, then $\phi_1-\phi_2$ has a pole of order at most $m_i-\lfloor m_i/2\rfloor=s_i/2$ at~$p_i$. Hence $\mathcal{Q}_P(S,M)$ is an affine space for the complex vector space $H^0(S,\omega_S(E)^{\otimes2})$ where we have introduced the divisor $E=\sum_i(s_i/2)\cdot p_i$. Under our assumption, the Riemann-Roch theorem implies that this vector space has complex dimension $3g-3+\sum_is_i/2$. Hence $\mathcal{Q}_P(S,M)$ is homeomorphic to a Euclidean space of the required dimension.
\end{proof}

\subsection{Signed differentials}
\label{sec:SignedDifferentials}

We will now define a modification of the moduli space $\mathcal{Q}(\mathbb{S},\mathbb{M})$ that parametrizes marked quadratic differentials with additional data associated to the poles of order two. If $\phi$ is a quadratic differential on~$S$ and $p\in S$ is a pole of~$\phi$ of order two, we have seen that the residue $\Res_p(\phi)=\pm4\pi\mathrm{i}\sqrt{a_p}$ is well defined up to a sign. By a \emph{signing} for~$\phi$, we mean the choice of one of these two values for each pole of order two. A \emph{signed differential} is a quadratic differential together with a signing.

There is a branched cover 
\[
\mathcal{Q}^\pm(\mathbb{S},\mathbb{M})\rightarrow\mathcal{Q}(\mathbb{S},\mathbb{M})
\]
where a point in the fiber over the class of $(S,\phi,\theta)$ is obtained by choosing a signing for~$\phi$. This definition makes sense because the leading coefficient is invariant under pullback. The cover has degree $2^{|\mathbb{P}|}$ and is branched precisely over the locus of points such that $\phi$ has at least one simple pole. The mapping class group $\MCG(\mathbb{S},\mathbb{M})$ acts on the set $\mathbb{P}\subset\mathbb{S}$, and we define the \emph{signed mapping class group} to be the corresponding semidirect product 
\[
\MCG^\pm(\mathbb{S},\mathbb{M})=\MCG(\mathbb{S},\mathbb{M})\ltimes\mathbb{Z}_2^{\mathbb{P}}.
\]
There is a natural action of this group $\MCG^\pm(\mathbb{S},\mathbb{M})$ on $\mathcal{Q}^\pm(\mathbb{S},\mathbb{M})$ where the $\mathbb{Z}_2^{\mathbb{P}}$ factor acts by changing the signing of a point in this space.

\begin{proposition}[\cite{Allegretti21}, Proposition~6.3]
Let $(\mathbb{S},\mathbb{M})$ be a marked bordered surface, and if $g(\mathbb{S})=0$ assume that $|\mathbb{M}|\geq3$. Then $\mathcal{Q}^\pm(\mathbb{S},\mathbb{M})$ has the structure of a complex manifold of dimension~$n$ given by~\eqref{eqn:dimension}.
\end{proposition}

\subsection{Horizontal strip decomposition}
\label{sec:HorizontalStripDecomposition}

We now review some geometric concepts related to quadratic differentials, following the treatment in~\cite{BridgelandSmith15}. Let $\phi$ be a quadratic differential on~$S$. By a \emph{straight arc}, we will mean a path $\alpha:I\rightarrow S\setminus\Crit(\phi)$, defined on some open interval $I\subset\mathbb{R}$, which makes a constant angle in the flat metric with the leaves of the horizontal foliation. We require that the tangent vectors to a straight arc have unit norm in the flat metric, and we consider two straight arcs to be equivalent if they are related by a reparametrization of the form $t\mapsto\pm t+\text{constant}$. We then define a \emph{trajectory} to be a straight arc which is not the restriction of a straight arc defined on a larger interval. There are three kinds of trajectories that will play a role in our discussion:
\begin{enumerate}
\item A \emph{saddle trajectory} is a trajectory which is asymptotic to a finite critical point at each of its ends.
\item A \emph{separating trajectory} is a trajectory which is asymptotic to a finite critical point at one end and an infinite critical point at the other end.
\item A \emph{generic trajectory} is a trajectory which is asymptotic to an infinite critical point at each of its ends.
\end{enumerate}
In the case of a saddle trajectory or generic trajectory, the two ends may be distinct points of~$S$, or they may coincide. We say that a quadratic differential is \emph{saddle-free} if there is no saddle trajectory which is also a leaf of the horizontal foliation.

\begin{lemma}[\cite{BridgelandSmith15}, Lemma~3.1]
\label{lem:completesaddlefree}
If $\phi$ is a complete, saddle-free GMN~differential, then every leaf of the horizontal foliation is either a generic trajectory or one of finitely many separating trajectories.
\end{lemma}

We will be interested in two kinds of regions on a surface determined by the horizontal foliation of a quadratic differential:
\begin{enumerate} 
\item A \emph{horizontal strip} is defined as a maximal region in~$S$ which is mapped by a distinguished local coordinate to a region of the form 
\[
\{w\in\mathbb{C}:a<\Im(w)<b\}\subset\mathbb{C}.
\]
Any leaf of the foliation contained in such a region is a generic trajectory connecting two (not necessarily distinct) poles. Each boundary component of a horizontal strip is composed of saddle trajectories and separating trajectories.
\item A \emph{half plane} is defined as a maximal region in~$S$ which is mapped by a distinguished local coordinate to a region of the form 
\[
\{w\in\mathbb{C}:\Im(w)>0\}\subset\mathbb{C}.
\]
Any leaf of the foliation contained in such a region is a generic trajectory connecting a fixed pole of order $>2$ to itself. The boundary of a half plane is composed of saddle trajectories and separating trajectories.
\end{enumerate}

In the situation of Lemma~\ref{lem:completesaddlefree}, if we delete the critical points and the separating trajectories from~$S$, then we are left with an open subsurface, which is a disjoint union of horizontal strips and half planes. Since the differential is assumed to be saddle-free, each of the lines $\Im(w)=a$ and $\Im(w)=b$ corresponding to the boundary of a horizontal strip contains a unique point corresponding to a finite critical point of~$\phi$. We get a distinguished saddle trajectory on~$S$ by connecting these two points by a straight line segment in~$\mathbb{C}$ (see Figure~\ref{fig:standardsaddle}) and then embedding the horizontal strip into~$S$. This distinguished saddle trajectory is called the \emph{standard saddle connection} associated to the horizontal strip.

\begin{figure}[ht]
\begin{center}
\begin{tikzpicture}
\draw[black, ultra thin] (-3,0) -- (3,0);
\draw[black, ultra thin] (-3,0.25) -- (3,0.25);
\draw[black, ultra thin] (-3,0.5) -- (3,0.5);
\draw[black, thin] (-3,0.75) -- (3,0.75);
\draw[black, ultra thin] (-3,1) -- (3,1);
\draw[black, thin] (-3,1.25) -- (3,1.25);
\draw[black, ultra thin] (-3,1.5) -- (3,1.5);
\draw[black, thin] (-1,0) -- (1,1.5);
\node at (-1,0) {{\tiny $\times$}};
\node at (1,1.5) {{\tiny $\times$}};
\node at (-3.5,0.75) {{\tiny $\dots$}};
\node at (3.5,0.75) {{\tiny $\dots$}};
\end{tikzpicture}
\end{center}
\caption{Image of a standard saddle connection.\label{fig:standardsaddle}}
\end{figure}
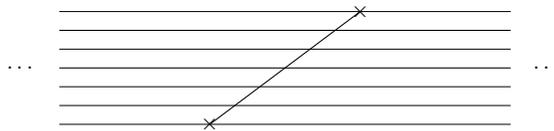

\subsection{The period map}

We conclude this section by describing a construction that can be used to locally parametrize quadratic differentials. Let $\phi$ be a GMN differential on a compact Riemann surface~$S$ with poles of order~$m_i$ at points $p_i\in S$. We can alternatively view this differential $\phi$ as a section 
\[
s\in H^0(S,\omega_S(E)^{\otimes2}), \quad E=\sum_i\left\lceil\frac{m_i}{2}\right\rceil p_i
\]
with simple zeros at the zeros and odd order poles of~$\phi$. In terms of this section, we can define an auxiliary surface, the \emph{spectral cover}, by 
\[
\Sigma_\phi=\{(p,\lambda(p)):p\in S, \lambda(p)\in F_p, \lambda(p)\otimes\lambda(p)=s(p)\}\subset F
\]
where $F$ denotes the total space of the line bundle $\omega_S(E)$. There is a natural projection $\pi:\Sigma_\phi\rightarrow S$ which is a double cover of~$S$ branched precisely at the simple zeros and odd order poles of~$\phi$. We will write $\Sigma_\phi^\circ=\pi^{-1}(S\setminus\Crit_\infty(\phi))$ and write $\tau:\Sigma_\phi^\circ\rightarrow\Sigma_\phi^\circ$ for the covering involution that exchanges the two sheets of the spectral cover.

As explained in Section~2.3 of~\cite{BridgelandSmith15}, the spectral cover gives rise to a finitely generated free abelian group, the \emph{hat homology} $\widehat{H}(\phi)$, which is given by 
\[
\widehat{H}(\phi)=\{\gamma\in H_1(\Sigma_\phi^\circ,\mathbb{Z}):\tau(\gamma)=-\gamma\}.
\]
On the other hand, there is a globally defined meromorphic 1-form $\lambda$ on the spectral cover with the property that $\pi^*(\phi)=\lambda\otimes\lambda$. It is anti-invariant under the action of the covering involution~$\tau$, and we can define the \emph{period map} to be the group homomorphism 
\[
Z_\phi:\widehat{H}(\phi)\rightarrow\mathbb{C}, \quad \gamma\mapsto\int_\gamma\lambda
\]
taking values in the complex numbers.

If $\phi$ is a complete, saddle-free differential, then the hat homology~$\widehat{H}(\phi)$ has a natural basis. To describe this basis, we first note that each standard saddle connection~$\alpha$ can be lifted to a cycle in~$\Sigma_\phi^\circ$. Indeed, since finite critical points are branch points for the covering map $\pi:\Sigma_\phi\rightarrow S$, the preimage of a standard saddle connection under this covering map is a closed loop in~$\Sigma_\phi^\circ$. By choosing an orientation for this loop, we get a homology class~$\gamma_\alpha$ which satisfies $\tau(\gamma_\alpha)=-\gamma_\alpha$, and we can fix this orientation uniquely by requiring that by requiring that the period $Z_\phi(\gamma_\alpha)$ lie in the upper half plane. Thus we associate, to each standard saddle connection~$\alpha$ for~$\phi$, a class $\gamma_\alpha\in\widehat{H}(\phi)$ called a \emph{standard saddle class}. By Lemma~3.2 of~\cite{BridgelandSmith15}, these classes form a basis for~$\widehat{H}(\phi)$. One can think of the periods $Z_\phi(\widehat{\alpha})$ as coordinates on the space of GMN~differentials inducing a given horizontal strip decomposition. See Section~4 of~\cite{BridgelandSmith15} for more precise statements.

\section{Enhanced Teichm\"uller space}
\label{sec:EnhancedTeichmullerSpace}

In this section, we define the enhanced Teichm\"uller space of a punctured surface and describe its parametrization by shear coordinates.

\subsection{Marked hyperbolic surfaces}

Below we will be interested in spaces parametrizing hyperbolic structures on surfaces. The surfaces we consider will be of the form 
\[
C=\bar{C}\setminus\left(\bigcup_{i=1}^sD_i\cup\bigcup_{j=1}^t\{p_j\}\right)
\]
where $\bar{C}$ is a closed oriented surface, $D_1,\dots,D_s\subset\bar{C}$ are open disks whose closures are disjoint, and $p_1,\dots,p_t\in\bar{C}$ are points disjoint from the interiors of the~$D_i$. By a \emph{hyperbolic surface}, we will mean a surface $C$ of this type equipped with a complete, finite area hyperbolic metric with totally geodesic boundary. Note that by completeness of the metric, there will be a cusp neighborhood around each point~$p_i$.

Let $C^\circ$ denote the surface obtained from~$C$ by removing all components of~$\partial C$ homeomorphic to~$S^1$. Then we define a \emph{marking} of $C$ by a marked bordered surface~$(\mathbb{S},\mathbb{M})$ to be a homeomorphism $\psi:\mathbb{S}\setminus\mathbb{M}\rightarrow C^\circ$. We define a \emph{marked hyperbolic surface} to be a hyperbolic surface together with a marking. Two marked hyperbolic surfaces $(C_1,\psi_1)$ and $(C_2,\psi_2)$ are considered to be equivalent if there exists an isometry $g:C_1\rightarrow C_2$ such that the maps $\psi_2$ and $g\circ\psi_1$ are isotopic. In the following, we will write $\mathcal{T}(\mathbb{S},\mathbb{M})$ for the set of all equivalence classes of hyperbolic surfaces with a marking by~$(\mathbb{S},\mathbb{M})$. The mapping class group $\MCG(\mathbb{S},\mathbb{M})$ acts naturally on this set by changing the marking of a hyperbolic surface.

A point $(C,\psi)\in\mathcal{T}(\mathbb{S},\mathbb{M})$ determines a corresponding representation 
\[
\rho:\Gamma\rightarrow G
\]
of $\Gamma=\pi_1(\mathbb{S}\setminus\mathbb{M})$ into $G=\PSL_2(\mathbb{R})$, unique up to the action of~$G$ by conjugation. Indeed, the universal cover of~$C$ can be identified isometrically with a subset of the hyperbolic plane~$\mathbb{H}$ with totally geodesic boundary. The marking $\psi$ induces an isomorphism $\Gamma\cong\pi_1(C)$, and the fundamental group $\pi_1(C)$ acts on the universal cover by orientation preserving isometries of~$\mathbb{H}$. Since the group of all orientation preserving isometries of~$\mathbb{H}$ is identified with~$G$, we get a representation $\rho:\Gamma\rightarrow G$. Note that this construction depends on a choice of embedding of the universal cover into~$\mathbb{H}$. If we make a different choice, the resulting representation changes by conjugation by an element of~$G$. The conjugacy class of~$\rho$ is called the \emph{monodromy} of~$(C,\psi)$.

\subsection{Topological structure}

We will use the monodromy to define a natural topology on the space $\mathcal{T}(\mathbb{S},\mathbb{M})$ for any marked bordered surface $(\mathbb{S},\mathbb{M})$. We begin by considering the special case where $\mathbb{S}$ is a closed surface and $\mathbb{M}=\mathbb{P}$ consists entirely of punctures.

\begin{lemma}
\label{lem:embedding}
Let $\mathbb{S}$ be a closed surface and $\mathbb{P}\subset\mathbb{S}$ a nonempty finite set. Then the map 
\[
\mathcal{T}(\mathbb{S},\mathbb{P})\rightarrow\Hom(\Gamma,G)/G
\]
sending a marked hyperbolic surface to its monodromy is injective.
\end{lemma}

\begin{proof}
Suppose $(C_1,\psi_1)$ and $(C_2,\psi_2)$ are marked hyperbolic surfaces with the same monodromy. Let us fix embeddings of the universal covers of the $C_i$ into the hyperbolic plane. Then we get corresponding representations $\rho_i:\Gamma\rightarrow G$ such that $\rho_2=g\cdot\rho_1\cdot g^{-1}$ for some element $g\in G$. We can think of this element an isometry $g:\mathbb{H}\rightarrow\mathbb{H}$. The quotient $C_i'=\mathbb{H}/\rho_i(\Gamma)$ is a surface equipped with a hyperbolic metric, and $g$ descends to an isometry $g:C_1'\rightarrow C_2'$. The surface $C_i$ is identified with the \emph{convex core} of~$C_i'$, defined as the smallest closed convex subset of~$C_i'$ such that the inclusion into $C_i'$ is a homotopy equivalence. Thus we get an isometry $g:C_1\rightarrow C_2$.

By construction, the maps $\psi_2$ and $g\circ\psi_1$ induce the same isomorphism $\pi_1(\mathbb{S}\setminus\mathbb{P})\cong\pi_1(C_2^\circ)$. Since the surfaces $\mathbb{S}\setminus\mathbb{P}$ and $C_2^\circ$ are $K(\Gamma,1)$-spaces, there is a unique homotopy class of homotopy equivalences $\mathbb{S}\setminus\mathbb{P}\rightarrow C_2^\circ$ inducing this isomorphism. In particular, we see that the maps $\psi_2$ and $g\circ\psi_1$ are homotopic. Hence $(C_1,\psi_1)$ and $(C_2,\psi_2)$ are equivalent marked hyperbolic surfaces, and the map in the statement of the lemma is injective.
\end{proof}

Using Lemma~\ref{lem:embedding}, we can equip the set $\mathcal{T}(\mathbb{S},\mathbb{P})$ with a natural topology. Indeed, we can equip~$\Gamma$ with the discrete topology and $G$ with its standard topology as a Lie group. Then the set $\Hom(\Gamma,G)$ has the compact open topology, and $\mathcal{T}(\mathbb{S},\mathbb{P})$ is a subspace of a quotient of this set.

On the other hand, if $(\mathbb{S},\mathbb{M})$ is a marked bordered surface such that $\mathbb{S}$ has nonempty boundary, then we can define a corresponding marked bordered surface $(\mathbb{\mathbb{S}}^\circ,\mathbb{M}^\circ)$ where $\mathbb{\mathbb{S}}^\circ$ is the surface~$\mathbb{S}$ equipped with the opposite orientation and $\mathbb{M}^\circ$ is the image of~$\mathbb{M}$ under the tautological map $\mathbb{S}\rightarrow\mathbb{S}^\circ$. We then define a marked bordered surface $(\widehat{\mathbb{S}},\widehat{\mathbb{M}})$ where $\widehat{\mathbb{S}}$ is obtained by gluing $\mathbb{S}$ and~$\mathbb{S}^\circ$ along corresponding boundary segments and $\widehat{\mathbb{M}}$ is the union of the images of~$\mathbb{M}$ and~$\mathbb{M}^\circ$ in the resulting surface. Note that $\widehat{\mathbb{S}}$ is a closed surface and $\widehat{\mathbb{M}}$ consists entirely of punctures so that we may apply Lemma~\ref{lem:embedding} to give $\mathcal{T}(\widehat{\mathbb{S}},\widehat{\mathbb{M}})$ the structure of a topological space. If $(C,\psi)\in\mathcal{T}(\mathbb{S},\mathbb{M})$, then by a similar doubling of $C$ along its totally geodesic boundary segments, we get a point $(\widehat{C},\widehat{\psi})\in\mathcal{T}(\widehat{S},\widehat{M})$ where the hyperbolic surface $\widehat{C}$ is invariant under a natural involution. This defines an embedding $\mathcal{T}(\mathbb{S},\mathbb{M})\hookrightarrow\mathcal{T}(\widehat{\mathbb{S}},\widehat{\mathbb{M}})$, and hence $\mathcal{T}(\mathbb{S},\mathbb{M})$ acquires a topology.

\subsection{Length data}

Below we will be interested in some data naturally associated to the boundary of a hyperbolic surface. Given a point $(\psi,C)\in\mathcal{T}(\mathbb{S},\mathbb{M})$, we will assign a number $L_i$ called a \emph{length~datum} to each boundary component $\partial_i$ of the associated surface~$\mathbb{S}'$ defined in Section~\ref{sec:ModuliSpaces}. We first consider a boundary component $\partial_i$ with no marked points. Such a boundary component arises from a puncture in the original surface~$\mathbb{S}$. By applying the map~$\psi$, we see that this corresponds either to a closed boundary component or a cusp of~$C$. In the first case, we define $L_i$ to be the hyperbolic length of the boundary component; in the second case, we define $L_i=0$.

Next suppose that $\partial_i$ is a component of~$\partial\mathbb{S}'$ containing marked points. Let $p_1,\dots,p_k$ denote the marked points on~$\partial_i$ in cyclic order where we consider the indices modulo~$k$. We start off by assuming that $k$ is even. By applying the map~$\psi$ and then lifting to the universal cover, we see that these points $p_j$ correspond to a $\pi_1(C)$-invariant collection of points on~$\partial\bar{\mathbb{H}}$. Let us choose a $\pi_1(C)$-invariant collection of horocycles centered around the latter points whose projections to~$C$ are disjoint. If $b_j$ is the segment of $\partial_i$ connecting $p_j$ and $p_{j+1}$, then we can lift the corresponding arc in~$C$ to a geodesic $\widetilde{b}_j\subset\mathbb{H}$ connecting some points $\widetilde{p}_j$,~$\widetilde{p}_{j+1}\in\partial\bar{\mathbb{H}}$. We define $\mu_j$ to be the length of the geodesic $\widetilde{b}_j$ between the horocycles at its endpoints and set 
\[
L_i=\pm\frac{1}{2}\sum_j(-1)^j\mu_j.
\]
This is well defined up to a sign and, by the argument of~\cite{Gupta19}, Lemma~2.10, is independent of the choice of horocycles. Borrowing terminology from~\cite{Gupta19}, we call this quantity the \emph{metric residue} associated to~$\partial_i$. If $\partial_i$ is component of~$\partial\mathbb{S}'$ with an odd number of marked points, then we define the metric residue to be $L_i=0$.

\subsection{Hyperbolic crowns}

Following~\cite{Gupta19}, we define a \emph{crown} to be a hyperbolic surface which is topologically a closed annulus with finitely many points deleted from one of its boundary components. Note that geometrically such a surface has a single closed geodesic boundary component and a cusp corresponding to each of the deleted points (see the right hand side of Figure~\ref{fig:crown}). Given a crown~$A$, let us label the cusps by $1,\dots,k$ in cyclic order where $k$ is the number of cusps on~$A$. We denote by $\widetilde{A}\subset\mathbb{H}$ the universal cover of~$A$, considered as a subset of the Poincar\'e disk model of the hyperbolic plane~$\mathbb{H}$. By applying a suitable isometry of~$\mathbb{H}$, we can assume that the geodesic connecting $\pm1\in\partial\bar{\mathbb{H}}$ projects to the closed boundary component of~$A$ and that the point $\mathrm{i}\in\partial\bar{\mathbb{H}}$ corresponds to the cusp labeled~1 (see Figure~\ref{fig:crown}).

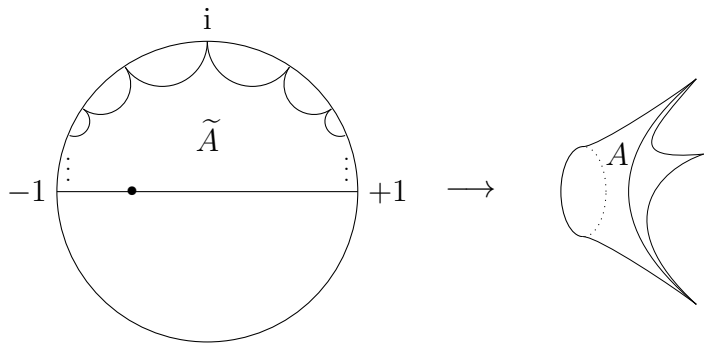
\begin{figure}[ht]
\begin{center}
\begin{tikzpicture}
\draw[black, thin] (0,0) circle (2);
\draw[black, thin] (-2,0) -- (2,0);
\draw[black, thin] (0,2) arc(180:325:0.6);
\draw[black, thin] (1.091491,1.65585) arc(145:305:0.4);
\draw[black, thin] (1.648582,1.098754) arc(125:290:0.2);
\draw[black, thin] (0,2) arc(0:-145:0.6);
\draw[black, thin] (-1.091491,1.65585) arc(35:-125:0.4);
\draw[black, thin] (-1.648582,1.098754) arc(55:-110:0.2);
\node at (-1.85,0.4) {{\tiny $\vdots$}};
\node at (1.85,0.4) {{\tiny $\vdots$}};
\node at (0,0.75) {{$\widetilde{A}$}};
\node at (-1,0) {{\tiny $\bullet$}};
\node at (0,2.3) {$\mathrm{i}$};
\node at (-2.4,0) {$-1$};
\node at (2.4,0) {$+1$};
\node at (3.5,0) {$\longrightarrow$};
\begin{scope}
\clip(5,-0.6) rectangle (5.3,0.6);
\draw[black, thin, dotted] (5,0) ellipse (0.3 and 0.6);
\end{scope}
\begin{scope}
\clip(4.7,-0.6) rectangle (5,0.6);
\draw[black, thin] (5,0) ellipse (0.3 and 0.6);
\end{scope}
\draw[black, ultra thin] (5,0.6) .. controls (5.2,0.6) and (6.25,1.3) .. (6.5,1.5);
\draw[black, ultra thin] (5,-0.6) .. controls (5.2,-0.6) and (6.25,-1.3) .. (6.5,-1.5);
\draw[black, ultra thin] (6.5,1.5) .. controls (5.3,0.5) and (5.3,-0.5) .. (6.5,-1.5);
\draw[black, ultra thin] (6.6,0.5) .. controls (6,0.3) and (5.3,-0.4) .. (6.5,-1.5);
\draw[black, ultra thin] (6.5,1.5) .. controls (5.9,0.8) and (5.5,0.4) .. (6.6,0.5);
\node at (5.45,0.5) {{$A$}};
\end{tikzpicture}
\end{center}
\caption{The universal cover of a hyperbolic crown.\label{fig:crown}}
\end{figure}

Having fixed the normalization of~$\widetilde{A}$ in this way, let us choose a basepoint on the geodesic connecting $\pm1\in\partial\bar{\mathbb{H}}$. Then we can introduce an additional real parameter, the \emph{boundary twist}, which is defined as the signed distance from this basepoint to~$0\in\mathbb{H}$. Below we will use this parameter to glue hyperbolic crowns in the same way that the Fenchel-Nielsen twist parameter is used in classical Teichm\"uller theory to glue pairs of pants.

Now let $\mathbb{S}$ be a closed oriented disk and $\mathbb{M}\subset\mathbb{S}$ a set consisting of $k$ marked points on~$\partial\mathbb{S}$ and one puncture in the interior of~$\mathbb{S}$. Label the marked points on the boundary by $1,\dots,k$ in cyclic order. If we are given a crown $A$ with cusps labeled by $1,\dots,k$, then the correspondence between marked points and cusps gives rise to a homeomorphism $\psi:\mathbb{S}\setminus\mathbb{M}\rightarrow A^\circ$ and hence a pair $(A,\psi)$, well defined up to equivalence. Thus we can view the set of crowns with $k$ labeled cusps as a subspace of $\mathcal{T}(\mathbb{S},\mathbb{M})$. Choosing a nonnegative number $L\in\mathbb{R}_{\geq0}$ so that $L=0$ if $k$ is odd, we define
\[
\Cr_L(k)\subset\mathcal{T}(\mathbb{S},\mathbb{M})
\]
to be the subset of crowns with metric residue equal to~$L$. We will also write $\Cr_L^*(k)\cong\Cr_L(k)\times\mathbb{R}$ for the space parametrizing crowns in~$\Cr_L(k)$ together with a choice of boundary twist.

\begin{lemma}[\cite{Gupta19}, Lemma~2.13]
\label{lem:parametrizecrowns}
Let $q=k$ if $k$ is odd and $q=k-1$ if $k$ is even. Then there is a homeomorphism $\Cr_L(k)\cong\mathbb{R}^q$.
\end{lemma}

\subsection{Prescribing the length data}

We now define spaces parametrizing marked hyperbolic surfaces with prescribed length data. Let $(\mathbb{S},\mathbb{M})$ be a marked bordered surface, and let $\mathbb{S}'$ be the surface from Section~\ref{sec:ModuliSpaces} with boundary components $\partial_1,\dots,\partial_d$. Choose a vector $L=(L_i)\in\mathbb{R}_{\geq0}^d$ so that $L_i=0$ whenever $\partial_i$ has an odd number of marked points. We then write 
\[
\mathcal{T}_L(\mathbb{S},\mathbb{M})\subset\mathcal{T}(\mathbb{S},\mathbb{M})
\]
for the set of all marked hyperbolic surfaces $(C,\psi)$ such that $L_i$ is the length datum corresponding to~$\partial_i$. This set $\mathcal{T}_L(\mathbb{S},\mathbb{M})$ has a topology coming from the topology on~$\mathcal{T}(\mathbb{S},\mathbb{M})$.

\begin{lemma}
\label{lem:dimfixedboundarylength}
Let $(\mathbb{S},\mathbb{M})$ be a marked bordered surface, and let $k_i$ be the number of marked points on the $i$th boundary component of the associated surface~$\mathbb{S}'$. Suppose 
\[
q=6g-6+\sum_is_i\geq0
\]
where $g=g(\mathbb{S})$ and we write $s_i=k_i+3$ if $k_i$ is odd and $s_i=k_i+2$ if $k_i$ is even. Then there is a homeomorphism $\mathcal{T}_L(\mathbb{S},\mathbb{M})\cong\mathbb{R}^q$.
\end{lemma}

\begin{proof}
If $g=0$ and $\mathbb{S}'$ has exactly one boundary component with $k$ marked points, then the condition $q\geq0$ implies that $k\geq3$. If $(C,\psi)\in\mathcal{T}(\mathbb{S},\mathbb{M})$ then we can identify $C$ with an ideal $k$-gon in~$\mathbb{H}$. If $p_1,\dots,p_k\in\mathbb{M}$ are the marked points in cyclic order, then after composing with an isometry of~$\mathbb{H}$, we can assume that $p_1$, $p_2$, and~$p_3$ correspond to the points 0, 1, and~$\infty\in\mathbb{RP}^1\cong\partial\bar{\mathbb{H}}$, respectively. If $k$ is odd, then we have $q=k-3$, so by using the positions of the remaining $k-3$ vertices as coordinates, we can construct a homeomorphism $\mathcal{T}_L(\mathbb{S},\mathbb{M})=\mathcal{T}(\mathbb{S},\mathbb{M})\cong\mathbb{R}^q$. If $k$ is even, then we can freely choose the positions of $k-4$ vertices, and the remaining one is determined by the metric residue. Hence we have $\mathcal{T}_L(\mathbb{S},\mathbb{M})\cong\mathbb{R}^q$ in this case as well.

Next, suppose that $g>0$ or that $\mathbb{S}'$ has more than one boundary component. If $\partial_i$ is a component of~$\partial\mathbb{S}'$ with $k_i>0$ marked points, then it follows from our assumption that there is a nontrivial loop $\gamma_i\subset\mathbb{S}$ which is retractable to~$\partial_i$. Given a general point $(C,\psi)\in\mathcal{T}(\mathbb{S},\mathbb{M})$, the image $\psi(\gamma_i)$ is homotopic to a geodesic in~$C$. By cutting along this geodesic for each boundary component $\partial_i$ having marked points, we obtain a decomposition 
\[
C=C_0\cup\bigcup_iA_i
\]
where $C_0$ is a surface of genus~$g$ with punctures and boundary and $A_i$ is a crown with $k_i$ cusps. By Lemma~\ref{lem:parametrizecrowns}, the hyperbolic structure on~$A_i$ with its prescribed metric residue is determined by $k_i$ independent real parameters if $k_i$ is odd or $k_i-1$ parameters if $k_i$ is even. A pants decomposition for~$C_0$ consists of $3g-3+d$ essential closed curves where $d$ is the number of boundary components of~$\mathbb{S}'$, so a hyperbolic structure on $C_0$ with specified boundary lengths is determined by $6g-6+2d$ Fenchel-Nielsen coordinates corresponding to the curves in the pants decomposition. To glue the surfaces $C_0$ and $A_i$, we require a choice of boundary twist. Combining these facts one sees that $\mathcal{T}_L(\mathbb{S},\mathbb{M})\cong\mathbb{R}^q$ as desired.
\end{proof}

\subsection{Orientation of boundary components}

We now give the definition of the enhanced Teichm\"uller space of a marked bordered surface~$(\mathbb{S},\mathbb{M})$. This construction modifies the space $\mathcal{T}(\mathbb{S},\mathbb{M})$ defined previously by associating additional data to the boundary components of the surfaces.

\begin{definition}
The \emph{enhanced Teichm\"uller space} $\mathcal{T}^\pm(\mathbb{S},\mathbb{M})$ is the branched $2^{|\mathbb{P}|}$-fold cover 
\[
\mathcal{T}^\pm(\mathbb{S},\mathbb{M})\rightarrow\mathcal{T}(\mathbb{S},\mathbb{M})
\]
where a point in the fiber over $(C,\psi)$ is obtained by choosing an orientation for each component of~$\partial C$ which is homeomorphic to~$S^1$. This cover is branched over the set of points $(C,\psi)$ where the surface $C$ has at least one puncture with a cusp neighborhood.
\end{definition}

There is a natural action of the signed mapping class group $\MCG^\pm(\mathbb{S},\mathbb{M})$ on $\mathcal{T}^\pm(\mathbb{S},\mathbb{M})$ where the $\mathbb{Z}_2^{\mathbb{P}}$ factor acts by changing the orientations of the boundary components. We will see below that the enhanced Teichm\"uller space is homeomorphic to a Euclidean space of some dimension given explicitly in terms of the genus and number of marked points on~$(\mathbb{S},\mathbb{M})$.

\subsection{Ideal triangulations}

Given a marked bordered surface~$(\mathbb{S},\mathbb{M})$, we define an \emph{arc} to be a smooth path $\alpha$ on~$\mathbb{S}$ connecting points of~$\mathbb{M}$ whose interior lies in $\mathbb{S}\setminus\mathbb{M}$ and which has no self-intersections in its interior. In addition, we require that $\alpha$ is not homotopic to a single point or to a segment of $\partial\mathbb{S}$ containing no marked points in its interior via a homotopy through such paths. Two arcs are considered equivalent if they are related by a homotopy through arcs or a reversal of orientations. Two arcs are \emph{compatible} if we can find arcs in their respective equivalence classes that do not intersect in~$\mathbb{S}\setminus\mathbb{M}$. An \emph{ideal triangulation} of~$(\mathbb{S},\mathbb{M})$ is a maximal collection of pairwise compatible arcs on~$(\mathbb{S},\mathbb{M})$, considered up to equivalence.

Given an ideal triangulation~$T$ of~$(\mathbb{S},\mathbb{M})$, we always choose representatives for the arcs of~$T$ which do not intersect in~$\mathbb{S}\setminus\mathbb{M}$. Then a \emph{triangle} of~$T$ is defined to be the closure in~$\mathbb{S}$ of a connected component of the complement of all arcs of~$T$. Topologically, such a triangle is a disk containing two or three marked points. If a triangle has only two marked points, it is said to be \emph{self-folded} (see Figure~\ref{fig:triangles}). In this case, there is a marked point in the interior of the triangle and a unique arc incident to this point which we call the \emph{internal edge}. The remaining edge of a self-folded triangle is called the \emph{encircling edge}.

\begin{figure}[ht]
\begin{center}
\begin{tikzpicture}
\draw[black, thin] (-3,0) -- (-1,0);
\draw[black, thin] (-3,0) -- (-2,1.732);
\draw[black, thin] (-1,0) -- (-2,1.732);
\draw[black, thin] (2,1) circle (1);
\draw[black, thin] (2,0) -- (2,1);
\node at (-1,0) {{\tiny $\bullet$}};
\node at (-3,0) {{\tiny $\bullet$}};
\node at (-2,1.732) {{\tiny $\bullet$}};
\node at (2,0) {{\tiny $\bullet$}};
\node at (2,1) {{\tiny $\bullet$}};
\end{tikzpicture}
\end{center}
\caption{An ordinary triangle (left) and a self-folded triangle (right).\label{fig:triangles}}
\end{figure}
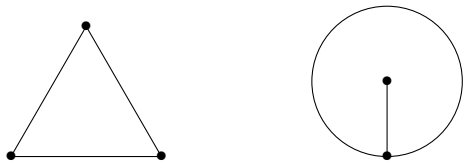

If $T$ is an ideal triangulation of~$(\mathbb{S},\mathbb{M})$, then there is an associated skew-symmetric matrix that encodes the combinatorics of~$T$. For each arc~$\alpha$ of~$T$, let us write $\pi_T(\alpha)$ for the arc defined as follows: If $\alpha$ is the internal edge of a self-folded triangle, then $\pi_T(\alpha)$ is the encircling edge, and $\pi_T(\alpha)=\alpha$ otherwise. Then for each non-self-folded triangle~$t$ of the triangulation~$T$, we define a number~$\varepsilon_{\alpha\beta}^t$ by
\begin{enumerate}
\item $\varepsilon_{\alpha\beta}^t=+1$ if $\pi_T(\alpha)$ and $\pi_T(\beta)$ are arcs of~$t$ with $\pi_T(\beta)$ immediately following $\pi_T(\alpha)$ as we travel around~$t$ in the counterclockwise direction.
\item $\varepsilon_{\alpha\beta}^t=-1$ if the same holds with the counterclockwise direction.
\item $\varepsilon_{\alpha\beta}^t=0$ otherwise.
\end{enumerate}
Then we define the \emph{exchange matrix} of~$T$ to be the square matrix indexed by arcs of~$T$ with $(\alpha,\beta)$~entry given by $\varepsilon_{\alpha\beta}^T=\sum_t\varepsilon_{\alpha\beta}^t$ where the sum is taken over all non-self-folded triangles.

There is also an extension of the notion of an ideal triangulation that will be important later. Given a marked bordered surface $(\mathbb{S},\mathbb{M})$ as before, we define a \emph{signing} to be a function $\epsilon:\mathbb{P}\rightarrow\{\pm1\}$ which associates a sign $\epsilon(p)$ to each puncture $p\in\mathbb{P}$. We define a \emph{signed triangulation} of~$(\mathbb{S},\mathbb{M})$ to be a pair $(T,\epsilon)$ consisting of an ideal triangulation~$T$ of~$(\mathbb{S},\mathbb{M})$ and a signing~$\epsilon$. Two signed triangulations $(T_1,\epsilon_1)$ and $(T_2,\epsilon_2)$ are considered to be equivalent if we have $T_1=T_2$ and the signings $\epsilon_i$ differ only at a puncture in the interior of a self-folded triangle. This generates an equivalence relation on the set of all signed triangulations, and an equivalence class of signed triangulations is called a \emph{tagged triangulation}. Given a tagged triangulation~$\tau$, any arc of a representative signed triangulation determines a \emph{tagged arc} of~$\tau$. Suppose $(T,\epsilon_1)$ and $(T,\epsilon_2)$ are two representatives for~$\tau$ where the $\epsilon_i$ differ only at a puncture in the interior of a self-folded triangle. Let $\alpha$ be the interior edge of this triangle and $\beta$ the encircling edge. Then the tagged arc represented by~$\alpha$ in~$(T,\epsilon_1)$ is considered to be equivalent to the tagged arc represented by~$\beta$ in~$(T,\epsilon_2)$.

\subsection{Shear coordinates}

Suppose that we are given a point of $\mathcal{T}^\pm(\mathbb{S},\mathbb{P})$. This consists of a point $(C,\psi)\in\mathcal{T}(\mathbb{S},\mathbb{P})$ together with a choice of orientation for each component of the boundary of~$C$ which is homeomorphic to~$S^1$. We can view the universal cover of~$C$ as a subset $\widetilde{C}\subset\mathbb{H}$ with totally geodesic boundary, and if $g\subset\widetilde{C}$ is a geodesic that projects to an $S^1$ boundary component of~$C$, then the orientation of this boundary component determines a distinguished endpoint of~$g$. If $T$ is an ideal triangulation of~$(\mathbb{S},\mathbb{P})$, then we can represent $T$ by a collection of arcs on~$\mathbb{S}$ which do not intersect in~$\mathbb{S}\setminus\mathbb{P}$. Let $\alpha$ be any one of these arcs, and let $\widetilde{\alpha}\subset\widetilde{C}$ be a curve that projects onto $\psi(\alpha)\subset C$ (see Figure~\ref{fig:liftarc}). We will modify this curve in two ways. First, if $\widetilde{\alpha}$ has an endpoint on some boundary geodesic $g$ of~$\widetilde{C}$ projecting to an $S^1$ boundary component of~$C$, then we will drag this endpoint along $g$ until it coincides with the distinguished endpoint of~$g$ determined by the orientation. We will then straighten the resulting curves to geodesics in~$\mathbb{H}$ (see Figure~\ref{fig:geodesicspiral}).

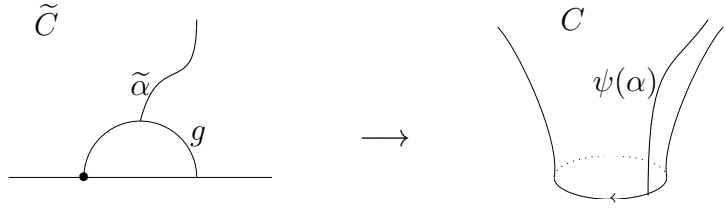
\begin{figure}[ht]
\begin{center}
\begin{tikzpicture}
\node at (-4.5,2.1) {$\widetilde{C}$};
\node at (-3.25,1.25) {$\widetilde{\alpha}$};
\draw[black, thin] (-2.5,2.1) .. controls (-2.5,1) and (-3,1.75) .. (-3.25,0.75);
\node at (-2.48,0.55) {$g$};
\draw[black, ultra thin] (-2.5,0) arc(0:180:0.75);
\draw[black, ultra thin] (-5,0) -- (-1.5,0);
\node at (-4,0) {{\tiny $\bullet$}};
\node at (0,0.5) {$\longrightarrow$};
\node at (2.5,2.1) {$C$};
\node at (3.2,1.25) {$\psi(\alpha)$};
\draw[black, thin] (4.3,2.1) .. controls (3.8,1.4) and (3.5,1.6) .. (3.5,-0.24);
\draw[black, ultra thin] (1.5,2) .. controls (1.75,1.75) and (2.35,0.5) .. (2.25,0);
\draw[black, ultra thin] (4.5,2) .. controls (4.25,1.75) and (3.65,0.5) .. (3.75,0);
\begin{scope}
\clip(2.25,-0.281) rectangle (3.75,0);
\draw[black, ultra thin] (3,0) ellipse (0.75 and 0.29);
\end{scope}
\begin{scope}
\clip(2.25,0) rectangle (3.75,0.281);
\draw[black, thin, dotted] (3,0) ellipse (0.75 and 0.29);
\end{scope}
\draw[black, thin, ->] (3.001,-0.27)--(3,-0.27);
\end{tikzpicture}
\end{center}
\caption{Lifting an arc to the universal cover.\label{fig:liftarc}}
\end{figure}

In this way, we obtain a $\pi_1(C)$-invariant collection of geodesics in~$\mathbb{H}$ which decompose the universal cover $\widetilde{C}$ into ideal triangles. If $C$ has a boundary component homeomorphic to~$S^1$, then the images of these geodesics under the covering map $\widetilde{C}\rightarrow C$ spiral into this boundary component in the direction prescribed by the orientation (see Figure~\ref{fig:geodesicspiral}).

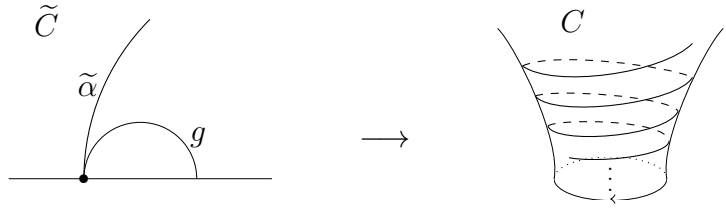
\begin{figure}[ht]
\begin{center}
\begin{tikzpicture}
\node at (-4.5,2.1) {$\widetilde{C}$};
\node at (-3.95,1.25) {$\widetilde{\alpha}$};
\draw[black, thin] (-4,0) arc(0:-45:-3);
\node at (-2.48,0.55) {$g$};
\draw[black, ultra thin] (-2.5,0) arc(0:180:0.75);
\draw[black, ultra thin] (-5,0) -- (-1.5,0);
\node at (-4,0) {{\tiny $\bullet$}};
\node at (0,0.5) {$\longrightarrow$};
\node at (2.5,2.1) {$C$};
\draw[black, ultra thin] (1.5,2) .. controls (1.75,1.75) and (2.35,0.5) .. (2.25,0);
\draw[black, ultra thin] (4.5,2) .. controls (4.25,1.75) and (3.65,0.5) .. (3.75,0);
\begin{scope}
\clip(2.45,0.25) rectangle (3.78,0.5);
\draw[black, thin] (2.22,0.4) .. controls (2.25,0.15) and (3.75,0.25) .. (3.78,0.5);
\end{scope}
\draw[black, thin, dashed] (2.17,0.7) .. controls (2.25,0.8) and (3.75,0.8) .. (3.78,0.5);
\draw[black, thin] (2.17,0.7) .. controls (2.25,0.45) and (3.75,0.55) .. (3.9,0.9);
\draw[black, thin, dashed] (2,1.1) .. controls (2.25,1.2) and (3.75,1.1) .. (3.9,0.9);
\draw[black, thin] (2,1.1) .. controls (2.25,0.8) and (3.9,1) .. (4.09,1.35);
\draw[black, thin, dashed] (1.83,1.5) .. controls (2.25,1.7) and (3.75,1.5) .. (4.09,1.35);
\draw[black, thin] (1.83,1.5) .. controls (2.25,1.2) and (3.7,1.4) .. (4.09,1.8);
\node[black] at (3,0.1) {{\tiny $\vdots$}};
\begin{scope}
\clip(2.25,-0.281) rectangle (3.75,0);
\draw[black, ultra thin] (3,0) ellipse (0.75 and 0.29);
\end{scope}
\begin{scope}
\clip(2.25,0) rectangle (3.75,0.281);
\draw[black, thin, dotted] (3,0) ellipse (0.75 and 0.29);
\end{scope}
\draw[black, thin, ->] (3.001,-0.27)--(3,-0.27);
\end{tikzpicture}
\end{center}
\caption{A geodesic spiral.\label{fig:geodesicspiral}}
\end{figure}

Let $\alpha$ be an arc of the ideal triangulation~$T$, and let $\widetilde{\alpha}$ be one of the corresponding geodesics in~$\mathbb{H}$ obtained by the above construction. Then the decomposition of~$\widetilde{C}$ includes two ideal triangles which share the edge~$\widetilde{\alpha}$. Their union is a quadrilateral with vertices in~$\partial\bar{\mathbb{H}}$. Let us label the vertices as~$z_1,\dots,z_4\in\partial\bar{\mathbb{H}}$ so that the order is compatible with the orientation of $\partial\bar{\mathbb{H}}\cong\mathbb{RP}^1$ and $\widetilde{\alpha}$ connects~$z_1$ and~$z_3$ (see Figure~\ref{fig:shear}). Then we can associate to $\alpha$ the cross ratio 
\begin{equation}
\label{eqn:crossratio}
Y_\alpha=\frac{(z_1-z_2)(z_3-z_4)}{(z_2-z_3)(z_1-z_4)}\in\mathbb{R}_{>0}.
\end{equation}
One can check that this cross ratio is independent of the choice of the geodesic lift~$\widetilde{\alpha}$ as well as the labeling of the vertices of the quadrilateral. The logarithm $\log Y_\alpha$ is known as the \emph{shear coordinate} associated to the arc~$\alpha$.

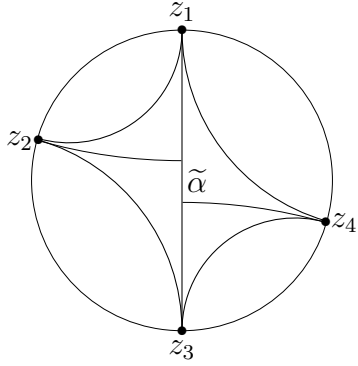
\begin{figure}[ht]
\begin{center}
\begin{tikzpicture}
\draw[black, ultra thin] (0,2) circle (2);
\draw[black, thin] (0,0) -- (0,4);
\draw[black, thin] (0,0) arc(180:74:1.5);
\draw[black, thin] (0,4) arc(180:254:2.64);
\draw[black, thin] (0,4) arc(0:-106:1.5);
\draw[black, thin] (0,0) arc(0:74:2.64);
\node[black] at (0,0) {{\tiny $\bullet$}};
\node[black] at (0,4) {{\tiny $\bullet$}};
\node[black] at (1.91,1.44) {{\tiny $\bullet$}};
\node[black] at (-1.91,2.53) {{\tiny $\bullet$}};
\node[black] at (0,4.25) {{$z_1$}};
\node[black] at (2.16,1.44) {{$z_4$}};
\node[black] at (0,-0.25) {$z_3$};
\node[black] at (-2.16,2.53) {{$z_2$}};
\draw[black, ultra thin] (1.91,1.44) arc(74:90:6.9);
\draw[black, ultra thin] (-1.91,2.53) arc(-106:-90:6.9);
\node[black] at (0.2,2) {{$\widetilde{\alpha}$}};
\end{tikzpicture}
\end{center}
\caption{Defining the shear coordinate.\label{fig:shear}}
\end{figure}

One can also describe the shear coordinate more geometrically as follows. For $i=2,4$, there exists a unique geodesic in~$\mathbb{H}$ that starts at~$z_i$ and meets the geodesic~$\widetilde{\alpha}$ orthogonally (see Figure~\ref{fig:shear}). If we write~$p_i$ for the point of intersection between this geodesic and~$\widetilde{\alpha}$, then the shear coordinate equals the signed distance between~$p_2$ and~$p_4$. The sign detects whether $p_2$ appears before or after~$p_4$ for a given choice of orientation of~$\widetilde{\alpha}$. See \cite{Penner12}, Chapter~1, Corollary~4.16 for the proof.

\begin{proposition}
\label{prop:shearcoordinates}
Let $(\mathbb{S},\mathbb{M})$ be a marked bordered surface, and if $g(\mathbb{S})=0$ assume that $|\mathbb{M}|\geq3$. Then for any ideal triangulation of~$(\mathbb{S},\mathbb{M})$, the associated cross ratios~\eqref{eqn:crossratio} provide a homeomorphism $\mathcal{T}^{\pm}(\mathbb{S},\mathbb{M})\cong\mathbb{R}_{>0}^n$ where the dimension~$n$ is given by~\eqref{eqn:dimension}.
\end{proposition}

\begin{proof}
By Proposition~2.10 of~\cite{FominShapiroThurston08}, any ideal triangulation of~$(\mathbb{S},\mathbb{M})$ consists of~$n$ arcs. The inverse homeomorphism is described in Chapter~2, Theorem~4.4 of~\cite{Penner12}.
\end{proof}

Note that the space $\mathcal{T}(\mathbb{S},\mathbb{M})$ considered above is the quotient of $\mathcal{T}^\pm(\mathbb{S},\mathbb{M})$ by the action of the group~$\mathbb{Z}_2^{\mathbb{P}}$. It therefore follows from Proposition~\ref{prop:shearcoordinates} that this space $\mathcal{T}(\mathbb{S},\mathbb{M})$ has the natural structure of an orbifold.

\subsection{Cluster coordinates}

From Proposition~\ref{prop:shearcoordinates} we see that, for any ideal triangulation of~$(\mathbb{S},\mathbb{M})$, the associated shear coordinates provide a global parametrization of enhanced Teichm\"uller space. We will also consider an extension of the above construction which is better suited to dealing with changes of triangulation. To describe the extended construction, suppose we are given a signed triangulation $(T,\epsilon)$. By changing the orientations of boundary components, we get an action of $\mathbb{Z}_2^{\mathbb{P}}$ on~$\mathcal{T}^\pm(\mathbb{S},\mathbb{M})$ by homeomorphisms. Given a point of $\mathcal{T}^\pm(\mathbb{S},\mathbb{M})$, let us act on this point by $\epsilon\in\mathbb{Z}_2^{\mathbb{P}}$ and write $Y_\alpha\in\mathbb{R}_{>0}$ for the cross ratio~\eqref{eqn:crossratio} associated to the resulting point of~$\mathcal{T}^\pm(\mathbb{S},\mathbb{M})$ and an arc~$\alpha$ of~$T$. We consider two cases:
\begin{enumerate}
\item If $\alpha$ is not the interior edge of a self-folded triangle, then we define $X_\alpha=Y_\alpha$.
\item If $\alpha$ is the interior edge of a self-folded triangle, let $\beta$ denote the encircling edge. Then we define $X_\alpha=Y_\alpha Y_\beta$.
\end{enumerate}
In this way, we get a collection of parameters $X_\alpha\in\mathbb{R}_{>0}$ which provide a homeomorphism as in the statement of Proposition~\ref{prop:shearcoordinates}. As in Lemma~9.6 of~\cite{AllegrettiBridgeland20}, one can show that the number~$X_\alpha$ associated to an arc~$\alpha$ of a signed triangulation depends only on the underlying tagged arc. These numbers are called \emph{cluster coordinates} or \emph{Fock-Goncharov coordinates}.

Let $\gamma$ be any arc of~$T$. We say that an ideal triangulation~$T'$ is obtained from~$T$ by a \emph{flip} of~$\gamma$ if $T'$ is different from~$T$ and there is an arc $\gamma'$ of~$T'$ such that $T\setminus\{\gamma\}=T'\setminus\{\gamma'\}$. In this case we also say that the signed triangulation $(T',\epsilon)$ is obtained from $(T,\epsilon)$ by a flip of~$\gamma$ and that the underlying tagged triangulation of $(T',\epsilon)$ is obtained from the underlying tagged triangulation of $(T,\epsilon)$ by a flip of the tagged arc~$\gamma$. Figure~\ref{fig:flip} illustrates a neighborhood of the arcs $\gamma$ and~$\gamma'$.
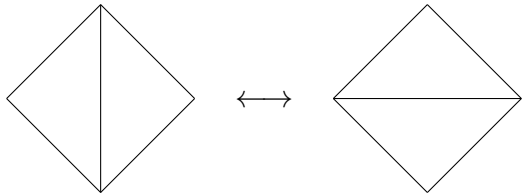
\begin{figure}[ht]
\begin{center}
\[
\begin{tikzpicture}[baseline=-0.65ex]
\coordinate (a) at (1.25,0);
\coordinate (b) at (0,1.25);
\coordinate (c) at (-1.25,0);
\coordinate (d) at (0,-1.25);
\draw[black, thin] (a) -- (b);
\draw[black, thin] (b) -- (c);
\draw[black, thin] (c) -- (d);
\draw[black, thin] (d) -- (a);
\draw[black, thin] (b) -- (d);
\end{tikzpicture}
\quad
\longleftrightarrow
\quad
\begin{tikzpicture}[baseline=-0.65ex]
\coordinate (a) at (1.25,0);
\coordinate (b) at (0,1.25);
\coordinate (c) at (-1.25,0);
\coordinate (d) at (0,-1.25);
\draw[black, thin] (a) -- (b);
\draw[black, thin] (b) -- (c);
\draw[black, thin] (c) -- (d);
\draw[black, thin] (d) -- (a);
\draw[black, thin] (a) -- (c);
\end{tikzpicture}
\]
\end{center}
\caption{A flip of an arc.\label{fig:flip}}
\end{figure}

If $(T,\epsilon)$ and~$(T',\epsilon)$ are related by a flip of the arc~$\gamma$, then there is a natural bijection between the arcs of $T$ and the arcs of~$T'$. Abusing notation, we will use the same symbol to denote an arc of~$T$ and the corresponding arc of~$T'$. We will write $X_\alpha$ and $X_\alpha'$ for the Fock-Goncharov coordinates with respect to an arc~$\alpha$ of the signed triangulations $(T,\epsilon)$ and $(T',\epsilon)$, respectively.

\begin{proposition}
\label{prop:changecoordinates}
Take notation as in the last paragraph. Then the coordinates~$X_\alpha$ and~$X_\alpha'$ are related by 
\[
X_\alpha'=
\begin{cases}
X_\gamma^{-1} & \text{if $\alpha=\gamma$} \\
X_\alpha\left(1+X_\gamma^{-\operatorname{sgn}(\varepsilon_{\alpha\gamma})}\right)^{-\varepsilon_{\alpha\gamma}} & \text{if $\alpha\neq\gamma$}
\end{cases}
\]
where $\varepsilon_{\alpha\gamma}=\varepsilon_{\alpha\gamma}^T$ is the exchange matrix associated to~$T$.
\end{proposition}

\begin{proof}
The same statement appears in~\cite{AllegrettiBridgeland20}, Proposition~9.8, for cluster coordinates on moduli spaces of local systems. The proof is the same in our setting.
\end{proof}

\section{Harmonic maps}
\label{sec:HarmonicMaps}

In this section, we review the basic theory of harmonic maps between Riemann surfaces and prove a result on the existence and uniqueness of such maps.

\subsection{Basic definitions}

We begin by recalling the general notion of a harmonic map between Riemannian manifolds $(M,g)$ and~$(N,h)$. In terms of local coordinates $\{x^{\alpha}\}$ on~$M$ and $\{y^i\}$ on~$N$, the Riemannian metrics can be written $g=g_{\alpha\beta}dx^\alpha\otimes dy^\beta$ and $h=h_{ij}dy^i\otimes dy^j$. (Here we employ the Einstein summation convention and sum over repeated indices.) Given a smooth map $f:M\rightarrow N$, we define the \emph{energy density} of~$f$ by the formula 
\[
e(f)=\frac{\partial f^i}{\partial x^\alpha}\frac{\partial f^j}{\partial x^\beta}g^{\alpha\beta}h_{ij}\circ f.
\]
One can check that this is independent of the choice of local coordinates and gives a globally defined function on~$M$. We define the \emph{energy} of $f$ on~$V\subset M$ by 
\[
E_V(f)=\int_Ve(f)d\mu_g
\]
where $\mu_g$ is the volume form associated to~$g$. We can view $E_M$ as a functional on the space of smooth maps $M\rightarrow N$. Such a map is called \emph{harmonic} if it is a critical point of this functional~\cite{DW07}.

In this paper, we will deal exclusively with harmonic maps between Riemann surfaces. Let $M$ and $N$ be Riemann surfaces, and assume the metrics are \emph{conformal}. This means that for any local coordinate $z$ on~$M$ and any local coordinate $w$ on~$N$, the metrics are given by local expressions of the form $g|dz|^2$ and $h|dw|^2$ for some positive functions $g$ and~$h$. In this case, one can rewrite the energy density of~$f$ as 
\[
e(f)=\frac{h(f(z))}{g(z)}\left(|f_z|^2+|f_{\bar{z}}|^2\right)
\]
where we write $f_z\coloneqq\partial_zf$, $f_{\bar{z}}\coloneqq\partial_{\bar{z}}f$, etc., and by a common abuse of notation we confuse $f$ with $w\circ f$. The energy density is the sum of the \emph{holomorphic energy} $\mathcal{H}(f)$ and the \emph{antiholomorphic energy} $\mathcal{L}(f)$, given by the expressions 
\[
\mathcal{H}(f)=\frac{h(f(z))}{g(z)}|f_z|^2, \quad \mathcal{L}(f)=\frac{h(f(z))}{g(z)}|f_{\bar{z}}|^2.
\]
The \emph{Jacobian} of~$f$ is defined as the difference $\mathcal{J}(f)=\mathcal{H}(f)-\mathcal{L}(f)$. The statement that $f$ is harmonic is equivalent to an Euler-Lagrange equation, which in this case takes the form 
\begin{equation}
\label{eqn:harmonicity}
f_{z\bar{z}}+\frac{h_w}{h}f_zf_{\bar{z}}=0
\end{equation}
in a local patch on~$M$ (see~\cite{DW07}). This equation implies that the condition of being harmonic does not depend on the exact form of the metric on~$M$ but only on its conformal class. From~equation~\eqref{eqn:harmonicity}, one can derive the \emph{Bochner equation} 
\begin{equation}
\label{eqn:Bochner}
\Delta_g\log\mathcal{H}(f)=-2K_N\mathcal{J}(f)+2K_M
\end{equation}
where $\Delta_g=(4/g)\partial^2/\partial z\partial\bar{z}$ is the Laplace-Beltrami operator, and~$K_M$ and~$K_N$ are the Gaussian curvature of $M$ and~$N$, respectively. A proof of this identity can be found in Section~1 of~\cite{SchoenYau78}.

Using the identities $f^*(dw)=f_zdz+f_{\bar{z}}d\bar{z}$ and $f^*(d\bar{w})=\bar{f}_zdz+\bar{f}_{\bar{z}}d\bar{z}$ together with $\bar{f}_z=\overline{f_{\bar{z}}}$ and $\bar{f}_{\bar{z}}=\overline{f_z}$, one computes 
\begin{equation}
\label{eqn:pullback}
f^*\left(h|dw|^2\right)=\varphi(z) dz^{\otimes2}+g(z)e(f)dz\otimes d\bar{z}+\overline{\varphi(z)}d\bar{z}^{\otimes2}
\end{equation}
where we have written $\varphi(z)=h(f(z))f_z\bar{f}_z$. From this we see that the $(2,0)$-part of the pullback of the metric on $N$ is the quadratic differential $\varphi(z) dz^{\otimes2}$. It is called the \emph{Hopf differential} of the map~$f$. One can also check that 
\[
\partial_{\bar{z}}\varphi=h(f(z))\left(\bar{f}_z\tau(f)+f_z\overline{\tau(f)}\right)
\]
where $\tau(f)$ denotes the \emph{tension} of~$f$, defined as the expression appearing on the left hand side of the harmonicity equation~\eqref{eqn:harmonicity}. From this we see that if the map $f$ is harmonic, then its associated Hopf~differential is holomorphic. The converse also holds provided the Jacobian is nowhere vanishing (see~\cite{Sampson78}).

\subsection{Model maps into crowns}

We will now review a result of Gupta~\cite{Gupta19} which establishes the existence of a harmonic map from the punctured disk $\mathbb{D}^*=\{z\in\mathbb{C}:0<|z|<1\}\subset\mathbb{C}$ to a hyperbolic crown. We will see later that such a map provides a local asymptotic model of a harmonic map into a more general cusped hyperbolic surface.

We consider meromorphic quadratic differentials~$\phi$ on the disk $\mathbb{D}=\{z\in\mathbb{C}:|z|<1\}\subset\mathbb{C}$ which have the form 
\[
\phi(z)=\left(a_mz^{-m}+a_{m-1}z^{-m+1}+\dots+a_2z^{-2}\right)dz^{\otimes2}
\]
for some $m\geq3$. If $P(z)=\pm\frac{1}{z^{m/2}}f(z)dz$ is a principal differential, we will denote by $\Q_P(m)$ the space of quadratic differentials of this form with principal part~$P(z)$. It follows from the arguments of~\cite{Gupta19} that this space is homeomorphic to the Euclidean space $\mathbb{R}^q$ where $q=m-1$ if $m$~is odd and $q=m-2$ if $m$ is even. On the other hand, by Lemma~\ref{lem:parametrizecrowns}, we know that the space $\Cr_L^*(m-2)$ of hyperbolic crowns with prescribed metric residue and a choice of boundary twist is homeomorphic to a Euclidean space of the same dimension.

We can construct a canonical homeomorphism between these spaces using the theory of harmonic maps.  In the following, we define the \emph{analytic residue} of a principal differential $P(z)$ to be the coefficient of $z^{-1}$, which is well defined up to a sign.

\begin{theorem}[\cite{Gupta19}, Theorem~3.2]
\label{thm:modelmap}
Let $P(z)$ be a principal differential, and let $L$ be the real part of its analytic residue. For any quadratic differential $\phi\in\Q_P(m)$, there exists a hyperbolic crown $A\in\Cr_L(m-2)$ and a harmonic map 
\[
f:\mathbb{D}^*\rightarrow A\setminus\partial A
\]
with Hopf differential $\phi$. There is a canonical choice of boundary twist for~$A$ so that the resulting map 
\[
\Q_P(m)\rightarrow\Cr_L^*(m-2)
\]
is a homeomorphism.
\end{theorem}

\subsection{Compatibility conditions}
\label{sec:CompatibilityConditions}

For the rest of this section, we fix a compact Riemann surface $S$ and a nonempty finite subset $M=\{p_1,\dots,p_d\}\subset S$. For each point $p_i\in M$, we fix a neighborhood $p_i\in U_i\subset S$ and a local coordinate $z_i$ providing a biholomorphism $U_i\cong\mathbb{D}$ with $z_i(p_i)=0$. We also choose a principal differential $P_i(z_i)$ in this local coordinate and set $P=(P_i(z_i))$.

The square $\phi_i(z_i)=P_i(z_i)^{\otimes2}$ defines a quadratic differential $\phi_i$ on~$U_i$ with a pole of some order~$m_i$ at the point~$p_i$. By definition, we can have either $m_i=0$,~$2$, or $m_i\geq3$. If we have $m_i\geq3$, then the quadratic differential $\phi_i$ determines $m_i-2$ asymptotic horizontal directions at~$p_i$. Using this fact, we can define a marked bordered surface $(S_P,M_P)$ associated to $(S,M)$ and~$P$. The surface $S_P$ is defined by taking an oriented real blow up of~$S$ at each point $p_i$ for which $m_i\geq3$. The asymptotic horizontal directions determine a collection of points on the boundary of the resulting surface, and we define~$M_P$ to consist of these points together with all points $p_i\in M$ for which $m_i\leq2$. In~particular, note that if $\phi\in\mathcal{Q}_P(S,M)$, then $(S_P,M_P)$ coincides with the marked bordered surface that we previously associated to the pair~$(S,\phi)$.

By applying the constructions of Section~\ref{sec:EnhancedTeichmullerSpace}, we get a space $\mathcal{T}_L(S_P,M_P)$ parametrizing marked hyperbolic surfaces with prescribed length data. Note that the tuples $L=(L_i)$ and $P=(P_i(z_i))$ can be indexed by the same set. We will say that $L$ is \emph{compatible} with $P$ if the following conditions are satisfied: 
\begin{enumerate}
\item If $m_i\leq2$ then we have 
\[
L_i^2=16\pi^2|a_i|\sin^2(\theta_i/2)
\]
where $a_i$ is the leading coefficient of $\phi_i$ at~$p_i$ and $\theta_i=\arg(a_i)$ is its argument.
\item If $m_i\geq3$ then $L_i$ equals the real part of the analytic residue of~$P_i(z_i)$.
\end{enumerate}
If $L=(L_i)$ is a tuple of length data compatible with $P=(P_i(z_i))$, then we will use the abbreviation $\mathcal{T}_P(S,M)\coloneqq\mathcal{T}_L(S_P,M_P)$.

\subsection{Construction of harmonic maps}
\label{sec:ConstructionOfHarmonicMaps}

Take notation as above. In the discussion that follows, we will assume that $\chi(S\setminus M)<0$ and also that each phase $\theta_i$ lies in the interval $(0,2\pi)$. Given a point $(C,\psi)\in\mathcal{T}_P(S,M)$, we will define a collection of maps 
\[
f_i:U_i\setminus\{p_i\}\rightarrow C^\circ
\]
called the \emph{model maps}. We will then construct a collection of harmonic maps, defined on subsets of~$S\setminus M$, whose behavior near~$p_i$ is controlled by~$f_i$.

We first suppose that $p_i\in M$ is a point with $m_i=0$. In this case, we have $L_i=0$ in the compatibility condition, and so this point $p_i$ corresponds via the map $\psi$ to a cusp in~$C$. This cusp has a neighborhood conformally equivalent to a punctured disk in~$\mathbb{C}$, and we define $f_i$ to be any biholomorphism from $U_i\setminus\{p_i\}$ onto such a neighborhood.

Next we suppose that $p_i\in M$ is a point with $m_i=2$. In this case, we have $L_i\neq0$ in the compatibility condition, and so $p_i$ corresponds via the map~$\psi$ to a totally geodesic boundary component $\partial_i\subset\partial C$. Note that if we equip $S\setminus M$ with its uniformizing metric, then there is an isometry
\begin{equation}
\label{eqn:cusp}
U_i\setminus\{p_i\}\cong\{z=x+\mathrm{i}y\in\mathbb{C}:0\leq x\leq1,y>y_0\}/z\sim z+1
\end{equation}
for some $y_0>0$ where the metric on the right hand side is $|dz|^2/y^2$. If $\alpha:[0,1]\rightarrow\partial_i$ is a constant speed parametrization of~$\partial_i$, then we let $u_i:U_i\setminus\{p_i\}\rightarrow\partial_i$ be the map $u_i(x+\mathrm{i}y)=\alpha(x)$ and let $t_i:U_i\setminus\{p_i\}\rightarrow U_i\setminus\{p_i\}$ be the map $t_i(z)= z+\delta_i\cdot(y-y_0)$ where $\delta_i=(\sin\theta_i)/(1-\cos\theta_i)$. We then define $f_i=u_i\circ t_i$.

Finally, suppose that $p_i\in M$ is a point for which $m_i\geq3$. In this case, $p_i$ corresponds to a boundary component of the blow up~$S_P$ with $m_i-2$ marked points. Let $\gamma_i$ be a loop on~$S_P$ which is retractible to this boundary component. Under our assumption on~$(S,M)$, its image $\psi(\gamma_i)$ is essential and is therefore homotopic to a unique geodesic in~$C$. We will denote by $A_i$ the corresponding hyperbolic crown obtained by cutting along this geodesic. By the compatibility condition, the metric residue~$L_i$ equals the real part of the analytic residue of~$P_i(z_i)$. Hence, after choosing a boundary twist for~$A_i$, we can apply Theorem~\ref{thm:modelmap} to get a harmonic map $f_i:U_i\setminus\{p_i\}\rightarrow A_i\setminus\partial A_i\subset C^\circ$ whose Hopf differential has principal part $P_i(z_i)$.

We now use the model maps to construct a collection of harmonic maps. For each index~$i$ and real number $l>y_0$, let $U_i^l\subset U_i$ be the disk such that the isometry~\eqref{eqn:cusp} provides an identification 
\[
U_i^l\setminus\{p_i\}\cong\{x+\mathrm{i}y:0\leq x<1,y>l\}/\sim.
\]
If we define $S_l=S\setminus\bigcup_iU_i^l$, then these sets $S_l$ form a compact exhaustion of~$S\setminus M$, that is, a nested collection of compact sets such that $\bigcup_lS_l=S\setminus M$. Note that there is a canonical homeomorphism $S\setminus M\cong S_P\setminus(M_P\cup\partial S_P)$, and hence we may view the $S_l$ as subsets of~$S_P\setminus M_P$. After replacing~$\psi$ by a homotopic map if necessary, we may assume it agrees with~$f_i$ on $\partial U_i^l\subset\partial S_l$, and then the work of Lemaire~\cite{Lemaire82} implies that there exists a unique harmonic map $F_l:S_l\rightarrow C^\circ$ in the homotopy class of~$\psi|_{\partial S_l}$ relative to~$\partial S_l$. An important fact that we will use below is that this map $F_l$ has least energy among maps with the given boundary conditions.

\subsection{Convergence of harmonic maps}
\label{sec:ConvergenceOfHarmonicMaps}

Our next goal is to prove that there is a subsequence of the maps $F_l$ that converges to a harmonic map $F:S\setminus M\rightarrow C^\circ$. Note that the map $F_l$ depends on the phases $\theta_i$ associated to points $p_i\in M$ for which $m_i=2$. We write $F_l=F_l^\theta$ where $\theta=(\theta_i)$ to indicate the dependence on these parameters explicitly. Then we have the following bound for the energy of this map on~$S_l$.

\begin{lemma}
\label{lem:twistenergy}
There exists a constant $N_0>0$ such that 
\[
E_{S_l}(F_l^\theta)\leq E_{S_l}(F_l^0)+N_0
\]
for all~$l$.
\end{lemma}

\begin{proof}
Define $t^\theta:S_l\rightarrow S_l$ by 
\[
t^\theta(z)=
\begin{cases}
t_i(z) & \text{if $z\in U_i$ for $m_i=2$} \\
z & \text{otherwise}
\end{cases}
\]
where $t_i$ is the map defined above using the parameter $\theta_i$. Then the composition~$F_l^0\circ t^\theta$ agrees with~$F_l^\theta$ on~$\partial S_l$. Since the map $F_l^\theta$ is harmonic, we have $E_{S_l}(F_l^\theta)\leq E_{S_l}(F_l^0\circ t^\theta)$. By Lemma~4.4 of~\cite{Gupta19}, there is a constant $N_0$, independent of~$l$, such that $E_{S_l}(F_l^0\circ t^\theta)\leq E_{S_l}(F_l^0)+N_0$. Combining this with the previous inequality gives the desired bound.
\end{proof}

In the following, we will write $f_i^\theta$ for the model map $f_i:U_i\setminus\{p_i\}\rightarrow C^\circ$ that we constructed above using the data~$\theta=(\theta_j)$. Note that although we use this notation for all of the model maps for the sake of uniformity, $f_i^\theta$ is actually independent of $\theta$ unless $m_i=2$. We will also use the notation $U_i^{k,l}=U_i^k\setminus U_i^l$ whenever $l>k$.

\begin{lemma}
\label{lem:energynearpole}
There exists a constant $N_i>0$, such that 
\[
E_{U_i^{k,l}}(f_i^0)-E_{U_i^{k,l}}(F_l^\theta)\leq N_i
\]
for all~$l$.
\end{lemma}

\begin{proof}
Let us first assume that $m_i=0$. In this case, the model map $f_i^0:U_i\setminus\{p_i\}\rightarrow C^\circ$ is a biholomorphism onto its image. It follows that the energy of~$f_i^0$ on a subsurface $V\subset U_i\setminus\{p_i\}$ equals the area of $f_i^0(V)$ with respect to the hyperbolic metric on~$C^\circ$. We then obtain the desired bound 
\[
E_{U_i^{k,l}}(f_i^0)-E_{U_i^{k,l}}(F_l^\theta)\leq E_{U_i^{k,l}}(f_i^0)\leq N_i
\]
where $N_i$ is defined as the total area of $f_i^0(U_i\setminus\{p_i\})$. On the other hand, if $m_i=2$ then we have $E_{U_i^{k,l}}(f_i^0)\leq E_{U_i^{k,l}}(F_l^\theta)$ as in Lemma~3.2 of~\cite{Wolf91} so that the desired bound holds for any $N_i>0$. Finally, if $m_i\geq3$ then we define $g$ to be a map $U_i^{k,l}\rightarrow U_i^{k,l}$ having the lowest possible energy subject to the constraint $g|_{\partial U_i^l}=f_i|_{U_i^{k,l}}$. (In~\cite{Gupta19}, such a function $g$ is called a solution of the \emph{partially free Dirichlet boundary problem}.) By Lemma~4.6 of~\cite{Gupta19}, there is a constant $N_i>0$, independent of~$l$, such that 
\[
E_{U_i^{k,l}}(f_i^0)\leq E_{U_i^{k,l}}(g)+N_i.
\]
Thus we have 
\[
E_{U_i^{k,l}}(f_i^0)-E_{U_i^{k,l}}(F_l^\theta)\leq E_{U_i^{k,l}}(f_i^0)-E_{U_i^{k,l}}(g)\leq N_i
\]
as desired.
\end{proof}

\begin{lemma}
\label{lem:uniformbound}
The maps $F_l$ have energy uniformly bounded on compact sets. That is, for any compact set $K\subset S\setminus M$, there is a constant $N>0$ such that $E_K(F_l)\leq N$ for all~$l$.
\end{lemma}

\begin{proof}
Given any compact subset $K\subset S\setminus M$, we can find a real number $k>y_0$ such that $K\subset S_k$. It then suffices to show that the energies $E_{S_k}(F_l)$ for $l\geq k$ are bounded by some positive constant which is independent of~$l$. According to Lemma~\ref{lem:twistenergy}, we can find a constant $N_0>0$ independent of~$l$ so that 
\begin{align*}
E_{S_k}(F_l^\theta) &= E_{S_l}(F_l^\theta)-\sum_iE_{U_i^{k,l}}(F_l^\theta) \\
&\leq E_{S_l}(F_l^0)-\sum_iE_{U_i^{k,l}}(F_l^\theta)+N_0.
\end{align*}
Consider the map $g:S_l\rightarrow C^\circ$ given by the formula 
\[
g(z)=
\begin{cases}
f_i^0(z) & \text{if $z\in U_i^{k,l}$} \\
F_k^0(z) & \text{if $z\in S_k$}.
\end{cases}
\]
It is continuous and satisfies $g|_{\partial U_l}=F_l^0|_{\partial U_l}$. Since $F_l^0$ is harmonic, it follows that $E_{S_l}(F_l^0)\leq E_{S_l}(g)$. Combining this fact with Lemma~\ref{lem:energynearpole}, we obtain 
\begin{align*}
E_{S_k}(F_l^\theta) &\leq E_{S_k}(F_k^0)+\sum_iE_{U_i^{k,l}}(f_i^0)-\sum_iE_{U_i^{k,l}}(F_l^\theta)+N_0 \\
&\leq E_{S_k}(F_k^0)+\sum_iN_i+N_0
\end{align*}
for some constants $N_i>0$ independent of~$l$. This completes the proof.
\end{proof}

\begin{lemma}
\label{lem:convergence}
There is a subsequence $\{F_{l_j}\}$ of the maps $F_l$ which converges with respect to the $C^1$-norm on compact sets to a harmonic map $F:S\setminus M\rightarrow C^\circ$.
\end{lemma}

\begin{proof}
Consider any real number $k>y_0$. As shown in Section~1 of~\cite{SchoenYau76}, there is a constant $R>0$ independent of~$l$ such that $e(F_l)\leq R\cdot E_{S_k}(F_l)$ on~$S_k$. Combining this with Lemma~\ref{lem:uniformbound}, we see that the energy densities $e(F_l)$ are uniformly bounded on~$S_k$. The harmonicity equation~\eqref{eqn:harmonicity} applied to~$F_l$ can be written 
\begin{equation}
\label{eqn:Flharmonic}
\Delta F_l=-\frac{h_w}{h}\partial_zF_l\partial_{\bar{z}}F_l,
\end{equation}
and from what we have said, we get a $C^0$ bound for the right hand side on~$S_k$. By a standard regularity theorem for Poisson's equation (see~\cite{Jost13}, Theorem~10.1.2), the $F_l$ are uniformly bounded on~$S_k$ with respect to the $C^{1,\alpha}$-norm. This gives a $C^\alpha$ bound for the right hand side of~\eqref{eqn:Flharmonic}, and by applying the regularity theorem once again, we conclude that the $F_l$ are uniformly bounded on~$S_k$ with respect to the $C^{2,\alpha}$-norm. By the Arzel\`a-Ascoli theorem, there is a subsequence of the~$F_l$ that converges with respect to the $C^1$-norm on $S_k$.

We can now repeat the above argument, replacing $k$ by $k+1$, to get a further subsequence that converges in~$C^1$ on~$S_{k+1}$. If we continue this process inductively and let $\{F_{l_j}\}$ be a diagonal subsequence, then $\{F_{l_j}\}$ converges in~$C^1$ on compact subsets to a map $F:S\setminus M\rightarrow C^\circ$. Since each $F_{l_j}$ satisfies the Euler-Lagrange equation~\eqref{eqn:harmonicity}, so does the limit~$F$. Hence this limit map is harmonic.
\end{proof}

\subsection{Computing the principal parts}

Our next goal is to compute the principal parts of the Hopf differential of~$F$. We begin by analyzing the behavior of this map near a point $p_i\in S$ for which $m_i=0$.

\begin{lemma}
\label{lem:energysimplepole}
If $m_i=0$ then the map~$F$ has finite energy on~$U_i$.
\end{lemma}

\begin{proof}
Fix some $k>y_0$. Then we can find a smooth embedding $g:U_i\setminus\{p_i\}\rightarrow C^\circ$ such that $g|_{\partial U_i}=F|_{\partial U_i}$ and $g|_{\partial U_i^l}=f_i|_{\partial U_i^l}$ for all $l\geq k$. By~\cite{Lemaire82} there is a harmonic map $h_l:U_i^{y_0,l}\rightarrow C^\circ$ in the homotopy class of $g|_{U_i^{y_0,l}}$ relative to $\partial U_i^{y_0,l}$ which coincides with $g$ on this boundary. By the energy minimizing property of $h_l$, we have 
\[
E_{U_i^{y_0,l}}(h_l)\leq E_{U_i}(g)\leq E_{U_i^{y_0,k}}(g)+E_{U_i^k}(f_i).
\]
Since the model map $f_i$ is a biholomorphism, the energy $E_{U_i^k}(f_i)$ equals the hyperbolic area of the image $f_i(U_i^k)$. Thus there exists a constant $N>0$, independent of~$l$, such that $E_{U_i^{y_0,l}}(h_l)\leq N$. Arguing as in the proof of Lemma~\ref{lem:convergence}, we see that there is a sequence $\{h_{l_j}\}$ that converges uniformly on compact sets to a harmonic map $h:U_i\setminus\{p_i\}\rightarrow C^\circ$. We can assume that the indices $l_j$ form a subsequence of the ones appearing in the statement of Lemma~\ref{lem:convergence}. Note that, by construction, the limit map~$h$ satisfies $E_{U_i}(h)\leq N$.

Now the function $\dist(F_l,h_l)$ is subharmonic on~$U_i^{y_0,l}$, being a distance function between two harmonic maps, so it attains its maximum on~$\partial U_i^{y_0,l}$. Since $F_l$ and $h_l$ agree on~$\partial U_i^l$, this maximum must in fact lie on~$\partial U_i$. By Lemma~\ref{lem:convergence}, we have $\dist(F_{l_j},h_{l_j})\rightarrow0$ uniformly on $\partial U_i$ as $j\rightarrow\infty$. Hence $h=\lim_{j\rightarrow\infty}h_{l_j}=\lim_{j\rightarrow\infty}F_{l_j}=F$ and $E_{U_i}(F)\leq N$.
\end{proof}

\begin{lemma}
\label{lem:Hopfintegrable}
If $m_i=0$ then the Hopf differential of~$F$ is integrable near $p_i$.
\end{lemma}

\begin{proof}
Recall that the Hopf differential of~$F$ is given explicitly by $\phi(z_i)=\varphi(z_i)dz_i^{\otimes2}$ where $\varphi(z_i)=h(F(z_i))F_{z_i}\bar{F}_{z_i}$ and $h$ is the function defining the hyperbolic metric on~$C$. By the inequality of arithmetic and geometric means, one has 
\[
|\varphi(z_i)|\leq\frac{1}{2}h(F(z_i))\left(|F_{z_i}|^2+|F_{\bar{z_i}}|^2\right).
\]
Integrating both sides yields 
\[
\int_{U_i}|\varphi(z_i)|dz_id\bar{z}_i\leq\frac{E_{U_i}(F)}{2}<\infty
\]
by Lemma~\ref{lem:energysimplepole}.
\end{proof}

We can now prove the desired statement about principal parts.

\begin{lemma}
\label{lem:Hopfprincipalparts}
The Hopf differential of~$F$ has principal part $P_i(z_i)$ at~$p_i$ for every~$i$.
\end{lemma}

\begin{proof}
Let $\phi$ denote the Hopf differential of~$F$. If $m_i=0$ then we have $P_i(z_i)=0$. On the other hand, it follows from Lemma~\ref{lem:Hopfintegrable} that $\phi$ has at worst a simple pole at any $p_i\in M$. Thus the differential~$\phi$ has vanishing leading coefficient at~$p_i$ and its principal part is zero. If $m_i=2$ then by Proposition~5.5 of~\cite{Sagman19}, the Hopf differential $\phi$ has a pole of order two at~$p_i$ with leading coefficient $-\Lambda(\delta_i)L_i^2/16\pi^2$ where 
\[
\Lambda(\delta)=1-\delta^2-2i\delta
\]
and we write $\delta_i=(\sin\theta_i)/(1-\cos\theta_i)$ as above. Using the compatibility condition, one easily checks that this equals the leading coefficient of~$\phi_i=P_i(z_i)^{\otimes2}$. Hence $\phi$ has principal part $P_i(z_i)$ at~$p_i$. Finally, if $m_i\geq3$ then the desired statement is proved in Lemma~4.9 of~\cite{Gupta19}.
\end{proof}

\subsection{Uniqueness}
\label{sec:Uniqueness}

We have now constructed a harmonic map $F:S\setminus M\rightarrow C^\circ$ whose Hopf~differential has principal part $P_i(z_i)$ at the point $p_i\in M$. In the next two lemmas, we will write $F':S\setminus M\rightarrow C^\circ$ for any harmonic map in the homotopy class of~$\psi$ whose Hopf differential has principal part $P_i(z_i)$ at~$p_i$.

\begin{lemma}
\label{lem:boundeddistance}
The function $\dist(F,F')$ is bounded on~$S\setminus M$ and tends to zero near any $p_i$ for which $m_i=0$.
\end{lemma}

\begin{proof}
Let us first consider a point $p_i\in M$ for which $m_i=0$. Then by Lemma~\ref{lem:Hopfintegrable}, the Hopf~differentials of~$F$ and~$F'$ are integrable near~$p_i$. If we equip $S\setminus M$ with the uniformizing metric, then the holomorphic energy functions $\mathcal{H}(F)$ and~$\mathcal{H}(F')$ tend to unity near~$p_i$ by Proposition~3.13 of~\cite{Wolf91} while the antiholomorphic energy functions $\mathcal{L}(F)$ and~$\mathcal{L}(F')$ tend to zero near~$p_i$ as in Corollary~2 of~\cite{Lohkamp91}. Working in the coordinate $z=x+\mathrm{i}y$ provided by the isomorphism~\eqref{eqn:cusp}, we see that the Hopf~differentials of~$F$ and~$F'$ have the form $O(e^{-2\pi y})dz^{\otimes2}$ as $y\rightarrow\infty$ or equivalently as $z\rightarrow p_i$. If $h$ is the hyperbolic metric on~$C^\circ$, then it follows from equation~\eqref{eqn:pullback} that the tautological map of Riemannian manifolds $(U_i\setminus\{p_i\},F^*h)\rightarrow(U_i\setminus\{p_i\},(F')^*h)$ is approximately an isometry near~$p_i$. Equivalently, $F'\circ F^{-1}$ is approximately an isometry near the corresponding cusp of~$C^\circ$. Since an isometry of a cusp neighborhood is given by a parabolic transformation, the translation distance of such an isometry tends to zero as we approach the cusp. Hence $\dist(F,F')(z)\rightarrow0$ as $z\rightarrow p_i$.

In particular, we see that $\dist(F,F')$ is bounded on~$U_i\setminus\{p_i\}$ whenever $m_i=0$. In the case $m_i=2$, it was shown in Lemma~5.9 of~\cite{Sagman19} that $\dist(F,F')$ is also bounded on~$U_i\setminus\{p_i\}$. For $m_i\geq3$, this was shown in Proposition~3.9 of~\cite{Gupta19}. The function $\dist(F,F')$ is bounded on~$S\setminus\bigcup_iU_i$ since the latter is compact. Therefore it is bounded on all of~$S\setminus M$.
\end{proof}

\begin{lemma}
\label{lem:uniqueness}
The maps $F$ and $F'$ are equal.
\end{lemma}

\begin{proof}
On a punctured surface, any bounded subharmonic function is necessarily constant, so the distance function $\dist(F,F')$ equals some constant $k$ on~$S\setminus M$ by Lemma~\ref{lem:boundeddistance}. If this constant $k$ is nonzero, then we can construct a nonvanishing vector field on $C$ where the vector assigned to a point $F(z)$ points from $F(z)$ to~$F'(z)$. If $\{z_j\}\subset S\setminus M$ is a sequence of points converging to some $p_i\in M$ with $m_i=2$, then $F(z_j)\rightarrow w$ and $F'(z_j)\rightarrow w'$ where $w$ and~$w'$ lie on a common boundary component of~$C$ and $\dist(w,w')=k$. It follows that the vector field is tangent to those boundary components of~$C$ that are homeomorphic to~$S^1$. As explained in the proof of Corollary~3.10 of~\cite{Gupta19}, this vector field is also tangent to any boundary component of~$C$ which forms part of a hyperbolic crown, and the tangent vectors point alternately toward and away from the cusps as we travel around the crown. In particular, we cannot have any $p_i$ with $m_i\geq3$ odd. We also cannot have $p_i$ with $m_i=0$ since the distance function $\dist(F,F')$ tends to zero near such a point by Lemma~\ref{lem:boundeddistance}.

Let $\widehat{C}$ be the punctured surface obtained by doubling~$C$ along its boundary. We view $\widehat{C}$ as a closed surface with finitely many marked points obtained by filling in the punctures. This surface inherits a vector field from the one on~$C$, and this vector field has index $+1$ at each marked point. But then $\widehat{C}$ has negative Euler characteristic while the total index of the vector field is positive. This contradicts the Poincar\'e-Hopf theorem. It follows that $k=0$ and consequently $F=F'$.
\end{proof}

Thus we see that the harmonic map $F:S\setminus M\rightarrow C^\circ$ constructed above is independent of which functions~$f_i$ we choose for the model maps and also independent of which subsequence $\{F_{l_j}\}$ we take the limit of in Lemma~\ref{lem:convergence}. Combining these observations with Lemma~\ref{lem:Hopfprincipalparts}, we get a well defined map 
\[
\Phi_P^{S,M}:\mathcal{T}_P(S,M)\rightarrow\mathcal{Q}_P(S,M)
\]
taking a point $(\psi,C)\in\mathcal{T}_P(S,M)$ to the Hopf differential of the unique associated harmonic map~$F:S\setminus M\rightarrow C^\circ$.

\subsection{Regularity}

Finally, we will show that the map $F:S\setminus M\rightarrow C^\circ$ constructed above is a diffeomorphism onto its image.

\begin{lemma}
\label{lem:boundholomorphicenergy}
Assume $S\setminus M$ is equipped with the uniformizing metric, and let $\mathcal{H}(F)$ denote the holomorphic energy of~$F$ with respect to this choice of metric. Then $\mathcal{H}(F)\geq1$ on $S\setminus M$.
\end{lemma}

\begin{proof}
If $m_i=0$ then Proposition~3.13 of~\cite{Wolf91} says that $\lim_{z\rightarrow p_i}\mathcal{H}(F)(z)=1$. On the other hand, for $m_i\geq2$ we claim that $\lim_{z\rightarrow p_i}\mathcal{H}(F)(z)\geq1$. Indeed, we can choose a local coordinate~$z$ defined in a neighborhood of~$p_i$ so that the local expression for the metric is $g|dz|^2$ where $g(z)=\frac{1}{|z|^2\log^2|z|}$. Let $\phi(z)=\varphi(z)dz^{\otimes2}$ be the Hopf~differential of~$F$ in this coordinate. Since the map $F_l$ is an orientation preserving diffeomorphism, the Jacobian satisfies $\mathcal{J}(F_l)>0$ so that $\mathcal{H}(F)-\mathcal{L}(F)=\mathcal{J}(F)\geq0$ and hence $|\varphi|^2/g^2=\mathcal{L}\mathcal{H}\leq\mathcal{H}^2$. This implies the claim since $\varphi(z)$ has a pole of order $\geq2$ at~$z=0$. To complete the proof, suppose that $z_0$ is a local minimum for $\mathcal{H}(F)$ on~$S\setminus M$. Then the Bochner equation~\eqref{eqn:Bochner} implies  
\[
0\leq\left(\Delta_g\log\mathcal{H}\right)(z_0)=2\mathcal{J}(F)(z_0)-2.
\]
so that $1\leq\mathcal{J}(F)(z_0)\leq\mathcal{H}(F)(z_0)$. It follows that $\mathcal{H}(F)\geq1$ on all of~$S\setminus M$.
\end{proof}

\begin{lemma}
\label{lem:localdiffeo}
We have $\mathcal{J}(F)>0$ on~$S\setminus M$ so that $F$ is an orientation preserving local diffeomorphism.
\end{lemma}

\begin{proof}
We already know that $\mathcal{J}(F)\geq0$, and Sard's theorem implies that $\mathcal{J}(F)$ is not identically zero. If $\mathcal{J}(F)(z)=0$, then the claim on page~270 of~\cite{SchoenYau78} implies $\mathcal{H}(F)(z)=0$, which contradicts Lemma~\ref{lem:boundholomorphicenergy}. Hence the Jacobian of~$F$ is nowhere vanishing.
\end{proof}

\begin{lemma}
\label{lem:Fsurjective}
$F$ surjects onto $C\setminus\partial C$.
\end{lemma}

\begin{proof}
If $m_i=0$ then Lemma~5 of~\cite{Lohkamp91} says that the diameter of $F(\partial U_i^l)$ tends to zero as $l\rightarrow\infty$. If $m_i=2$ then Lemma~5.4 of~\cite{Sagman19} says there exists a hyperbolic isometry $R$ stabilizing the corresponding boundary component of~$C$ such that $\dist(F,R\circ f_i)(z)\rightarrow0$ as $z\rightarrow p_i$. Finally, if $m_i\geq3$, then by considering a polygonal exhaustion as in Proposition~2.29 of~\cite{Gupta19}, one sees that $F$ surjects onto the hyperbolic crown corresponding to~$p_i$. Since $F\simeq\psi$ these facts imply that $F$ is surjective onto~$C\setminus\partial C$.
\end{proof}

\begin{lemma}
\label{lem:Finjective}
$F$ is injective.
\end{lemma}

\begin{proof}
Let $y\in C\setminus\partial C$ and consider the preimage $F^{-1}(y)$. We claim that this preimage is contained in a compact set in~$S\setminus M$. Indeed, if it is not contained in such a set, then we can find a sequence $\{x_i\}\subset F^{-1}(y)$ such that $\lim_{i\rightarrow\infty}x_i\in M$. But then the points $F(x_i)$ tend to a cusp or boundary component of~$C$ as~$i\rightarrow\infty$, and we can bound $F(x_i)$ away from~$y$ by choosing $i$ sufficiently large. This contradiction establishes the claim.

Now since $y_0$ is a regular value by Lemma~\ref{lem:localdiffeo}, the preimage $F^{-1}(y)$ is a compact zero-dimensional manifold, hence finite. Since $\psi$ is a diffeomorphism and $F\simeq\psi$, we know that $F$ has degree one. Therefore 
\[
1=\sum_{x\in F^{-1}(y)}\sgn\mathcal{J}(F)(x).
\]
From Lemma~\ref{lem:localdiffeo}, we know that $\mathcal{J}(F)(x)>0$ for all $x\in F^{-1}(y)$, and hence $|F^{-1}(y)|=1$.
\end{proof}

\begin{proposition}
$F$ is a diffeomorphism $S\setminus M\stackrel{\sim}{\longrightarrow}C\setminus\partial C$.
\end{proposition}

\begin{proof}
This follows from Lemmas~\ref{lem:localdiffeo}, \ref{lem:Fsurjective}, and~\ref{lem:Finjective}.
\end{proof}

\section{A map between moduli spaces}
\label{sec:AMapBetweenModuliSpaces}

In this section, we use our results on harmonic maps to construct a natural map from the moduli space of signed quadratic differentials to the enhanced Teichm\"uller space.

\subsection{A homeomorphism of moduli spaces}
\label{sec:AHomeomorphismOfModuliSpaces}

As in the previous section, we consider a compact Riemann surface~$S$ and a nonempty finite subset $M=\{p_1,\dots,p_d\}\subset S$. We choose a local coordinate~$z_i$ defined on a disk around~$p_i$ and set $P=(P_i(z_i))$ where $P_i(z_i)$ is a principal differential in this local coordinate. We write $m_i$ for the order of the pole of the quadratic differential $\phi_i(z_i)=P_i(z_i)^{\otimes2}$ and write $\theta_i$ for the argument of its leading coefficient when $m_i=2$.

Assuming $\chi(S\setminus M)<0$ and $\theta_i\in(0,2\pi)$, we have constructed a well defined map $\Psi_P^{S,M}:\mathcal{T}_P(S,M)\rightarrow\mathcal{Q}_P(S,M)$. Here we will show that it is a homeomorphism, following the approach of~\cite{Wolf89}.

\begin{lemma}
\label{lem:injective}
$\Phi_P^{S,M}$ is injective.
\end{lemma}

\begin{proof}
Suppose we have $\Psi_P^{S,M}(C_1,\psi_1)=\Psi_P^{S,M}(C_2,\psi_2)=\phi$. Then by construction $\phi$ arises as the Hopf differential of harmonic maps $F_1,F_2:S\setminus M\rightarrow C^\circ$ constructed from~$\psi_1$ and~$\psi_2$, respectively. These maps satisfy the Bochner equation~\eqref{eqn:Bochner}, namely 
\begin{equation}
\label{eqn:modifiedBochner}
\Delta w_k=e^{2w_k}-e^{-2w_k}|\varphi|^2
\end{equation}
where we have written $w_k=\frac{1}{2}\log\mathcal{H}(F_k)$ and $\phi(z)=\varphi(z)dz^{\otimes2}$, and by applying a conformal transformation we may assume the metric on~$S\setminus M$ has the expression $|dz|^2$ in the coordinate~$z$ and the Laplace-Beltrami operator in~\eqref{eqn:Bochner} specializes to the ordinary Laplacian.

If we define $\widetilde{w}_k=w_k-\frac{1}{2}\log|\varphi|=\frac{1}{2}\log\frac{\mathcal{H}(F_k)}{|\varphi|}$, then by the result in Section~5 of~\cite{Han96} (see also equation~(8) of~\cite{Gupta19}), we have the bound 
\[
0\leq\widetilde{w}_k(z)\leq e^{-CR(z)}
\]
where $C>0$ and $R(z)$ is the distance from $z$ to the nearest zero of~$\phi$, measured using the flat metric associated to~$\phi$. Exponentiating these inequalities, we obtain 
\[
1\leq\frac{\mathcal{H}(F_k)(z)^{1/2}}{|\varphi(z)|^{1/2}}\leq\exp(e^{-CR(z)}).
\]
When $m_i\geq2$, it follows that $\mathcal{H}(F_k)(z)/|\varphi(z)|\rightarrow1$ and hence $\mathcal{H}(F_1)(z)/\mathcal{H}(F_2)(z)\rightarrow1$ as $z\rightarrow p_i$. Proposition~3.13 of~\cite{Wolf91} implies that the same is true when $m_i=0$. Thus, for any value of~$m_i$, we see that $w_1(z)-w_2(z)\rightarrow0$, as $z\rightarrow p_i$.

If we have $w_1>w_2$ at some point of $S\setminus M$, then it follows from what we have said that there exists a local maximum of the difference $w_1-w_2$. But then at this local maximum point the Bochner equation~\eqref{eqn:modifiedBochner} implies 
\[
0\geq\Delta(w_1-w_2)=(e^{2w_1}-e^{2w_2})-|\varphi|^2(e^{-2w_1}-e^{-2w_2})>0,
\]
which is a contradiction. If $w_2>w_1$ then we reach a similar contradiction by interchanging the roles of~$w_1$ and $w_2$. We must therefore have $w_1\equiv w_2$ and $\mathcal{H}_1\equiv\mathcal{H}_2$ on~$S\setminus M$. Since $\mathcal{L}(F_k)=|\varphi|^2/\mathcal{H}(F_k)$, we have $e(F_1)\equiv e(F_2)$ on~$S\setminus M$ as well. Then from~\eqref{eqn:pullback}, we see that the hyperbolic metrics on~$C_1$ and~$C_2$ pull back via~$F_1$ and~$F_2$ to give the same metric on~$S\setminus M$. Hence $(C_1,\psi_1)$ and $(C_2,\psi_2)$ represent the same point of $\mathcal{T}_P(S,M)$.
\end{proof}

\begin{lemma}
\label{lem:continuous}
$\Phi_P^{S,M}$ is continuous.
\end{lemma}

\begin{proof}
Suppose $\{\rho_i\}_{i\geq1}\subset\mathcal{T}_P(S,M)$ is a sequence of points that converge to a point $\rho_\infty\in\mathcal{T}_P(S,M)$. For $i=1,\dots,\infty$, we can represent $\rho_i$ as a pair $(C_i,\psi_i)$ where $C_i^\circ$ is the surface $S\setminus M$ equipped with a metric tensor $h_i$ and $\psi_i$ is the identity map. We can further assume that the sequence of metric tensors $h_i$ converges to $h_\infty$ as $i\rightarrow\infty$.

Let $F_i:S\setminus M\rightarrow C_i^\circ$ be the harmonic map associated to~$(C_i,\psi_i)$ as in Section~\ref{sec:HarmonicMaps}. Recall that it is constructed as a limit of harmonic maps $F_{i,l}:S_l\rightarrow C^\circ$ defined on sets $S_l$ forming a compact exhaustion of $S\setminus M$. We can assume the boundary conditions for the maps $F_{i,l}$ on~$\partial S_l$ converge as $i\rightarrow\infty$ so that the results in Sections~3 and~4 of~\cite{EellsLemaire81} imply that $F_{i,l}\rightarrow F_{\infty,l}$ with respect to the~$C^1$-norm as $i\rightarrow\infty$. Combining this observation with Lemma~\ref{lem:uniformbound}, we see that for any fixed compact set $K$ there is a constant $N>0$ such that $E_K(F_{i,l})<N$ for all~$i$ and all~$l$.

Let $\{\varepsilon_i\}_{i\geq1}$ be a strictly decreasing sequence of positive real numbers such that $\varepsilon_i\rightarrow0$ as $i\rightarrow\infty$. Then by Lemma~\ref{lem:convergence}, we can find an increasing sequence of numbers $l_i$ such that $l_i\rightarrow\infty$ and $\dist(F_{i,l_i},F_i)<\varepsilon_i$ on~$S_{l_i}$. Since the maps $G_i=F_{i,l_i}$ have energy uniformly bounded on compact sets, the argument of Lemma~\ref{lem:convergence} implies that there is a subsequence of indices $i_k$ such that $G_{i_k}$ converges with respect to the $C^1$-norm on compact sets to some harmonic map~$G$. By our uniqueness results, we have $G=\lim_{i\rightarrow\infty}G_i$ without passing to a subsequence, and $G=F_\infty$. Therefore $\lim_{i\rightarrow\infty}F_i=\lim_{i\rightarrow\infty}F_{i,l_i}=F_\infty$. Since the $F_i$ converge in~$C^1$ on compact sets, their Hopf~differentials converge to the Hopf~differential of~$F_\infty$. That is, $\Phi_P^{S,M}(\rho_i)\rightarrow\Phi_P^{S,M}(\rho_\infty)$.
\end{proof}

\begin{lemma}
\label{lem:proper}
$\Phi_P^{S,M}$ is proper.
\end{lemma}

\begin{proof}
Let $K\subset\mathcal{Q}_P(S,M)$ be a compact set in the space of quadratic differentials, and let us consider its preimage $\Phi^{-1}(K)\subset\mathcal{T}_P(S,M)$, where to simplify notation we have written $\Phi=\Phi_P^{S,M}$. If $\{\rho_i\}_{i\geq1}\subset\Phi^{-1}(K)$ is any sequence of points, then $\Phi(\rho_i)\in K$ for $i\geq1$ so that after passing to a subsequence we may assume the $\Phi(\rho_i)$ converge to a point of~$K$.

Let $S_l\subset S\setminus M$ be the sets forming the compact exhaustion from Section~\ref{sec:ConstructionOfHarmonicMaps}. For any~$l$, let us write $\|\phi\|_{S_l}$ for the total area of~$S_l$ with respect to the flat metric determined by a quadratic differential~$\phi$ on~$S\setminus M$. Let $(C_i,\psi_i)$ be a marked hyperbolic surface representing~$\rho_i$ where $C_i$ is $S\setminus M$ equipped with some hyperbolic metric~$h_i$ and $\psi_i$ is the identity map. Let $F_i:S\setminus M\rightarrow C_i$ be the associated harmonic map. Then the calculation in Lemma~3.2 of~\cite{Wolf89} shows that 
\[
E_{S_l}(F_i)\leq2\|\Phi(\rho_i)\|_{S_l}+A(F_i(S_l))
\]
for every~$i$ where $A$ is the area function for the hyperbolic metric on~$C_i$. Note that we have $A(F_i(S_l))\leq A(C_i)<\infty$. Since the $\Phi(\rho_i)$ converge, it follows that there exists a constant $N>0$, independent of~$i$, such that $E_{S_l}(F_i)\leq N$.

Arguing as in the proof of Lemma~\ref{lem:convergence}, we see that there is a subsequence $\{F_{i_k}\}$ of the $F_i$ which converges with respect to the $C^1$-norm on compact subsets of $S\setminus M$. Then the energy densities~$e(F_i)$, and by~\eqref{eqn:pullback} the pullback metrics $F_i^*(h_i)$, converge uniformly on compact sets. From this we get convergence of the subsequence $\{\rho_{i_k}\}$. This proves that $\Phi^{-1}(K)$ is sequentially compact, hence compact.
\end{proof}

\begin{proposition}
$\Phi_P^{S,M}$ is a homeomorphism.
\end{proposition}

\begin{proof}
From Lemmas~\ref{lem:dimfixedprincipalpart} and~\ref{lem:dimfixedboundarylength}, we know that $\mathcal{Q}_P(S,M)$ and $\mathcal{T}_P(S,M)$ are homeomorphic to Euclidean spaces of the same dimension. The proposition therefore follows from Lemmas~\ref{lem:injective}, \ref{lem:continuous}, \ref{lem:proper} and the invariance of domain.
\end{proof}

\subsection{Exceptional surfaces}

In the above discussion, we assumed that the Riemann surface~$S$ and set $M\subset S$ satisfy $\chi(S\setminus M)<0$. On the other hand, if $\chi(S\setminus M)\geq0$ then $S$ is isomorphic to the Riemann sphere and $M$ consists of one or two marked points on~$S$. We will now explain how our arguments can be extended to include these cases. Since the proofs are similar to those in previous sections, some details are left to the reader.

First suppose that $M=\{p_1\}$ consists of a single marked point. In this case, we may assume $m_1\geq3$ since otherwise $\mathcal{Q}_P(S,M)$ and $\mathcal{T}_P(S,M)$ are both empty. Then the associated marked bordered surface $(S_P,M_P)$ is a disk with $m_1-2$ marked points on its boundary. It was shown in Proposition~3.12 of~\cite{Gupta19} that for any $(C,\psi)\in\mathcal{T}_P(S,M)$, there exists a unique harmonic map $F:S\setminus M\rightarrow C$ homotopic to~$\psi$ whose Hopf~differential lies in~$\mathcal{Q}_P(S,M)$. Moreover, the map taking $(C,\psi)$ to the Hopf differential of this harmonic map~$F$ is a homeomorphism $\Phi_P^{S,M}:\mathcal{T}_P(S,M)\rightarrow\mathcal{Q}_P(S,M)$ as before.

Next suppose that $M=\{p_1,p_2\}$ consists of two marked points. In this case, we must have $m_1\geq3$ or $m_2\geq3$ since otherwise $\mathcal{T}_P(S,M)$ is empty. If we have both $m_1\geq3$ and $m_2\geq3$, then $(S_P,M_P)$ is an annulus with $m_1-2$ marked points on one boundary component and $m_2-2$ marked points on the other boundary component. In such examples, the construction of the homeomorphism $\Phi_P^{S,M}$ proceeds exactly as above.

It remains to consider the cases where $m_1\geq3$ and $m_2\leq2$. In these cases, the marked bordered surface $(S_P,M_P)$ is a once-punctured disk with $m_1-2$ marked points on its boundary. If $m_2=2$ then for any $(C,\psi)\in\mathcal{T}_P(S,M)$, the compatibility condition implies that $C$ is a hyperbolic crown. Then exactly as above, we can then construct a harmonic map $F:S\setminus M\rightarrow C^\circ$ whose Hopf~differential lies in~$\mathcal{Q}_P(S,M)$. On the other hand, if $m_2=0$ then more care is needed. Indeed, if $(C,\psi)\in\mathcal{T}_P(S,M)$, then $p_2$ corresponds to a cusp in~$C$, and this surface $C$ has no closed geodesics. In particular, we cannot apply Theorem~\ref{thm:modelmap} to construct a model map into~$C$.

To get around this difficulty, first note that if $F:S\setminus M\rightarrow C$ is a harmonic map homotopic to~$\psi$ with Hopf differential in~$\mathcal{Q}_P(S,M)$, then $F$ is the unique map with these properties by the arguments of Section~\ref{sec:Uniqueness}. Let us choose a sequence of principal differentials $P_{2,i}(z_2)=\pm\sqrt{a_i}\,dz_2/z_2$ with $a_i\in\mathbb{C}^*$ for all~$i$ such that $a_i\rightarrow0$ as~$i\rightarrow\infty$ and define $P_i=(P_1(z_1),P_{2,i}(z_2))$. Then $(S_P,M_P)$ and the $(S_{P_i},M_{P_i})$ are all equal to a common marked bordered surface~$(\mathbb{S},\mathbb{M})$. Suppose $\{\rho_i\}\subset\mathcal{T}(\mathbb{S},\mathbb{M})$ is a sequence of points $\rho_i\in\mathcal{T}_{P_i}(S,M)$ that converges to some $(C,\psi)\in\mathcal{T}_P(S,M)$, and let $F_i$ be the harmonic map associated to~$\rho_i$. As in the proof of Lemma~\ref{lem:continuous}, the maps $F_i$ converge to a unique harmonic map $S\setminus M\rightarrow C$ in the homotopy class of~$\psi$ with Hopf differential in $\mathcal{Q}_P(S,M)$. One then shows as in Section~\ref{sec:AHomeomorphismOfModuliSpaces} that there is a homeomorphism $\Phi_P^{S,M}:\mathcal{T}_P(S,M)\rightarrow\mathcal{Q}_P(S,M)$ taking $(C,\psi)$ to the Hopf differential of this harmonic map.

\subsection{The inverse homeomorphism}
 
Up until now, we have assumed that the phase~$\theta_i$ associated to a point $p_i\in M$ with $m_i=2$ lies in~$(0,2\pi)$. Equivalently, we have assumed that the leading coefficient of a differential in $\mathcal{Q}_P(S,M)$ at such a point~$p_i$ is not a positive real number. We will now show that this assumption can be dropped, and in doing so, establish Theorem~\ref{thm:introgeneralizeWolf} from the introduction.

\begin{theorem}
\label{thm:homeo}
Let $P=(P_i(z_i))$ be a tuple of principal differentials with no restriction on the phases of the leading coefficients. Then there exists a unique homeomorphism 
\[
\Psi_P^{S,M}:\mathcal{Q}_P(S,M)\rightarrow\mathcal{T}_P(S,M)
\]
such that the image of a differential $\phi$ is represented by a pair $(C,\psi)$ where $\psi:S\setminus M\rightarrow C\setminus\partial C$ is a harmonic map with Hopf differential~$\phi$.
\end{theorem}

\begin{proof}
As usual we will write $m_i$ for the order of the differential $\phi_i(z_i)=P_i(z_i)^{\otimes2}$ at the point~$p_i\in M$. We write $a_i$ for the leading coefficient of~$\phi_i$ when $m_i=2$ and write $\theta_i=\arg(a_i)\in[0,2\pi)$. Choose a real number $\theta\in\mathbb{R}$ so that $e^{2\mathrm{i}\theta}\cdot a_i\not\in\mathbb{R}_{>0}$ for all $p_i\in M$ with $m_i=2$ and define $P(\theta)=(P_i^\theta(z_i))$ where $P_i^\theta(z_i)=e^{\mathrm{i}\theta}\cdot P_i(z_i)$. Then our earlier results imply that there exists a homeomorphism 
\[
\Psi_{P(\theta)}^{S,M}:\mathcal{Q}_{P(\theta)}(S,M)\rightarrow\mathcal{T}_{P(\theta)}(S,M)
\]
defined as the inverse of $\Phi_{P(\theta)}^{S,M}$. Given a differential $\phi\in\mathcal{Q}_P(S,M)$, we can apply this homeomorphism $\Psi_{P(\theta)}^{S,M}$ to the rotated differential $e^{2\mathrm{i}\theta}\cdot\phi\in\mathcal{Q}_{P(\theta)}(S,M)$ to get an equivalence class of marked hyperbolic surfaces. This equivalence class can be represented by a pair $(C_\theta,\psi_\theta)$ where $C_\theta\setminus\partial C_\theta$ is identified with the surface $S\setminus M$ equipped with an auxiliary metric~$h_\theta$, and $\psi_\theta$ is the identity map. According to~\eqref{eqn:pullback}, the metric $h_\theta$ can be written in a local coordinate~$z$ as 
\[
e^{2\mathrm{i}\theta}\cdot\varphi(z)dz^{\otimes2}+g(z)e(\mathrm{id})dz\otimes d\bar{z}+e^{-2\mathrm{i}\theta}\cdot\overline{\varphi(z)}d\bar{z}^{\otimes2}
\]
where $\phi(z)=\varphi(z)dz^{\otimes2}$. We can define a modified metric $h$ on~$S\setminus M$ given in the local coordinate by 
\[
\varphi(z)dz^{\otimes2}+g(z)e(\mathrm{id})dz\otimes d\bar{z}+\overline{\varphi(z)}d\bar{z}^{\otimes2}.
\]
Since $\varphi(z)$ is holomorphic and the Jacobian $\mathcal{J}(\mathrm{id})$ of the identity map is strictly positive, we see that the identity map is a harmonic map of Riemannian manifolds $(S\setminus M,g)\rightarrow(S\setminus M,h)$ with Hopf differential $\phi$. By equipping the surface $S\setminus M$ with the metric~$h$ and defining $\psi=\mathrm{id}$, we get a point $\Psi_P^{S,M}(\phi)=(C,\psi)\in\mathcal{T}(S,M)$. It follows from the uniqueness of $\Psi_{P(\theta)}^{S,M}$ that this point is independent of the choice of~$\theta$. Since $h$ is the limit of the metrics $h_\theta$ as $\theta\rightarrow0$, we see that $C$ has length data compatible with~$P$. Hence $\Psi_{P}^{S,M}(\phi)\in\mathcal{T}_P(S,M)$. Since the map $\Psi_{P(\theta)}^{S,M}$ is unique and a homeomorphism, it follows that the resulting map $\Psi_P^{S,M}:\mathcal{Q}_P(S,M)\rightarrow\mathcal{T}_P(S,M)$ has these same properties.
\end{proof}

\subsection{Incorporating the markings}

Let $(\mathbb{S},\mathbb{M})$ be a marked bordered surface, and if $g(\mathbb{S})=0$ assume that $|\mathbb{M}|\geq3$. Recall that the moduli space $\mathcal{Q}(\mathbb{S},\mathbb{M})$ parametrizes triples $(S,\phi,\theta)$ where $\phi$ is a GMN~differential on~$S$ and $\theta$ is a marking of~$(S,\phi)$ by~$(\mathbb{S},\mathbb{M})$. Suppose we are given such a triple $(S,\phi,\theta)$. Let $M=\{p_1,\dots,p_d\}\subset S$ be the set of poles of~$\phi$ and choose a local coordinate $z_i$ defined on a disk around~$p_i$. We will write $P=(P_i(z_i))$ for a tuple of principal differentials where $P_i(z_i)$ is the principal part of~$\phi$ at~$p_i$. Then the point $\Psi_P^{S,M}(\phi)\in\mathcal{T}_P(S,M)$ is represented by a pair~$(C,\psi)$, and the composition 
\[
\mathbb{S}\setminus\mathbb{M}\xrightarrow{\theta|_{\mathbb{S}\setminus\mathbb{M}}}S_P\setminus M_P\stackrel{\psi}{\longrightarrow}C^\circ
\]
determines a marking of $C$ by~$(\mathbb{S},\mathbb{M})$. Hence we have a marked hyperbolic surface whose equivalence class defines a point $\Psi(S,\phi,\theta)\in\mathcal{T}(\mathbb{S},\mathbb{M})$.

\begin{proposition}
This construction gives a well defined $\MCG(\mathbb{S},\mathbb{M})$-equivariant map 
\[
\Psi:\mathcal{Q}(\mathbb{S},\mathbb{M})\rightarrow\mathcal{T}(\mathbb{S},\mathbb{M}).
\]
\end{proposition}

\begin{proof}
We first note that the above construction is independent of the choice of local coordinate~$z_i$. Indeed, if we choose a different local coordinate~$z_i'$ near each~$p_i$, then we get a tuple $P'=(P_i'(z_i'))$ of principal parts for~$\phi$. This $P'$ determines the same compatible length data as~$P$, and so we have $\mathcal{T}_{P'}(S,M)=\mathcal{T}_P(S,M)$. The proof of Lemma~\ref{lem:injective} shows that a point in $\mathcal{T}_P(S,M)$ represented by a hyperbolic surface with harmonic marking is uniquely determined by the Hopf~differential of the marking. Hence we must have $\Psi_{P'}^{S,M}(\phi)=\Psi_P^{S,M}(\phi)$.

It remains to check that the construction is independent of the choice of representative for a point of $\mathcal{Q}(\mathbb{S},\mathbb{M})$. Suppose that $(S_1,\phi_1,\theta_1)$ and $(S_2,\phi_2,\theta_2)$ represent the same point of $\mathcal{Q}(\mathbb{S},\mathbb{M})$. For $k=1,2$, let $M_k\subset S_k$ be the set of poles of~$\phi_k$, and let $P_k$ be a tuple of principal differentials encoding the principal parts of~$\phi_k$. Then $\Psi_{P_k}^{S_k,M_k}(\phi_k)$ is a point of $\mathcal{T}_{P_k}(S_k,M_k)$ represented by a pair $(C_k,\psi_k)$. From the definition of equivalence of marked quadratic differentials, we get a biholomorphism $f:S_1\rightarrow S_2$ that preserves the $\phi_k$ and commutes with the markings~$\theta_k$. This map $f$ induces an diffeomorphism $(S_1)_{P_1}\setminus(M_1)_{P_1}\rightarrow (S_2)_{P_2}\setminus(M_2)_{P_2}$, also denoted~$f$, and we claim that the pairs $(C_1,\psi_1)$ and $(C_2,\psi_2\circ f)$ represent the same point of~$\mathcal{T}_{P_1}(S_1,M_1)$.

Recall that $S_k\setminus M_k$ can be considered as a subsurface of $(S_k)_{P_k}\setminus(M_k)_{P_k}$. By Theorem~\ref{thm:homeo}, we can assume that the map $\psi_k$ is harmonic on $S_k\setminus M_k$. Since $f$ is conformal, it follows that $\psi_2\circ f$ is harmonic as well. Let us write $\alpha^{(p,q)}$ for the $(p,q)$-part of a differential form~$\alpha$. Since $f$ is holomorphic, the pullback along~$f$ of a $(p,q)$-form on~$S_2$ is a $(p,q)$-form on~$S_1$. It follows that if $h$ is the hyperbolic metric on~$C_2$, then 
\[
\left(f^*(\psi_2^*h)\right)^{(2,0)}=f^*\left((\psi_2^*h)^{(2,0)}\right)=f^*\phi_2=\phi_1.
\]
This means the maps $\psi_2\circ f$ and $\psi_1$ have the same Hopf differential. The claim now follows from the proof of Lemma~\ref{lem:injective}.

Thus we see that there exists an isometry $g:C_1\rightarrow C_2$ such that $\psi_2\circ f$ and $g\circ\psi_1$ are homotopic. By a simple diagram chase, one sees that the compositions $\psi_2\circ\theta_2$ and $g\circ\psi_1\circ\theta_1$ are homotopic. Hence we have two equivalent representatives for a point in $\mathcal{T}(\mathbb{S},\mathbb{M})$, and the map $\Psi$ is well defined. It follows immediately from the way we have defined the $\MCG(\mathbb{S},\mathbb{M})$-actions that this map is $\MCG(\mathbb{S},\mathbb{M})$-equivariant.
\end{proof}

\subsection{A lift to covering spaces}
\label{sec:ALiftToCoveringSpaces}

Now suppose we are given a point of the moduli space $\mathcal{Q}^\pm(\mathbb{S},\mathbb{M})$. We can think of this as a point $(S,\phi,\theta)$ in $\mathcal{Q}(\mathbb{S},\mathbb{M})$ together with a choice of $r_p\in\mathbb{C}$ for each~$p\in\mathbb{P}$ of~$\phi$ satisfying 
\[
a_p=(r_p/4\pi i)^2
\]
where $a_p$ is the leading coefficient of~$\phi$ at~$\theta(p)$. Let $(C,\psi)$ be a marked hyperbolic surface representing the point $\Psi(S,\phi,\theta)\in\mathcal{T}(\mathbb{S},\mathbb{M})$. From the compatibility condition, one sees that if $\Re(r_p)=0$ then $p$ corresponds via the map $\psi$ to a cusp of~$C$. Otherwise it corresponds to a boundary component of~$C$. We will choose an orientation of the boundary component to agree with the orientation of~$C$ if $\Re(r_p)>0$, and we choose the opposite orientation of the boundary component if~$\Re(r_p)<0$. This defines a point of $\mathcal{T}^\pm(\mathbb{S},\mathbb{M})$, and hence we have the following.

\begin{proposition}
\label{prop:lift}
This construction defines a lift of~$\Psi$ to an $\MCG^\pm(\mathbb{S},\mathbb{M})$-equivariant map 
\[
\Psi^\pm:\mathcal{Q}^\pm(\mathbb{S},\mathbb{M})\rightarrow\mathcal{T}^\pm(\mathbb{S},\mathbb{M})
\]
of the covering spaces.
\end{proposition}

Our goal in the remainder of this section is to prove that the map~$\Psi^\pm$ of Proposition~\ref{prop:lift} is continuous.

\subsection{Fixing the marking}

Let us consider a fixed compact Riemann surface~$S$, a finite subset $M=\{p_1,\dots,p_d\}\subset S$, and positive integers $m_1,\dots,m_d\in\mathbb{Z}_{>0}$. Let $\mathbb{S}$ be the smooth surface obtained by taking an oriented real blowup of~$S$ at each $p_i$ for which~$m_i\geq3$, and let $\mathbb{M}\subset\mathbb{S}$ be a set consisting of $m_i-2$ marked points on the resulting boundary component of~$\mathbb{S}$, together with those points $p_i$ for which $m_i\leq2$. Let $D=\sum_im_ip_i$ and write 
\[
\mathcal{N}=\mathcal{N}(\mathbb{S},\mathbb{M})\subset H^0(S,\omega_S^{\otimes2}(D))
\]
for the set of quadratic differentials whose associated marked bordered surface equals the fixed surface~$(\mathbb{S},\mathbb{M})$. Thus the asymptotic horizontal directions associated to a differential~$\phi$ remain constant as we vary~$\phi$ within~$\mathcal{N}$.

Fix a choice of local coordinate $z_i$ defined in a neighborhood of~$p_i\in S$. Then for any $\phi\in\mathcal{N}$, there is a tuple $P_\phi=(P_{\phi,i}(z_i))$ of principal differentials giving the principal parts of~$\phi$ in these local coordinates. We can view $\Psi_{P_\phi}^{S,M}(\phi)$ as a point of $\mathcal{T}(\mathbb{S},\mathbb{M})$, and hence there is a map 
\[
\Psi_{\mathcal{N}}:\mathcal{N}\rightarrow\mathcal{T}(\mathbb{S},\mathbb{M})
\]
given by $\phi\mapsto\Psi_{P_\phi}^{S,M}(\phi)$.

\begin{lemma}
\label{lem:continuityfixedcomplexstructure}
This map $\Psi_{\mathcal{N}}$ is continuous.
\end{lemma}

\begin{proof}
For each $\rho\in\mathcal{T}(\mathbb{S},\mathbb{M})$, choose a tuple $P_\rho=(P_{\rho,i}(z_i))$ of principal differentials so that $P_\rho$ depends continuously on~$\rho$ and a quadratic differential in $\mathcal{Q}_{P_\rho}(S,M)$ has associated marked bordered surface $(\mathbb{S},\mathbb{M})$. Then the map $\Phi_{\mathcal{N},P}:\mathcal{T}(\mathbb{S},\mathbb{M})\rightarrow\mathcal{N}$ given by $\rho\mapsto(\Psi_{P_\rho}^{S,M})^{-1}(\rho)$ is injective and continuous by the proofs of Lemmas~\ref{lem:injective} and~\ref{lem:continuous}, respectively. Since $\mathcal{T}(\mathbb{S},\mathbb{M})$ is locally compact and $\mathcal{N}$ is Hausdorff, $\Phi_{\mathcal{N},P}$ is locally a homeomorphism onto its image.

By construction the map $\Psi_{\mathcal{N}}$ is a left inverse for $\Phi_{\mathcal{N},P}$, so it is continuous when restricted to the image of~$\Phi_{\mathcal{N},P}$. Suppose $\phi\in\mathcal{N}$ is any point and $\{\phi_j\}\subset\mathcal{N}$ is a sequence of points converging to~$\phi$. We can choose the $P_\rho$ so that $P_{\rho_j}=P_{\phi_j}$ for $\rho_j=\Psi_{\mathcal{N}}(\phi_j)$. Then it follows that $\Psi_{\mathcal{N}}(\phi_j)\rightarrow\Psi_{\mathcal{N}}(\phi)$ and hence $\Psi_{\mathcal{N}}$ is continuous at~$\phi$.
\end{proof}

\subsection{Proof of continuity}

Using Lemma~\ref{lem:continuityfixedcomplexstructure}, we can prove the continuity of the maps~$\Psi$ and~$\Psi^\pm$ constructed previously. This establishes Theorem~\ref{thm:introPsipm} from the introduction.

\begin{proposition}
\label{prop:Psicontinuous}
The map $\Psi:\mathcal{Q}(\mathbb{S},\mathbb{M})\rightarrow\mathcal{T}(\mathbb{S},\mathbb{M})$ is continuous.
\end{proposition}

\begin{proof}
Consider a point $q\in\mathcal{Q}(\mathbb{S},\mathbb{M})$ which is represented by a marked quadratic differential $(S,\phi,\theta)$, and write $M\subset S$ for the set of poles of~$\phi$. Let $B=\mathcal{T}(g,d)$ be the Teichm\"uller space parametrizing marked Riemann surfaces of genus~$g$ with $d=|M|$ marked points. Then the pair $(S,M)$ can be considered as a point $b\in B$. If $\pi:X\rightarrow B$ is the universal curve, then $S$ is the fiber of~$\pi$ over~$b$. Let $b\in U\subset B$ be a neighborhood of~$b$ over which $\pi$ is trivial. Then there exists a local trivialization 
\[
s_U:\pi^{-1}(U)\rightarrow U\times S
\]
which provides an isomorphism of the fiber over any $b'\in U$ with~$S$. Let $q'\in\mathcal{Q}(\mathbb{S},\mathbb{M})$ represented by a marked quadratic differential $(S',\phi',\theta')$, and write $M'\subset S'$ for the set of poles of~$\phi'$. By choosing~$q'$ sufficiently close to~$q$ in the topology of the moduli space, we can assume that $b'=(S',M')$ defines a point in~$U$ and that $s_U|_{\pi^{-1}(b')}\circ\theta'$ and $\theta$ are isotopic. Then $(S',\phi',\theta')$ is equivalent to a marked quadratic differential with marking~$\theta$. Lemma~\ref{lem:continuityfixedcomplexstructure} shows that $\Psi$ is continuous when restricted to a subspace of $\mathcal{Q}(\mathbb{S},\mathbb{M})$ parametrizing quadratic differentials on a fixed Riemann surface with a fixed marking. Thus, by choosing $q'$ sufficiently close to~$q$, we can ensure that the points $\Psi(q)$ and~$\Psi(q')$ are close. It follows that $\Psi$ is continuous at~$q$.
\end{proof}

\begin{proposition}
\label{prop:Psipmcontinuous}
The map $\Psi^\pm:\mathcal{Q}^\pm(\mathbb{S},\mathbb{M})\rightarrow\mathcal{T}^\pm(\mathbb{S},\mathbb{M})$ is continuous.
\end{proposition}

\begin{proof}
This follows immediately from Proposition~\ref{prop:Psicontinuous} and the fact that $\Psi^\pm$ is a lift of~$\Psi$ to a map of the covering spaces $\mathcal{Q}^\pm(\mathbb{S},\mathbb{M})$ and $\mathcal{T}^\pm(\mathbb{S},\mathbb{M})$.
\end{proof}

\section{Asymptotic property}
\label{sec:AsymptoticProperty}

In this section, we prove an asymptotic property of the map from signed differentials to the enhanced Teichm\"uller space.

\subsection{Images of horizontal and vertical paths}

We have seen in Section~\ref{sec:HorizontalStripDecomposition} that a quadratic differential $\phi$ defined on a Riemann surface~$S$ naturally determines a flat Riemannian metric on $S\setminus\Crit(\phi)$. This defines a metric space structure on the complement of the poles of~$\phi$, and it is a well known fact that this metric space is complete if and only if $\phi$ is complete in the sense of Definition~\ref{def:GMNdifferential}.

We have also seen that $\phi$ determines a foliation of~$S\setminus\Crit(\phi)$ whose leaves are mapped by the distinguished local coordinates to horizontal lines in~$\mathbb{C}$. In what follows, we will say that a smooth path in $S\setminus\Crit(\phi)$ is \emph{horizontal} if its image under any distinguished local coordinate has constant imaginary part. We will say that a smooth path in $S\setminus\Crit(\phi)$ is \emph{vertical} if its image under any distinguished local coordinate has constant real part.

To prove the main result of this section, we will make use of the following statements. The proofs are sketched in Section~2.5 of~\cite{Gupta19}, and further details can be found in the references cited there.

\begin{proposition}[\cite{Gupta19}, Section~2.5]
\label{prop:horizontalverticallength}
Let $\phi$ be a complete meromorphic quadratic differential on~$S$, let $M\subset S$ be the set of poles of~$\phi$, and let $C$ be a surface equipped with a hyperbolic metric. Let $\alpha_h$ be a horizontal path and $\alpha_v$ a vertical path in~$S$, each of length $L$ and lying at a distance $R>0$ from the critical points, with respect to the flat metric.

If $\psi:S\setminus M\rightarrow C$ is a harmonic map which is a diffeomorphism onto its image with Hopf~differential~$\phi$, then there is a universal constant $k>0$ such that 
\begin{enumerate}
\item The image $\psi(\alpha_h)$ has length 
\[
2L+O(e^{-kR}) \quad\text{as $R\rightarrow\infty$}.
\]
Moreover, this image $\psi(\alpha_h)$ is a curve having geodesic curvature $O(e^{-kR})$. It is contained in an $\epsilon$-neighborhood of a geodesic segment where $\epsilon\rightarrow0$ as $R\rightarrow\infty$.
\item The image $\psi(\alpha_v)$ has length 
\[
O(Le^{-kR}) \quad\text{as $R\rightarrow\infty$}.
\]
\end{enumerate}
\end{proposition}

If $w=x+\mathrm{i}y$ is the distinguished local coordinate on $S\setminus\Crit(\phi)$, then it follows from equation~(5) of~\cite{Gupta19} that the tangent vectors $\partial_x$ and~$\partial_y$ are orthogonal with respect to the metric obtained by pulling back the hyperbolic metric along~$\psi$. Thus Proposition~\ref{prop:horizontalverticallength} implies that $\psi(\alpha_v)$ is an approximately horocyclic arc for $R\gg0$.

\subsection{A contour integral}
\label{sec:AContourIntegral}

Let $\phi$ be a complete and saddle-free differential on~$S$, and let $M\subset S$ be the set of poles of~$\phi$. We will consider the universal cover $\pi:U\rightarrow S\setminus M$ of the surface $S\setminus M$. The quadratic differential $\phi$ determines a singular horizontal foliation of~$S\setminus M$, and this can be lifted to a singular foliation of~$U$. Given a standard saddle connection~$\alpha$ for~$\phi$, we can consider a lift $\widetilde{\alpha}$ of this curve~$\alpha$ to~$U$. We define $c_{\widetilde{\alpha}}$ to be the boundary of an octagon in~$U$ which contains~$\widetilde{\alpha}$ in its interior and whose sides alternately project to vertical and horizontal paths on~$S$. An example is illustrated in bold in~Figure~\ref{fig:octagon}, which also illustrates the lifted foliation.

\begin{figure}[ht]
\begin{center}
\begin{tikzpicture}[scale=2]
\begin{scope}
\clip  (-1.5,-0.9) rectangle (1.5,0.9);
\draw[black,ultra thin] plot[smooth] coordinates {(-1.500,1.466) (-1.363,1.255) (-1.205,1.012) (-0.999,0.720) (-0.673,0.378) (-0.201,0.145) (0.277,-0.012) (0.806,-0.258) (1.116,-0.671) (1.295,-0.977) (1.373,-1.111) (1.508,-1.340)};
\draw[black,ultra thin] plot[smooth] coordinates {(-1.883,-0.931) (-1.673,-0.755) (-1.414,-0.539) (-0.999,-0.250) (-0.384,-0.255) (0.083,-0.385) (0.500,-0.581) (0.827,-0.842) (1.058,-1.098) (1.236,-1.325)};
\draw[black,ultra thin] plot[smooth] coordinates {(-1.851,-1.035) (-1.636,-0.877) (-1.373,-0.695) (-0.999,-0.500) (-0.496,-0.445) (-0.045,-0.537) (0.352,-0.697) (0.680,-0.915) (0.929,-1.145) (1.124,-1.359)};
\draw[black,ultra thin] plot[smooth] coordinates {(1.500,-1.466) (1.363,-1.255) (1.205,-1.012) (0.999,-0.720) (0.673,-0.378) (0.201,-0.145) (-0.277,0.012) (-0.806,0.258) (-1.116,0.671) (-1.295,0.977) (-1.373,1.111) (-1.508,1.340)};
\draw[black,ultra thin] plot[smooth] coordinates {(1.883,0.931) (1.673,0.755) (1.414,0.539) (0.999,0.250) (0.384,0.255) (-0.083,0.385) (-0.500,0.581) (-0.827,0.842) (-1.058,1.098) (-1.236,1.325)};
\draw[black,ultra thin] plot[smooth] coordinates {(1.851,1.035) (1.636,0.877) (1.373,0.695) (0.999,0.500) (0.496,0.445) (0.045,0.537) (-0.352,0.697) (-0.680,0.915) (-0.929,1.145) (-1.124,1.359)};
\draw[black,ultra thin] plot[smooth] coordinates {(1.82,1.2) (1.606,1.068) (1.352,0.927) (1.033,0.794) (0.637,0.725) (0.231,0.761) (-0.141,0.873) (-0.464,1.036)};
\draw[black,ultra thin] plot[smooth] coordinates {(-1.82,-1.2) (-1.606,-1.068) (-1.352,-0.927) (-1.033,-0.794) (-0.637,-0.725) (-0.231,-0.761) (0.141,-0.873) (0.464,-1.036)};
\draw[black,ultra thin] plot[smooth] coordinates {(1.544,-1.145) (1.421,-0.884) (1.285,-0.554) (1.2,0.001) (1.500,0.422) (1.737,0.664) (1.937,0.853)};
\draw[black,ultra thin] plot[smooth] coordinates {(1.591,-1.101) (1.478,-0.835) (1.364,-0.496) (1.35,0.001) (1.582,0.387) (1.799,0.632) (1.989,0.824)};
\draw[black,ultra thin] plot[smooth] coordinates {(-1.544,1.145) (-1.421,0.884) (-1.285,0.554) (-1.2,-0.001) (-1.500,-0.422) (-1.737,-0.664) (-1.937,-0.853)};
\draw[black,ultra thin] plot[smooth] coordinates {(-1.591,1.101) (-1.478,0.835) (-1.364,0.496) (-1.35,-0.001) (-1.582,-0.387) (-1.799,-0.632) (-1.989,-0.824)};
\draw[black,ultra thin] plot[smooth] coordinates {(-0.999,0) (-0.227,-0.198) (0.234,-0.351) (0.649,-0.583) (0.946,-0.870) (1.153,-1.129) (1.317,-1.354)};
\draw[black,ultra thin] plot[smooth] coordinates {(0.999,0) (0.227,0.198) (-0.234,0.351) (-0.649,0.583) (-0.946,0.870) (-1.153,1.129) (-1.317,1.354)};
\draw[black,ultra thin] plot[smooth] coordinates {(1.001,0.01) (1.449,0.455) (1.699,0.690) (1.905,0.876)};
\draw[black,ultra thin] plot[smooth] coordinates {(1.001,-0.01) (1.237,-0.607) (1.387,-0.927) (1.517,-1.182)};
\draw[black,ultra thin] plot[smooth] coordinates {(-1.001,-0.01) (-1.449,-0.455) (-1.699,-0.690) (-1.905,-0.876)};
\draw[black,ultra thin] plot[smooth] coordinates {(-1.001,0.01) (-1.237,0.607) (-1.387,0.927) (-1.517,1.182)};
\node at (-1,0) {{\tiny $\times$}};
\node at (1,0) {{\tiny $\times$}};
\draw[black, very thick] (0.999,0.250) .. controls (1.03,0.17) and (1.1,0.1) .. (1.215,0.047);
\draw[black, very thick] (1.215,0.047) .. controls (1.175,0) and (1.2,-0.25) .. (1.285,-0.554);
\draw[black, very thick] (-0.999,-0.250) .. controls (-1.03,-0.17) and (-1.1,-0.1) .. (-1.215,-0.047);
\draw[black, very thick] (-1.215,-0.047) .. controls (-1.175,0) and (-1.2,0.25) .. (-1.285,0.554);
\draw[black, very thick] (0.815,-0.83) .. controls (0.825,-0.825) and (1,-0.65) .. (1.285,-0.554);
\draw[black, very thick] (-0.815,0.83) .. controls (-0.825,0.825) and (-1,0.65) .. (-1.285,0.554);
\draw[black, very thick] (-0.815,0.83) .. controls (-0.45,0.47) and (0.3,0.14) .. (0.999,0.250);
\draw[black, very thick] (0.815,-0.83) .. controls (0.45,-0.47) and (-0.3,-0.14) .. (-0.999,-0.250);
\node at (0,0.45) {{\tiny $\sigma_4$}};
\node at (0,-0.45) {{\tiny $\sigma_2$}};
\node at (-1.3,0.2) {{\tiny $\sigma_1$}};
\node at (1.3,-0.2) {{\tiny $\sigma_3$}};
\node at (-1.1,0.75) {{\tiny $\rho_1$}};
\node at (1.1,-0.75) {{\tiny $\rho_3$}};
\node at (-1.15,-0.2) {{\tiny $\rho_2$}};
\node at (1.15,0.2) {{\tiny $\rho_4$}};
\end{scope}
\end{tikzpicture}
\end{center}
\caption{An octagon associated to a saddle connection.\label{fig:octagon}}
\end{figure}
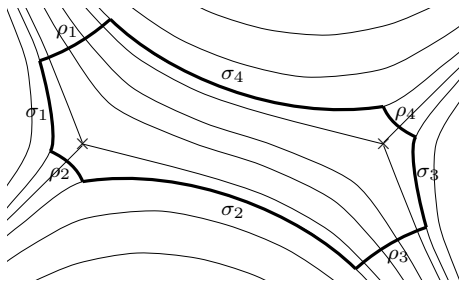

We will write $\widetilde{\phi}=\pi^*\phi$ for the pullback of the quadratic differential to the universal cover. After making a branch cut along~$\widetilde{\alpha}$, we can choose a branch $\lambda$ of the square root of $\widetilde{\phi}$ which is continuous on~$c_{\widetilde{\alpha}}$. We can also choose an orientation for $c_{\widetilde{\alpha}}$ in such a way that the integral $\int_{c_{\widetilde{\alpha}}}\lambda$ over the resulting contour has positive imaginary part. Then we have the following statement, which is closely related to Lemma~2.28 of~\cite{Gupta19}.

\begin{lemma}
\label{lem:contourintegral}
Let $L_1,\dots,L_4$ be the lengths of the segments labeled $\sigma_1,\dots,\sigma_4$, respectively, in Figure~\ref{fig:octagon}, where length is defined using the pullback of the flat metric on~$S\setminus\Crit(\phi)$. Then 
\[
\Re\left(\int_{c_{\widetilde{\alpha}}}\lambda\right)=\sum_i(-1)^iL_i.
\]
\end{lemma}

\begin{proof}
If $e$ is a segment of $c_{\widetilde{\alpha}}$ that projects to a vertical path on~$S$, then the integral of~$\lambda$ over~$e$ is a pure imaginary number with absolute value equal to the length of~$e$. If $e$ is a segment of $c_{\widetilde{\alpha}}$ that projects to a horizontal path on~$S$, then the integral over~$e$ is a real number with absolute value equal to the length of~$e$. Moreover, the integrals over two successive segments that project to paths of the same type (horizontal or vertical) have opposite arguments. The segments labeled~$\rho_1$ and~$\rho_3$ in Figure~\ref{fig:octagon} are strictly longer than the segments labeled~$\rho_2$ and~$\rho_4$, and therefore the requirement that $\int_{c_{\widetilde{\alpha}}}\lambda$ have positive imaginary part implies that $\int_{\rho_i}\lambda$ has positive imaginary part for $i=1,3$. It follows that $\int_{\sigma_i}\lambda=-L_i$ for $i=1,3$ while $\int_{\sigma_i}\lambda=L_i$ for $i=2,4$.
\end{proof}

\subsection{Asymptotics of cross ratios}
\label{sec:AsymptoticsOfCrossRatios}

Continuing with the notation of Section~\ref{sec:AContourIntegral}, let us now consider the rescaled differential $R^2\cdot\phi$ for $R\in\mathbb{R}$. Let us write $P=(P_i(z_i))$ for a tuple of principal differentials encoding the principal parts of~$\phi$ at points~$p_i\in M$. Then there is a point $\Psi_{R\cdot P}^{S,M}(R^2\cdot\phi)\in\mathcal{T}_{R\cdot P}(S,M)$, which can be represented by a marked hyperbolic surface~$(C_R,\psi_R)$ where $\psi_R:S\setminus M\rightarrow C_R$ is a harmonic diffeomorphism onto its image. This map $\psi_R$ can be lifted to a map $\widetilde{\psi}_R:U\rightarrow\mathbb{H}$ which intertwines the action of $\pi_1(S\setminus M)$ on~$U$ by deck transformations with its action on~$\mathbb{H}$ via the monodromy representation.

Let $\alpha$ be any standard saddle connection for~$\phi$, and consider once again the octagon~$c_{\widetilde{\alpha}}$ illustrated in Figure~\ref{fig:octagon}. Each side $\sigma_i$ is contained in a leaf~$Z_i$ of the horizontal foliation, and Proposition~\ref{prop:horizontalverticallength}(1) implies that as we travel along this leaf~$Z_i$ in either direction, the geodesic curvature of $\widetilde{\psi}_R(Z_i)$ tends to zero. It follows that the curve~$\widetilde{\psi}_R(Z_i)$ is asymptotic to a point on~$\partial\bar{\mathbb{H}}$. Moreover, if $\sigma_i$ and~$\sigma_j$ are successive horizontal sides of~$c_{\widetilde{\alpha}}$, then Proposition~\ref{prop:horizontalverticallength}(2) implies that $\widetilde{\psi}(Z_i)$ and $\widetilde{\psi}(Z_j)$ are asymptotic to a common point on~$\partial\bar{\mathbb{H}}$. Hence the images of the leaves containing the sides $\sigma_i$ are asymptotic to exactly four points on~$\partial\bar{\mathbb{H}}$. 

Let us label these points as $z_1,\dots,z_4\in\partial\bar{\mathbb{H}}$ so that the order is compatible with the orientation of $\partial\bar{\mathbb{H}}\cong\mathbb{RP}^1$. We choose these labels in such a way that the arc $\rho_i$ in Figure~\ref{fig:octagon} connects two leaves of the horizontal foliation whose images are asymptotic to the point labeled~$z_i$. Our goal is to understand the cross ratio $Y_\alpha=Y_{\phi,\alpha}(R)$ of $z_1,\dots,z_4$, defined by the formula~\eqref{eqn:crossratio}. To do this, we will first review an elementary fact from hyperbolic geometry.

Consider a geodesic~$g$ in~$\mathbb{H}$ which connects two points on~$\partial\bar{\mathbb{H}}$, and suppose we are given disjoint horocycles around the endpoints of this geodesic (see Figure~\ref{fig:lambdalength}). In this situation, we let $\mu$ be the distance between the points of intersection of~$g$ with the horocycles. Following Penner~\cite{Penner12}, we define the \emph{lambda length} to be $\lambda=\exp(\mu/2)$.

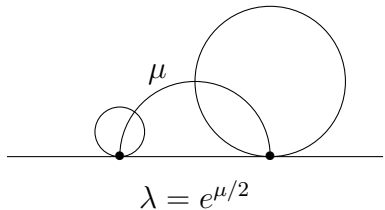
\begin{figure}[ht]
\begin{center}
\begin{tikzpicture}
\draw[black, thin] (-2.5,0) -- (2.5,0);
\draw[black, thin] (-1,0) arc(180:0:1);
\node[black] at (-1,0) {{\tiny $\bullet$}};
\node[black] at (1,0) {{\tiny $\bullet$}};
\draw[black, thin] (-1,0.33) circle (0.33);
\draw[black, thin] (1,1) circle (1);
\node[black] at (0,-0.5) {{$\lambda=e^{\mu/2}$}};
\node[black] at (-0.5,1.1) {{$\mu$}};
\end{tikzpicture}
\end{center}
\caption{Defining the lambda length.\label{fig:lambdalength}}
\end{figure}

In particular, we can consider the ideal quadrilateral with vertices $z_1,\dots,z_4$, and we can choose disjoint horocycles around these vertices (see Figure~\ref{fig:decoratedquadrilateral}). Then by applying the above construction to the geodesic connecting the points~$z_i$ and~$z_j$, we get a corresponding lambda length~$\lambda_{ij}$.

\begin{figure}[ht]
\begin{center}
\begin{tikzpicture}
\draw[black, ultra thin] (0,2) circle (2);
\draw[black, thin] (0,0) arc(180:74:1.5);
\draw[black, thin] (0,4) arc(180:254:2.64);
\draw[black, thin] (0,4) arc(0:-106:1.5);
\draw[black, thin] (0,0) arc(0:74:2.64);
\node[black] at (0,0) {{\tiny $\bullet$}};
\node[black] at (0,4) {{\tiny $\bullet$}};
\node[black] at (1.91,1.44) {{\tiny $\bullet$}};
\node[black] at (-1.91,2.53) {{\tiny $\bullet$}};
\node[black] at (0,4.25) {{$z_1$}};
\node[black] at (2.16,1.44) {{$z_4$}};
\node[black] at (0,-0.25) {$z_3$};
\node[black] at (-2.16,2.53) {{$z_2$}};
\draw[black, thin] (0,3.65) circle (0.35);
\draw[black, thin] (0,0.35) circle (0.35);
\draw[black, thin] (1.67,1.52) circle (0.25);
\draw[black, thin] (-1.67,2.48) circle (0.25);
\end{tikzpicture}
\end{center}
\caption{An ideal quadrilateral with a horocycle around each vertex.\label{fig:decoratedquadrilateral}}
\end{figure}
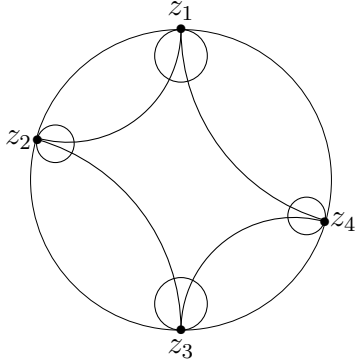

\begin{lemma}[\cite{Penner12}, Chapter~1, Corollary~4.16]
\label{lem:lambdalengthcrossratio}
The cross ratio~$Y_\alpha$ is given in terms of the lambda lengths by 
\begin{equation}
\label{eqn:lambdalengthcrossratio}
Y_\alpha=\frac{\lambda_{12}\lambda_{34}}{\lambda_{23}\lambda_{14}}.
\end{equation}
\end{lemma}

Lemma~\ref{lem:lambdalengthcrossratio} implies in particular that the expression on the right hand side of~\eqref{eqn:lambdalengthcrossratio} is independent of the choice of horocycles. Using this lemma, we can derive an asymptotic relation between the cross ratio~$Y_\alpha$ and the integral considered previously.

\begin{lemma}
\label{lem:asymptoticscrossratio}
Taking notation as above, we have 
\[
Y_{\phi,\alpha}(R)\cdot\exp\left(R\cdot\Re\int_{c_{\widetilde{\alpha}}}\lambda\right)\rightarrow1 \quad \text{as $R\rightarrow\infty$}.
\]
\end{lemma}

\begin{proof}
For each $R$, let us choose disjoint horocycles around the vertices $z_1,\dots,z_4$. Let us write $\mu_i$ for the hyperbolic distance between the horocycles around~$z_i$ and~$z_{i+1}$ where the indices are considered modulo~4. Then by definition we have $\lambda_{i,i+1}=\exp(\mu_i/2)$. Let us also write $L_i(R)$ for the length of the segment~$\sigma_i$ in Figure~\ref{fig:octagon}, where length is defined using the pullback of the $R^2\cdot\phi$-metric. By applying Lemma~\ref{lem:contourintegral} and Lemma~\ref{lem:lambdalengthcrossratio}, we find 
\begin{equation}
\label{eqn:crossratiointegral}
Y_{\phi,\alpha}(R)\cdot\exp\left(R\cdot\Re\int_{c_{\widetilde{\alpha}}}\lambda\right) = \exp\left(\sum_i(-1)^i\left(L_i(R)-\frac{\mu_i}{2}\right)\right).
\end{equation}
By Proposition~\ref{prop:horizontalverticallength}, we know that for $R\gg0$ the curve $\widetilde{\psi}_R(\sigma_i)$ approximates a segment of the geodesic connecting~$z_i$ and~$z_{i+1}$, while $\widetilde{\psi}_R(\rho_i)$ is an approximately horocyclic arc. Moreover, $L_i(R)$ is approximately half the length of $\widetilde{\psi}_R(\sigma_i)$. Since $\sum_i(-1)^i\mu_i$ is independent of the choice of horocycles, it follows that the sum on the right hand side of~\eqref{eqn:crossratiointegral} tends to zero as $R\rightarrow\infty$. This completes the proof.
\end{proof}

\subsection{The foliation near a double pole}

Consider again a quadratic differential $\phi$ on~$S$, and suppose that $\phi$ has a double pole at $p\in S$. In this case, one can show (\cite{Strebel84}, Theorem~6.3) that there exists a local coordinate~$t$, defined in a neighborhood of~$p$, such that 
\[
\phi(t)=\frac{a_p}{t^2}dt^{\otimes2}
\]
for a well defined constant $a_p\in\mathbb{C}^*$. It follows that away from~$p$, any branch of the function $w=\sqrt{a_p}\log(t)$ is a distinguished local coordinate. Assuming $a_p\not\in\mathbb{R}$, the condition $\Im(w)=\text{constant}$ describes a spiral in the $t$-plane around the point $t=0$ where the direction of spiraling depends on the value of the leading coefficient. The curve spirals in the clockwise direction if $\Im(a_p)<0$ and in the counterclockwise direction if $\Im(a_p)>0$. (See Figure~\ref{fig:foliationdoublepole}.)

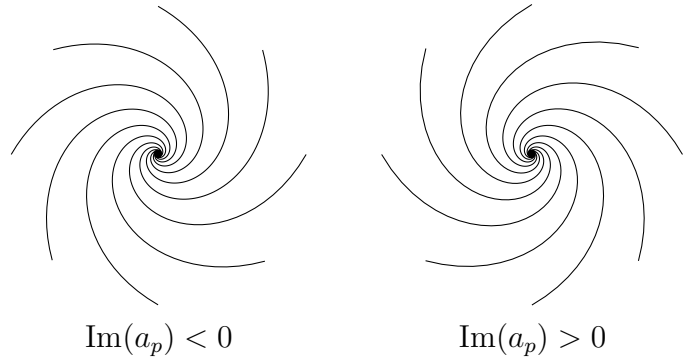
\begin{figure}[ht]
\begin{center}
\begin{tikzpicture}
\begin{polaraxis}[width=60mm,height=60mm,axis lines=none]
\addplot[domain=0:720,samples=150] {exp(0.01*x)}; 
\addplot[domain=0:675,samples=150] {1.6*exp(0.01*x)}; 
\addplot[domain=0:630,samples=150] {2.5*exp(0.01*x)}; 
\addplot[domain=0:585,samples=150] {3.9*exp(0.01*x)}; 
\addplot[domain=0:540,samples=150] {6.0*exp(0.01*x)}; 
\addplot[domain=0:495,samples=150] {9.5*exp(0.01*x)}; 
\addplot[domain=0:450,samples=150] {14.9*exp(0.01*x)}; 
\addplot[domain=0:405,samples=150] {23.3*exp(0.01*x)}; 
\node[black] at (0,0) {{\tiny $\bullet$}};
\end{polaraxis}
\node at (2.22,-0.25) {$\Im(a_p)<0$};
\end{tikzpicture}
\quad
\begin{tikzpicture}
\begin{polaraxis}[width=60mm,height=60mm,axis lines=none]
\addplot[domain=0:720,samples=150] {exp(-0.01*x)}; 
\addplot[domain=-45:720,samples=150] {0.638*exp(-0.01*x)}; 
\addplot[domain=-90:720,samples=150] {0.407*exp(-0.01*x)}; 
\addplot[domain=-135:720,samples=150] {0.259*exp(-0.01*x)}; 
\addplot[domain=-180:720,samples=150] {0.165*exp(-0.01*x)}; 
\addplot[domain=-225:720,samples=150] {0.105*exp(-0.01*x)}; 
\addplot[domain=-270:720,samples=150] {0.067*exp(-0.01*x)}; 
\addplot[domain=-315:720,samples=150] {0.043*exp(-0.01*x)}; 
\node[black] at (0,0) {{\tiny $\bullet$}};
\end{polaraxis}
\node at (2.22,-0.25) {$\Im(a_p)>0$};
\end{tikzpicture}
\end{center}
\caption{The horizontal foliation near a double pole.\label{fig:foliationdoublepole}}
\end{figure}

\begin{lemma}
\label{lem:spiralingdirection}
Let $\phi\in\mathcal{Q}_P(S,M)$ and let $(C,\psi)$ be a marked hyperbolic surface representing $\Psi_P^{S,M}(\phi)$ where $\psi:S\setminus M\rightarrow C$ is a harmonic diffeomorphism onto its image. If $\alpha$ is a horizontal trajectory for~$\phi$ which is asymptotic to a double pole $p\in M$ such that $a_p\not\in\mathbb{R}$, then 
\begin{enumerate}
\item The curve $\psi(\alpha)$ winds infinitely many times around the boundary component of~$C$ corresponding to~$p$, becoming arbitrarily close to this boundary component. (See Figure~\ref{fig:geodesicspiral}.)
\item The direction of spiraling is compatible with the orientation of~$C$ if $\Im(a_p)<0$. It is the opposite direction if $\Im(a_p)>0$.
\end{enumerate}
\end{lemma}

\begin{proof}
Consider a small disk $D\subset S$ centered at the point~$p$. If $e$ is a segment that points radially outward from~$p$, connecting this point to~$\partial D$, then by the above discussion, we know that $e$ intersects~$\alpha$ in infinitely many points which accumulate at~$p$. The map~$\psi$ takes $D\setminus\{p\}$ to an annulus $A\subset C$ and sends~$e$ to a curve connecting the two boundary components of this annulus. Since $\psi$ is a diffeomorphism onto its image, $\psi(\alpha)$ intersects $\psi(e)$ in infinitely many points that accumulate at the boundary of~$A$. This implies~(1). Part~(2) follows from the similar behavior of the horizontal foliation near~$p$.
\end{proof}

\subsection{The WKB triangulation}

An important fact first noted in~\cite{GaiottoMooreNeitzke13} is that a complete, saddle-free GMN differential determines an associated ideal triangulation. More precisely, suppose we are given a marked quadratic differential $(S,\phi,\theta)$ representing a point in the space $\mathcal{Q}(\mathbb{S},\mathbb{M})$. If $\phi$ is complete and saddle free, then we have seen that the horizontal foliation determines a decomposition of~$S$ into horizontal strips and half planes. By choosing a single generic trajectory within each of the horizontal strips, we obtain a collection of paths on~$S$. The preimage of this collection under the diffeomorphism~$\theta$ is a collection of arcs which define an ideal triangulation $T(\phi)$ of~$(\mathbb{S},\mathbb{M})$. Following the terminology of~\cite{GaiottoMooreNeitzke13}, we call $T(\phi)$ the \emph{WKB~triangulation}.

If the differential~$\phi$ is equipped with a signing, then $T(\phi)$ can be equipped with a natural signing denoted $\epsilon(\phi)$. Indeed, the fact that $\phi$ is saddle-free implies that the residue of~$\phi$ at a double pole corresponding to a puncture $p\in\mathbb{P}$ cannot be real (\cite{BridgelandSmith15}, Section~10.1), and so the chosen value of the residue must lie in either the upper or lower half plane. If this value lies in the upper half plane, we will set $\epsilon(\phi)(p)=+1$, and if it lies in the lower half plane, we will set $\epsilon(\phi)(p)=-1$. Then $(T(\phi),\epsilon(\phi))$ is called the \emph{signed WKB~triangulation}, and its equivalence class is called the \emph{tagged WKB~triangulation} and denoted~$\tau(\phi)$.

\subsection{Asymptotics of cluster coordinates}

We now come to the main result of this section. Fix a complete, saddle-free differential $\phi\in\mathcal{Q}^\pm(\mathbb{S},\mathbb{M})$. Here we use the same symbol to denote a point of $\mathcal{Q}^\pm(\mathbb{S},\mathbb{M})$ and an underlying quadratic differential. Then for any $R>0$, we can consider the image of the rescaled differential $R^2\cdot\phi$ under the map $\Psi^\pm:\mathcal{Q}^\pm(\mathbb{S},\mathbb{M})\rightarrow\mathcal{T}^\pm(\mathbb{S},\mathbb{M})$. Under our assumptions, the differential $\phi$ determines an associated tagged WKB~triangulation~$\tau(\phi)$, and we can define cluster coordinates with respect to~$\tau(\phi)$. In particular, suppose $\gamma=\gamma_\alpha\in\widehat{H}(\phi)$ is the class of a standard saddle connection~$\alpha$ for~$\phi$. In the following, we will use the same symbol $\alpha$ to denote the corresponding arc of the WKB triangulation, and we will write $X_{\phi,\gamma}(R)$ for the cluster coordinate of $\Psi^\pm(R^2\cdot\phi)$ with respect to this tagged arc~$\alpha$.

\begin{theorem}
\label{thm:clusterasymptotics}
Taking notation as above, we have 
\[
X_{\phi,\gamma}(R)\cdot\exp(R\cdot \Re Z_\phi(\gamma))\rightarrow1 \quad \text{as $R\rightarrow\infty$}.
\]
\end{theorem}

\begin{proof}
Let $\beta$ be an arc of the WKB triangulation, and let $(C_R,\psi_R)$ be a marked hyperbolic surface representing $\Psi^\pm(R^2\cdot\phi)$ with $\psi_R$ harmonic. We will begin by describing the image of $\beta$ under the map~$\psi_R$. To do this, suppose $p$ is a double pole of~$\phi$ which corresponds to an endpoint of this arc~$\beta$. We have a distinguished value $r_p$ of the residue of~$\phi$ at this pole~$p$. Suppose this value $r_p$ satisfies $\Re(r_p)>0$. If we also have $\Im(r_p)>0$, then the leading coefficient $a_p$ of~$\phi$ at~$p$ satisfies $\Im(a_p)<0$. Then it follows from Lemma~\ref{lem:spiralingdirection} that the image of~$\beta$ under $\psi_R$ spirals into the boundary of~$C_R$ in the direction compatible with the orientation of the surface. By the definitions of Section~\ref{sec:ALiftToCoveringSpaces}, this agrees with the orientation of the boundary component prescribed by $\Psi^\pm(R^2\cdot\phi)\in\mathcal{T}^\pm(\mathbb{S},\mathbb{M})$.

On the other hand, if $\Re(r_p)>0$ and $\Im(r_p)<0$, then the leading coefficient~$a_p$ of~$\phi$ at the pole~$p$ satisfies $\Im(a_p)>0$. In this case, Lemma~\ref{lem:spiralingdirection} implies that $\psi_R(\beta)$ spirals into the boundary of~$C_R$ in the direction opposite to the orientation of the surface. If we change the chosen orientation of this boundary component of~$C_R$ by acting on $\Psi^\pm(R^2\cdot\phi)$ by the signing~$\epsilon(\phi)$, then the definitions of Section~\ref{sec:ALiftToCoveringSpaces} imply that $\psi_R(\beta)$ spirals in the direction prescribed by the resulting orientation.

One can perform a similar analysis in the case where $\Re(r_p)<0$. If $\Re(r_p)=0$ then $p$ corresponds via the map $\psi_R$ to a cusp of~$C_R$, and the leaves of the horizontal foliation do not spiral into the pole~$p$. Thus we see that for any arc $\beta$ of the WKB~triangulation, the ends of~$\psi_R(\beta)$ spiral into the boundary of~$C_R$ in the direction prescribed by $\epsilon(\phi)\cdot\Psi^\pm(R^2\cdot\phi)$.

It follows that if $\alpha$ is not the interior edge of a self-folded triangle, then the cluster coordinate is given by $X_{\phi,\gamma}(R)=Y_{\phi,\alpha}(R)$ where $Y_{\phi,\alpha}(R)$ is defined as in Section~\ref{sec:AsymptoticsOfCrossRatios}, and the period is $Z_\phi(\gamma)=\int_{c_{\widetilde{\alpha}}}\lambda$. If $\alpha$ is the internal edge of a self-folded triangle and $\beta$ is the encircling edge, then $X_{\phi,\gamma}(R)=Y_{\phi,\alpha}(R)Y_{\phi,\beta}(R)$. In this case, one also has 
\[
Z_\phi(\gamma)=\int_{c_{\widetilde{\alpha}}}\lambda+\int_{c_{\widetilde{\beta}}}\lambda
\]
by the proof of Theorem~7.8 in~\cite{Allegretti19}. The desired statement therefore follows from Lemma~\ref{lem:asymptoticscrossratio}.
\end{proof}

\section{Triangulated categories}
\label{sec:TriangulatedCategories}

In this section, we review the construction of the 3-Calabi-Yau triangulated category associated to a triangulated surface.

\subsection{Quivers with potential}

Recall that a \emph{quiver} $Q$ is simply a directed graph. It consists of a finite set $Q_0$ of \emph{vertices}, a finite set $Q_1$ of \emph{arrows}, and maps $s:Q_1\rightarrow Q_0$ and $t:Q_1\rightarrow Q_0$ taking an arrow to its \emph{source} and \emph{target}, respectively. A \emph{path} of length~$d>0$ in~$Q$ is defined as a sequence of arrows $a_1,\dots,a_d$ such that $s(a_i)=t(a_{i+1})$ for every $i$. Such a path will be denoted by $p=a_1\dots a_d$. We define its \emph{source} by $s(p)=s(a_d)$ and its \emph{target} by $t(p)=t(a_1)$. A path is \emph{cyclic} if its source and target coincide. Two paths $p$ and $q$ are \emph{composable} if $s(p)=t(q)$, and in this case their composition $pq$ is defined by concatenation. We also allow paths of length zero and define composition analogously for such paths.

If $Q$ is a quiver and $\Bbbk$ is a field, we write $\Bbbk Q$ for the $\Bbbk$-vector space spanned by the set of all paths in~$Q$. This vector space has a natural bilinear product operation where the product of two paths is defined to be their composition if the paths are composable and is defined to be zero otherwise. The space $\Bbbk Q$ equipped with this multiplication is called the \emph{path algebra} of~$Q$. The paths of length one generate a two-sided ideal $\mathfrak{a}\subset\Bbbk Q$ in the path algebra, and the \emph{complete path algebra} $\widehat{\Bbbk Q}$ is defined as the completion of~$\Bbbk Q$ with respect to this ideal~$\mathfrak{a}$. Concretely, it is the vector space generated by possibly infinite $\Bbbk$-linear combinations of paths in~$Q$ with multiplication induced by composition of paths.

A \emph{potential} for~$Q$ is defined as an element of $\widehat{\Bbbk Q}$ each of whose terms is a cyclic path of positive length. Two potentials $W$ and~$W'$ for~$Q$ are \emph{cyclically equivalent} if their difference $W-W'$ lies in the closure of the vector subspace of~$\widehat{\Bbbk Q}$ spanned by elements of the form $a_1\dots a_d-a_2\dots a_da_1$ where $a_1\dots a_d$ is a cyclic path of positive length. A \emph{quiver with potential} is a pair $(Q,W)$ where $Q$ is a quiver and $W$ is a potential for~$Q$ considered up to cyclic equivalence.

\subsection{Quivers with potential from surfaces}

In this paper, we are interested in a particular class of quivers with potential associated to triangulated surfaces in the work of Labardini-Fragoso~\cite{LabardiniFragoso08}. Let $(\mathbb{S},\mathbb{M})$ be a marked bordered surface and $T$ an ideal triangulation of~$(\mathbb{S},\mathbb{M})$. If $p\in\mathbb{P}$ is any puncture, then the \emph{valency} of~$p$ with respect to~$T$ is defined as the number of half arcs of~$T$ that are incident to~$p$. We say that $T$ is \emph{regular} if every puncture has valency $\geq3$ with respect to~$T$. Below we will define a quiver with potential $(Q(T),W(T,\epsilon))$ associated to a signed triangulation $(T,\epsilon)$ for which the underlying ideal triangulation~$T$ is regular.

By definition, $Q(T)$ has exactly one vertex for each arc of~$T$. We typically draw the vertices at the midpoints of the corresponding arcs on the surface and use the same symbol to denote a vertex and the corresponding arc. The quiver $Q(T)$ has $\varepsilon_{ij}^T$ arrows from $j$ to~$i$ whenever $\varepsilon_{ij}^T>0$. Figure~\ref{fig:quiver} illustrates a portion of an ideal triangulation and the associated quiver.

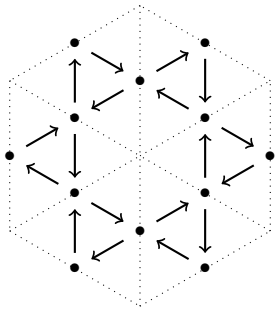
\begin{figure}[ht]
\begin{center}
\begin{tikzpicture}
\coordinate (a) at (0,0);
\coordinate (b) at (1.732,1);
\coordinate (c) at (0,2);
\coordinate (d) at (-1.732,1);
\coordinate (e) at (-1.732,-1);
\coordinate (f) at (0,-2);
\coordinate (g) at (1.732,-1);
\node (B) at (0.866,0.5) {{\tiny $\bullet$}};
\node (C) at (0,1) {{\tiny $\bullet$}};
\node (D) at (-0.866,0.5) {{\tiny $\bullet$}};
\node (E) at (-0.866,-0.5) {{\tiny $\bullet$}};
\node (F) at (0,-1) {{\tiny $\bullet$}};
\node (G) at (0.866,-0.5) {{\tiny $\bullet$}};
\node (BB) at (1.732,0) {{\tiny $\bullet$}};
\node (CC) at (0.866,1.5) {{\tiny $\bullet$}};
\node (DD) at (-0.866,1.5) {{\tiny $\bullet$}};
\node (EE) at (-1.732,0) {{\tiny $\bullet$}};
\node (FF) at (-0.866,-1.5) {{\tiny $\bullet$}};
\node (GG) at (0.866,-1.5) {{\tiny $\bullet$}};
\draw[black, thin, dotted] (a) -- (b);
\draw[black, thin, dotted] (b) -- (c);
\draw[black, thin, dotted] (c) -- (a);
\draw[black, thin, dotted] (c) -- (d);
\draw[black, thin, dotted] (d) -- (a);
\draw[black, thin, dotted] (d) -- (e);
\draw[black, thin, dotted] (e) -- (a);
\draw[black, thin, dotted] (e) -- (f);
\draw[black, thin, dotted] (f) -- (a);
\draw[black, thin, dotted] (f) -- (g);
\draw[black, thin, dotted] (g) -- (a);
\draw[black, thin, dotted] (g) -- (b);
\draw[black, thick,->] (B) -- (C);
\draw[black, thick,->] (C) -- (D);
\draw[black, thick,->] (D) -- (E);
\draw[black, thick,->] (E) -- (F);
\draw[black, thick,->] (F) -- (G);
\draw[black, thick,->] (G) -- (B);
\draw[black, thick,->] (BB) -- (G);
\draw[black, thick,->] (B) -- (BB);
\draw[black, thick,->] (CC) -- (B);
\draw[black, thick,->] (C) -- (CC);
\draw[black, thick,->] (DD) -- (C);
\draw[black, thick,->] (D) -- (DD);
\draw[black, thick,->] (EE) -- (D);
\draw[black, thick,->] (E) -- (EE);
\draw[black, thick,->] (FF) -- (E);
\draw[black, thick,->] (F) -- (FF);
\draw[black, thick,->] (GG) -- (F);
\draw[black, thick,->] (G) -- (GG);
\end{tikzpicture}
\end{center}
\caption{The quiver associated to an ideal triangulation.\label{fig:quiver}}
\end{figure}

The quiver $Q(T)$ constructed in this way has two obvious kinds of cyclic paths. Let us say that a triangle of~$T$ is \emph{internal} if all of its sides are arcs. Then there is a clockwise oriented path $\tau(t)$ of length three inscribed within each internal triangle~$t$ of the triangulation~$T$. On the other hand, there is a counterclockwise oriented path $\pi(p)$ of length~$\geq3$ encircling each puncture~$p\in\mathbb{P}$. Following Labardini-Fragoso~\cite{LabardiniFragoso08}, we define a canonical potential~$W(T,\epsilon)$ for~$Q(T)$ by the formula 
\[
W(T,\epsilon)=\sum_t\tau(t)-\sum_p\epsilon(p)\pi(p)
\]
where the first sum runs over all internal triangles~$t$ of~$T$ and the second sum runs over all punctures $p\in\mathbb{P}$. Thus we get a quiver with potential $(Q(T),W(T,\epsilon))$ canonically associated to $(T,\epsilon)$. This construction can be generalized to the case where $T$ is non-regular; we refer to~\cite{LabardiniFragoso08} for the details concerning non-regular triangulations.

In~\cite{DWZ08}, Derksen, Weyman, and Zelevinsky defined two quivers with potential $(Q,W)$ and $(Q',W')$ to be \emph{right equivalent} if there exists an isomorphism $F:\widehat{\Bbbk Q}\stackrel{\sim}{\rightarrow}\widehat{\Bbbk Q'}$ of their completed path algebras such that $F$ preserves paths of length zero and $F(W)$ is cyclically equivalent to~$W'$. The following result is due to Labardini-Fragoso.

\begin{theorem}[\cite{LabardiniFragoso16}, Theorem~6.1]
\label{thm:QPwelldefined}
Let $(\mathbb{S},\mathbb{M})$ be a marked bordered surface, and assume $(\mathbb{S},\mathbb{M})$ is not one of the following:
\begin{enumerate}
\item A sphere with $\leq5$ marked points.
\item An unpunctured disk with $\leq3$ marked points on its boundary.
\item A once-punctured disk with one marked point on its boundary.
\end{enumerate}
Then up to right equivalence, the quiver with potential associated to a signed triangulation of~$(\mathbb{S},\mathbb{M})$ depends only on the underlying tagged triangulation.
\end{theorem}

Let us say that a potential is \emph{reduced} if it is a sum of cycles of length $\geq3$. In the case where $Q$ is a quiver with reduced potential~$W$, and $k$ is a vertex of~$Q$ that is not contained in a cyclic path of length two, Derksen, Weyman, and Zelevinsky~\cite{DWZ08} define a new quiver with potential $\mu_k(Q,W)$ called the quiver with potential obtained by \emph{mutation} in the direction~$k$. It is well defined up to right equivalence and depends only on the right equivalence class of~$(Q,W)$. Although the general definition of the mutated quiver with potential $\mu_k(Q,W)$ is somewhat involved, it can be understood in a simple way for quivers with potential arising from triangulated surfaces, thanks to the following result of Labardini-Fragoso.

\begin{theorem}[\cite{LabardiniFragoso08}, Theorem~30]
\label{thm:flipmutation}
Let $(\mathbb{S},\mathbb{M})$ be a marked bordered surface satisfying the assumption of Theorem~\ref{thm:QPwelldefined}. Let $(T,\epsilon)$ be a signed triangulation of~$(\mathbb{S},\mathbb{M})$ and $(T',\epsilon)$ the signed triangulation obtained from $(T,\epsilon)$ by a flip of the arc~$k$. Then, up to right equivalence, 
\[
(Q(T'),W(T',\epsilon))=\mu_k(Q(T),W(T,\epsilon)).
\]
\end{theorem}

\subsection{The Jacobian algebra}

Let us consider once again a quiver $Q$. For any cyclic path $p=a_1\dots a_d$ and any arrow~$a$ in~$Q$, we define the \emph{cyclic derivative} of $p$ with respect to~$a$ by the formula 
\[
\partial_a(p)=\sum_{i:a_i=a}a_{i+1}\dots a_da_1\dots a_{i-1}\in\widehat{\Bbbk Q}.
\]
Extending this operation by linearity and continuity, we can define the cyclic derivative $\partial_a(W)\in\widehat{\Bbbk Q}$ of any potential~$W$ for~$Q$ with respect to the arrow~$a$. Then the \emph{Jacobian ideal} $\mathfrak{J}(Q,W)\subset\widehat{\Bbbk Q}$ is the closure of the two-sided ideal in the path algebra generated by the set $\{\partial_a(W):a\in Q_1\}$. The \emph{Jacobian algebra} of $(Q,W)$ is the quotient 
\[
J(Q,W)=\widehat{\Bbbk Q}/\mathfrak{J}(Q,W).
\]
It is easy to see from this definition that right equivalent quivers with potential have isomorphic Jacobian algebras. In the following, we will write $\mathcal{A}(Q,W)$ for the category of finite dimensional modules over the Jacobian algebra. It is an abelian category with finitely many isomorphism classes of simple objects in bijection with the vertices of~$Q$. Our main goal in this section is to describe a 3-Calabi-Yau triangulated category containing $\mathcal{A}(Q,W)$ as a distinguished subcategory.

\subsection{Hearts and tilting}

Let $\mathcal{D}$ denote a $\Bbbk$-linear triangulated category with shift functor~$[1]$. The notion of a t-structure on~$\mathcal{D}$ determines a full abelian subcategory of~$\mathcal{D}$ called the heart of the t-structure. In this paper we will be interested exclusively in bounded t-structures. By a heart in~$\mathcal{D}$, we will always mean the heart of a bounded t-structure, which can be characterized as follows.

\begin{definition}[\cite{Bridgeland07}, Lemma~3.2]
The \emph{heart} of a bounded t-structure on~$\mathcal{D}$ is a full additive subcategory $\mathcal{A}\subset\mathcal{D}$ such that 
\begin{enumerate}
\item If $j>k$ are integers, then $\Hom_\mathcal{D}(A[j],A[k])=0$ for all objects $A$,~$B\in\mathcal{A}$.
\item For every object $E\in\mathcal{D}$, there is a finite sequence of integers 
\[
k_1>k_2>\dots>k_s
\]
and a sequence of exact triangles 
\[
\xymatrix{
0=E_0 \ar[rr] & & E_1 \ar[ld] \ar[rr] & & E_2\ar[ld] \ar[r] & \cdots \ar[r] & E_{s-1} \ar[rr] & & \ar[ld] E_s=E \\
& A_1 \ar@{-->}[lu] & & A_2 \ar@{-->}[lu] & & & & A_s \ar@{-->}[lu]
}
\]
with $A_j\in\mathcal{A}[k_j]$ for each~$j$.
\end{enumerate}
\end{definition}

A heart in a triangulated category~$\mathcal{D}$ is a full abelian subcategory. To get other abelian subcategories of~$\mathcal{D}$, we use the operation of tilting. In this paper, we will only need a special case of the tilting construction which we describe now.

Let us say that a heart is of \emph{finite length} if it is noetherian and artinian as an abelian category. If $\mathcal{A}\subset\mathcal{D}$ is a finite length heart and $S\in\mathcal{A}$ is a simple object, then we write 
\[
S^\perp=\{E\in\mathcal{A}:\Hom_{\mathcal{A}}(S,E)=0\}, \quad ^\perp S=\{E\in\mathcal{A}:\Hom_{\mathcal{A}}(E,S)=0\}.
\]
Then the categories 
\[
\mu_S^-(\mathcal{A})=\langle S[1], {^\perp S}\rangle, \quad \mu_S^+(\mathcal{A})=\langle S^\perp,S[-1]\rangle
\]
are new hearts called the \emph{left tilt} and \emph{right tilt} of~$\mathcal{A}$ at~$S$, respectively. Here we use the notation~$\langle\mathcal{C}\rangle$ for the \emph{extension closure} of a collection $\mathcal{C}$ of objects in~$\mathcal{D}$. It is defined as the smallest full subcategory of~$\mathcal{D}$ which contains all objects in~$\mathcal{C}$ and is closed under extensions.

\subsection{Triangulated categories from surfaces}

In this paper, we will be interested in a class of 3-Calabi-Yau triangulated categories associated to triangulated surfaces. The existence of these categories is guaranteed by the following theorem.

\begin{theorem}[\cite{BridgelandSmith15}, Theorem~7.1]
\label{thm:category}
Let $(Q,W)$ be a quiver with reduced potential. Then there exists a corresponding 3-Calabi-Yau triangulated category $\mathcal{D}(Q,W)$ over~$\Bbbk$. It has a distinguished bounded t-structure whose heart is the category $\mathcal{A}(Q,W)$ of finite-dimensional modules over the Jacobian algebra $J(Q,W)$. Moreover, if $(Q',W')$ is another quiver with potential and $F:\widehat{\Bbbk Q}\rightarrow\widehat{\Bbbk Q'}$ is a right equivalence, then $F$ induces a canonical triangulated equivalence $\mathcal{D}(Q,W)\rightarrow\mathcal{D}(Q',W')$.
\end{theorem}

Explicitly, the category $\mathcal{D}(Q,W)$ of Theorem~\ref{thm:category} is defined as the full subcategory of the derived category of the complete Ginzburg dg-algebra of $(Q,W)$ consisting of dg-modules with finite-dimensional cohomology. Further details can be found in~\cite{KellerYang11}. For us the important property of this construction is the following result of Keller and Yang~\cite{KellerYang11}.

\begin{theorem}[\cite{BridgelandSmith15}, Theorem~7.3]
\label{thm:KellerYang}
Let $(Q,W)$ be a 2-acyclic quiver with potential, and let $(Q',W')=\mu_k(Q,W)$ be a quiver with potential obtained by mutation in the direction~$k$. Then there is a canonical pair of $\Bbbk$-linear triangulated equivalences 
\[
\Phi_\pm:\mathcal{D}(Q',W')\rightarrow\mathcal{D}(Q,W)
\]
such that if $S_k\in\mathcal{A}(Q,W)$ is the simple object corresponding to the vertex~$k$, then functors $\Phi_\pm$ induce tilting in the sense that 
\[
\Phi_\pm(\mathcal{A}(Q',W'))=\mu_{S_k}^\pm(\mathcal{A}(Q,W)).
\]
\end{theorem}

The category that we are interested in is obtained by applying Theorem~\ref{thm:category} to the quiver with potential associated to a triangulation of a marked bordered surface. In order to make use of the results of~\cite{BridgelandSmith15}, we consider only marked bordered surfaces of the following type.

\begin{definition}[\cite{BridgelandSmith15}]
\label{def:amenable}
A marked bordered surface $(\mathbb{S},\mathbb{M})$ is \emph{amenable} if it is not one of the following:
\begin{enumerate}
\item A closed surface with a single puncture.
\item A sphere with $\leq5$ punctures.
\item An unpunctured disk with $\leq4$ marked points on its boundary.
\item A once-punctured disk with one, two, or four marked points on its boundary.
\item A twice-punctured disk with two marked points on its boundary.
\item An annulus with one marked point on each boundary component.
\end{enumerate}
\end{definition}

If $\tau$ is a tagged triangulation of an amenable marked bordered surface $(\mathbb{S},\mathbb{M})$, then by Theorem~\ref{thm:QPwelldefined}, there is an associated quiver with potential $(Q(\tau),W(\tau))$ which is well defined up to right equivalence. Hence, by Theorem~\ref{thm:category} there is an associated triangulated category $\mathcal{D}(Q(\tau),W(\tau))$. If $\tau'$ is any other tagged triangulation of~$(\mathbb{S},\mathbb{M})$, then $\tau$ and~$\tau'$ are related by a sequence of flips of tagged arcs. It then follows from Theorems~\ref{thm:flipmutation} and~\ref{thm:KellerYang} that the categories $\mathcal{D}(Q(\tau),W(\tau))$ and $\mathcal{D}(Q(\tau'),W(\tau'))$ are equivalent. Hence the category $\mathcal{D}(Q(\tau),W(\tau))$ is determined by~$(\mathbb{S},\mathbb{M})$ up to a noncanonical equivalence.

\subsection{The exchange graph}

For any triangulated category~$\mathcal{D}$, the \emph{tilting graph} $\Tilt(\mathcal{D})$ is the graph whose vertices are the finite length hearts in~$\mathcal{D}$, where two vertices are connected by an edge if the corresponding hearts are related by a tilt at a simple object. In particular, we can take $\mathcal{D}=\mathcal{D}(Q,W)$ to be the category associated to a 2-acyclic quiver with potential $(Q,W)$. In this case, the graph $\Tilt(\mathcal{D})$ has a distinguished vertex given by the canonical heart $\mathcal{A}(Q,W)\subset\mathcal{D}(Q,W)$. We write $\Tilt_\Delta(\mathcal{D})$ for the connected component of~$\Tilt(\mathcal{D})$ containing this vertex.

Let $\Aut(\mathcal{D})$ denote the group of all triangulated autoequivalences of~$\mathcal{D}$. There is a natural action of this group on the tilting graph $\Tilt(\mathcal{D})$, and we write 
\[
\Aut_\Delta(\mathcal{D})\subset\Aut(\mathcal{D})
\]
for the subgroup preserving $\Tilt_\Delta(\mathcal{D})$. We let 
\[
\Nil_\Delta(\mathcal{D})\subset\Aut_\Delta(\mathcal{D})
\]
be the subgroup of autoequivalences that act trivially and write 
\[
\cAut_\Delta(\mathcal{D})=\Aut_\Delta(\mathcal{D})/\Nil_\Delta(\mathcal{D})
\]
for the quotient, which acts effectively on~$\Tilt_\Delta(\mathcal{D})$.

If $\mathcal{A}\in\Tilt_\Delta(\mathcal{D})$ and $S\in\mathcal{A}$ is a simple object, then $S$ is spherical, and hence by the work of Seidel and Thomas~\cite{SeidelThomas01}, there is an associated autoequivalence $\Tw_S\in\Aut(\mathcal{D})$ called a \emph{spherical twist}. It has the property $\Tw_S(\mu_S^-(\mathcal{A}))=\mu_S^+(\mathcal{A})$ by Proposition~7.1 of~\cite{BridgelandSmith15}. Moreover, the group $\Sph_\mathcal{A}(\mathcal{D})=\langle\Tw_S:S\in\mathcal{A}\text{ simple}\rangle\subset\Aut(\mathcal{D})$ is independent of the choice of vertex~$\mathcal{A}$, and therefore we can simply denote this group by $\Sph_\Delta(\mathcal{D})$. We write 
\[
\cSph_\Delta(\mathcal{D})\subset\cAut_\Delta(\mathcal{D})
\]
for the image of $\Sph_\Delta(\mathcal{D})$ in $\cAut_\Delta(\mathcal{D})$.

This group $\cSph_\Delta(\mathcal{D})$ acts on $\Tilt_\Delta(\mathcal{D})$, and the quotient 
\[
\Exch_\Delta(\mathcal{D})=\Tilt_\Delta(\mathcal{D})/\cSph_\Delta(\mathcal{D})
\]
is known in cluster theory as the \emph{exchange graph}. The quotient group 
\[
\mathcal{G}_\Delta(\mathcal{D})=\cAut_\Delta(\mathcal{D})/\cSph_\Delta(\mathcal{D})
\]
acts by symmetries on $\Exch_\Delta(\mathcal{D})$ and is known as the \emph{cluster modular group}.

In the case where the quiver with potential arises from a tagged triangulation of an amenable marked bordered surface $(\mathbb{S},\mathbb{M})$, we can give a simple description of these objects. Indeed, in this case one can define another graph $\Tri_{\bowtie}(\mathbb{S},\mathbb{M})$ whose vertices are tagged triangulations of~$(\mathbb{S},\mathbb{M})$, where two vertices are connected by an edge if the corresponding triangulations are related by a flip of a tagged arc. There is a natural action of the signed mapping class group $\MCG^\pm(\mathbb{S},\mathbb{M})$ on this graph $\Tri_{\bowtie}(\mathbb{S},\mathbb{M})$, and one has the following result.

\begin{theorem}[\cite{Allegretti21}, Theorem~10.1]
\label{thm:graphgroup}
Let $\tau$ be a tagged triangulation of an amenable marked bordered surface $(\mathbb{S},\mathbb{M})$. Let $(Q,W)=(Q(\tau),W(\tau))$ be the quiver with potential determined by~$\tau$, and let $\mathcal{D}=\mathcal{D}(Q,W)$ be the associated triangulated category. Then 
\begin{enumerate}
\item There is an isomorphism of graphs 
\[
\Tri_{\bowtie}(\mathbb{S},\mathbb{M})\cong\Exch_\Delta(\mathcal{D}).
\]
\item There is an isomorphism of groups 
\[
\MCG^\pm(\mathbb{S},\mathbb{M})\cong\mathcal{G}_\Delta(\mathcal{D}).
\]
\end{enumerate}
Under these isomorphisms, the action of the cluster modular group on the exchange graph coincides with the action of the signed mapping class group on the graph of tagged triangulations.
\end{theorem}

\section{The main results}
\label{sec:TheMainResults}

In this section, we formulate our main results in terms of the triangulated category introduced above.

\subsection{Stability conditions}

We begin by recalling the notion of a stability condition from~\cite{Bridgeland07}. If $\mathcal{A}$ is an abelian category, then a \emph{stability function} on~$\mathcal{A}$ is defined to be a group homomorphism $Z:K(\mathcal{A})\rightarrow\mathbb{C}$ such that for any nonzero object $E\in\mathcal{A}$, the complex number $Z(E)$ lies in the semi-closed upper half plane 
\[
\mathcal{H}=\{r\exp(\mathrm{i}\pi\phi):r>0\text{ and }0<\phi\leq1\}\subset\mathbb{C}.
\]
Given a stability function $Z:K(\mathcal{A})\rightarrow\mathbb{C}$, we can assign to any nonzero object $E\in\mathcal{A}$ a well defined \emph{phase} given by 
\[
\phi(E)=\frac{1}{\pi}\arg Z(E)\in(0,1].
\]
A nonzero object $E\in\mathcal{A}$ is said to be \emph{semistable} with respect to~$Z$ if every proper nonzero subject $A\subset E$ satisfies $\phi(A)\leq\phi(E)$.

If we are given a stability function on an abelian category~$\mathcal{A}$, then the semistable objects provide a way to filter arbitrary objects of~$\mathcal{A}$. More precisely, if $E\in\mathcal{A}$ is a nonzero object, then a \emph{Harder-Narasimhan filtration} of~$E$ is a finite sequence of subobjects 
\[
0=E_0\subset E_1\subset\dots\subset E_{n-1}\subset E_n=E
\]
such that each quotient $F_j=E_j/E_{j-1}$ is semistable, and 
\[
\phi(F_1)>\phi(F_2)>\dots>\phi(F_n).
\]
A stability function~$Z$ on~$\mathcal{A}$ is said to have the \emph{Harder-Narasimhan property} if every nonzero object of~$\mathcal{A}$ has a Harder-Narasimhan filtration. Using these concepts, we can give the following definition of a stability condition on a triangulated category.

\begin{definition}[\cite{Bridgeland07}, Proposition~5.3]
Let $\mathcal{D}$ be a triangulated category. Then a \emph{stability condition} $(\mathcal{A},Z)$ on~$\mathcal{D}$ consists of the heart~$\mathcal{A}$ of a bounded t-structure on~$\mathcal{D}$ together with a stability function~$Z$ on~$\mathcal{A}$ having the Harder-Narasimhan property.
\end{definition}

One can show using our definitions that if $\mathcal{A}\subset\mathcal{D}$ is a heart, then there is an isomorphism $K(\mathcal{A})\cong K(\mathcal{D})$. Thus a stability function on~$\mathcal{A}$ induces a homomorphism $Z:K(\mathcal{D})\rightarrow\mathbb{C}$ called then \emph{central charge}. In the examples that we consider, the Grothendieck group is a finite-rank lattice $K(\mathcal{D})\cong\mathbb{Z}^n$, and we will restrict attention to stability conditions satisfying the \emph{support property} from~\cite{KontsevichSoibelman08}: For any norm $\|\cdot\|$ on $K(\mathcal{D})\otimes_{\mathbb{Z}}\mathbb{R}$, there is a constant $C>0$ such that 
\[
\|\gamma\|<C\cdot|Z(\gamma)|
\]
for every class $\gamma\in K(\mathcal{D})$ represented by a semistable object. If we write $\Stab(\mathcal{D})$ for the set of stability conditions on~$\mathcal{D}$ satisfying this support property, then the main result of~\cite{Bridgeland07} can be formulated as follows.

\begin{theorem}[\cite{Bridgeland07}, Theorem~1.2]
\label{thm:localiso}
The set $\Stab(\mathcal{D})$ has the structure of a complex manifold such that the map 
\begin{equation}
\label{eqn:forgetful}
\Stab(\mathcal{D})\rightarrow\Hom_{\mathbb{Z}}(K(\mathcal{D}),\mathbb{C})
\end{equation}
taking a stability condition to its central charge is a local isomorphism.
\end{theorem}

\subsection{Group actions}

To describe the space of stability conditions on a triangulated category, one typically considers quotients by various group actions. For any triangulated category~$\mathcal{D}$, there is a natural action of the group $\Aut(\mathcal{D})$ on $\Stab(\mathcal{D})$. If $(\mathcal{A},Z)\in\Stab(\mathcal{D})$ is a stability condition and $\Phi\in\Aut(\mathcal{D})$ is an autoequivalence of~$\mathcal{D}$, then $\Phi\cdot(\mathcal{A},Z)=(\mathcal{A}',Z')$ is defined by 
\[
\mathcal{A}'=\Phi(\mathcal{A}), \quad Z'(E)=Z(\Phi^{-1}(E))
\]
for an object $E\in\mathcal{A}'$.

In the case where $\mathcal{D}=\mathcal{D}(Q,W)$ is the 3-Calabi-Yau triangulated category associated to a 2-acyclic quiver with potential~$(Q,W)$, there is a distinguished component $\Stab_\Delta(\mathcal{D})\subset\Stab(\mathcal{D})$ of the space of stability conditions on~$\mathcal{D}$. It is the component containing those stability conditions of the form~$(\mathcal{A},Z)$ where $\mathcal{A}$ is a vertex lying in the distinguished component $\Tilt_\Delta(\mathcal{D})\subset\Tilt(\mathcal{D})$ of the tilting graph. As explained in Section~7.7 of~\cite{BridgelandSmith15}, the subgroup $\Aut_\Delta(\mathcal{D})\subset\Aut(\mathcal{D})$ preserves this component while $\Nil_\Delta(\mathcal{D})\subset\Aut_\Delta(\mathcal{D})$ acts trivially on it. Hence there is an induced action of~$\cAut_\Delta(\mathcal{D})=\Aut_\Delta(\mathcal{D})/\Nil_\Delta(\mathcal{D})$ on $\Stab_\Delta(\mathcal{D})$. We write 
\[
\Sigma(Q,W)=\Stab_\Delta(\mathcal{D})/\cSph_\Delta(\mathcal{D})
\]
for the quotient of the distinguished component by the subgroup $\cSph_\Delta(\mathcal{D})\subset\cAut_\Delta(\mathcal{D})$.

In addition to the action of the group of autoequivalences, there is a natural action of the group of complex numbers on the space of stability conditions. This action has the property that the forgetful map~\eqref{eqn:forgetful} is $\mathbb{C}$-equivariant where $z\in\mathbb{C}$ acts on the space of central charges by mapping $Z\in\Hom_{\mathbb{Z}}(K(\mathcal{D}),\mathbb{C})$ to~$e^{-\mathrm{i}\pi z}\cdot Z$.

\subsection{Stability conditions from surfaces}

Let $\tau_0$ be a tagged triangulation of an amenable marked bordered surface $(\mathbb{S},\mathbb{M})$. We will denote by $(Q,W)=(Q(\tau_0),W(\tau_0))$ the associated quiver with potential and by $\mathcal{D}=\mathcal{D}(Q,W)$ the associated 3-Calabi-Yau triangulated category. In~\cite{BridgelandSmith15}, Bridgeland and Smith gave a description of the quotient $\Stab_\Delta(\mathcal{D})/\cAut_\Delta(\mathcal{D})$ in terms of meromorphic quadratic differentials on Riemann surfaces. Here we employ a slightly modified version of this result involving the space $\mathcal{Q}^\pm(\mathbb{S},\mathbb{M})$ introduced in Section~\ref{sec:SignedDifferentials}. In the following statement, we consider the $\mathbb{C}$-action on this space where a complex number $z\in\mathbb{C}$ sends a quadratic differential $\phi$ to the rescaled differential~$e^{-2\pi iz}\cdot\phi$.

\begin{theorem}[\cite{Allegretti21}, Theorem~10.3]
\label{thm:stabquad}
Take notation as in the previous paragraph. Then there is an isomorphism of complex manifolds 
\begin{equation}
\label{eqn:stabquad}
\Sigma(Q,W)\cong\mathcal{Q}^\pm(\mathbb{S},\mathbb{M})
\end{equation}
which is equivariant with respect to the actions of $\mathcal{G}_\Delta(\mathcal{D})\cong\MCG^\pm(\mathbb{S},\mathbb{M})$ and~$\mathbb{C}$.
\end{theorem}

If $\mathcal{A}\in\Tilt_\Delta(\mathcal{D})$, then there is a corresponding subset $\Stab(\mathcal{A})\subset\Stab(\mathcal{D})$ consisting of stability conditions of the form $(\mathcal{A},Z)$ for some stability function~$Z$. As explained in~\cite{BridgelandSmith15}, Section~7.7, it is a locally closed subset of the space of stability conditions. By Theorem~\ref{thm:graphgroup}, there is a unique corresponding vertex $\tau\in\Tri_{\bowtie}(\mathbb{S},\mathbb{M})$. An important property of the isomorphism~\eqref{eqn:stabquad} is that if $\sigma\in\Sigma(Q,W)$ is any point corresponding to a stability condition in the interior of~$\Stab(\mathcal{A})$, then $\sigma$ is mapped to a complete, saddle-free quadratic differential~$\phi$ having tagged WKB~triangulation~$\tau$. There is an isomorphism $K(\mathcal{A})\cong\widehat{H}(\phi)$ which identifies identifies the stability function~$Z_\sigma$ of~$\sigma$ with the period map~$Z_\phi$ and identifies the classes of simple objects in~$\mathcal{A}$ with standard saddle classes of~$\phi$.

\subsection{Positive points of cluster varieties}

In the theory of cluster algebras and cluster varieties, a \emph{seed} is defined to be an ordered triple $\mathbf{i}=(\Gamma,\{e_i\}_{i\in I}, \langle-,-\rangle)$ consisting of a lattice $\Gamma$ of finite rank, a basis $\{e_i\}_{i\in I}$ for~$\Gamma$ indexed by some finite set~$I$, and an integer-valued skew form $\langle-,-\rangle$. An isomorphism of seeds is defined to be an isomorphism of the underlying lattices that preserves the distinguished bases and the skew forms.

For example, suppose that $(Q,W)$ is a 2-acyclic quiver with potential and $\mathcal{D}=\mathcal{D}(Q,W)$ is the associated 3-Calabi-Yau triangulated category. Then we get a lattice $\Gamma=K(\mathcal{D})$ of finite rank. Given a heart $\mathcal{A}\in\Tilt_\Delta(\mathcal{D})$, there are finitely many simple objects $S_i\in\mathcal{A}$~($i\in I$) up to isomorphism, and their classes $e_i=[S_i]\in K(\mathcal{A})\cong K(\mathcal{D})$ form a basis for~$\Gamma$. There is a bilinear form $\langle-,-\rangle$ on~$\Gamma$ given by the Euler pairing 
\[
\langle [E],[F]\rangle=\sum_{i\in\mathbb{Z}}(-1)^i\dim_\Bbbk\Hom_{\mathcal{D}}^i(E,F)
\]
where $\Hom_{\mathcal{D}}^i(E,F)=\Hom_{\mathcal{D}}(E,F[i])$. The 3-Calabi-Yau property of~$\mathcal{D}$ implies that this form is skew-symmetric. In this way, we see that there is a seed naturally associated to any vertex $\mathcal{A}\in\Tilt_\Delta(\mathcal{D})$. If $\mathcal{A}'$ is obtained from~$\mathcal{A}$ by applying a spherical twist, then the seed associated to~$\mathcal{A}$ is isomorphic to the one associated to~$\mathcal{A}$. It follows that, up to isomorphism, there is a well defined seed $\mathbf{i}_t$ associated to each vertex $t\in\Exch_\Delta(\mathcal{D})$.

We say that two seeds $\mathbf{i}=(\Gamma,\{e_i\}_{i\in I}, \langle-,-\rangle)$ and $\mathbf{i}'=(\Gamma',\{e_i'\}_{i\in I}, \langle-,-\rangle')$ are related by \emph{mutation} at $k\in I$ if we have $\Gamma'=\Gamma$ and $\langle-,-\rangle'=\langle-,-\rangle$ and if the bases are related by 
\begin{equation}
\label{eqn:basismutation}
e_j'=
\begin{cases}
-e_k & \text{if $j=k$} \\
e_j+[\langle e_k,e_j\rangle]_+\cdot e_k & \text{if $j\neq k$}
\end{cases}
\end{equation}
where we write $[m]_+=\max(m,0)$.

Continuing the example from above, suppose that $t$,~$t'\in\Exch_\Delta(\mathcal{D})$ are connected by an edge. Then $t$ and~$t'$ correspond to hearts $\mathcal{A}$,~$\mathcal{A}'\in\Tilt_\Delta(\mathcal{D})$ so that if $S_i$ ($i\in I$) are the simple objects of~$\mathcal{A}$ up to isomorphism, then we have $\mathcal{A}'=\mu_{S_k}^+(\mathcal{A})$ for some~$k$. Then the proof of Proposition~7.1 in~\cite{BridgelandSmith15} shows that $\mathcal{A}'$ has simple objects $S_i'$~($i\in I$) in bijection with those of~$\mathcal{A}$, and the classes $e_i=[S_i]$ and $e_i'=[S_i']$ are related by~\eqref{eqn:basismutation}. It follows that, up to composition with seed isomorphisms, the seeds $\mathbf{i}_t$ and~$\mathbf{i}_{t'}$ are related by mutation at~$k\in I$.

To any seed $\mathbf{i}=(\Gamma,\{e_i\}_{i\in I}, \langle-,-\rangle)$, we associate the real manifold 
\[
\mathcal{T}_{\mathbf{i}}=\Hom_{\mathbb{Z}}(\Gamma,\mathbb{Z})\otimes_{\mathbb{Z}}\mathbb{R}_{>0}\cong\mathbb{R}_{>0}^I.
\]
For any $\gamma\in\Gamma$, there is a function $X_{\mathbf{i},\gamma}:\mathcal{T}_{\mathbf{i}}\rightarrow\mathbb{R}_{>0}$ given by $X_{\mathbf{i},\gamma}(g\otimes h)=h^{g(\gamma)}$. Note that if $\mathbf{i}'$ is a seed which is isomorphic to~$\mathbf{i}$, then the two seeds have the same lattice, and so we have a canonical homeomorphism $\mathcal{T}_{\mathbf{i}}\cong\mathcal{T}_{\mathbf{i}'}$. We can therefore think of~$\mathcal{T}_{\mathbf{i}}$ as depending only on the isomorphism class of~$\mathbf{i}$. For $k\in I$, we define a homeomorphism $\mu_k:\mathcal{T}_{\mathbf{i}}\rightarrow\mathcal{T}_{\mu_k(\mathbf{i})}$ by the formula 
\[
\mu_k^*(X_{\mu_k(\mathbf{i}),\gamma})=X_{\mathbf{i},\gamma}\cdot(1+X_{\mathbf{i},e_k})^{\langle\gamma,e_k\rangle}.
\]
We then define a space by gluing the manifolds $\mathcal{T}_{\mathbf{i}}$ using these homeomorphisms.

\begin{definition}
The \emph{enhanced Teichm\"uller space} of $Q$ is the space 
\[
\mathcal{T}(Q)=\left(\coprod_{t\in\Exch_\Delta(\mathcal{D})}\mathcal{T}_{\mathbf{i}_t}\right)/\sim
\]
where $\sim$ is the relation defined by gluing each pair of spaces $\mathcal{T}_{\mathbf{i}}$ and $\mathcal{T}_{\mu_k(\mathbf{i})}$ by the homeomorphism $\mu_k:\mathcal{T}_{\mathbf{i}}\rightarrow\mathcal{T}_{\mu_k(\mathbf{i})}$.
\end{definition}

This space $\mathcal{T}(Q)$ arises as the set of $\mathbb{R}_{>0}$-valued points of the cluster Poisson variety of~$Q$ (see Section~9.6 of~\cite{Allegretti21} for a definition of the cluster Poisson variety in our setup). In particular, it is independent of the choice of potential~$W$. Note that the action of the cluster modular group $\mathcal{G}_\Delta(\mathcal{D})$ on the exchange graph $\Exch_\Delta(\mathcal{D})$ gives rise to a natural action of $\mathcal{G}_\Delta(\mathcal{D})$ on~$\mathcal{T}(Q)$.

\subsection{Positive points of cluster varieties from surfaces}

Let $\tau_0$ once again be a tagged triangulation of an amenable marked bordered surface $(\mathbb{S},\mathbb{M})$. We write $(Q,W)=(Q(\tau_0),W(\tau_0))$ for the associated quiver with potential and $\mathcal{D}=\mathcal{D}(Q,W)$ for the associated triangulated category. We can then apply the construction described above to get a space $\mathcal{T}(Q)$, and we have the following result.

\begin{theorem}
\label{thm:clusterteich}
Take notation as in the previous paragraph. Then there is a homeomorphism 
\[
\mathcal{T}(Q)\cong\mathcal{T}^\pm(\mathbb{S},\mathbb{M})
\]
which is equivariant with respect to the action of $\mathcal{G}_\Delta(\mathcal{D})\cong\MCG^\pm(\mathbb{S},\mathbb{M})$.
\end{theorem}

\begin{proof}
Associated to any vertex $t\in\Exch_\Delta(\mathcal{D})$ is a seed $\mathbf{i}_t=(\Gamma,\{e_i\}_{i\in I}, \langle-,-\rangle)$ well defined up to isomorphism. The functions $X_{\mathbf{i}_t,e_i}$ provide a homeomorphism $\mathcal{T}_{\mathbf{i}_t}\cong\mathbb{R}_{>0}^I$. If $\tau$ is the tagged triangulation corresponding to~$t$ under the isomorphism of Theorem~\ref{thm:graphgroup}, then by Lemma~9.10 of~\cite{BridgelandSmith15}, the tagged arcs of $\tau$ are naturally in bijection with the basis elements $e_i$. Let us write $\alpha_i$ for the tagged arc corresponding to~$e_i$. Then there is a cluster coordinate $X_{\tau,\alpha_i}$ for each of these tagged arcs, and these coordinates provide a homeomorphism $\mathcal{T}^\pm(\mathbb{S},\mathbb{M})\cong\mathbb{R}_{>0}^I$. Hence there is a homeomorphism $\mathcal{T}^\pm(\mathbb{S},\mathbb{M})\cong\mathcal{T}_{\mathbf{i}_t}$.

Next suppose that $t'$ is connected to~$t$ by an edge of~$\Exch_\Delta(\mathcal{D})$, and let $\tau'$ be the tagged triangulation corresponding to~$t'$ under the isomorphism of Theorem~\ref{thm:graphgroup}. Then $\tau'$ is obtained from~$\tau$ by a flip of some tagged arc~$\gamma$. By Lemma~9.10 of~\cite{BridgelandSmith15}, one has $\langle e_i,e_j\rangle=\varepsilon_{\alpha_i\alpha_j}^T$. Using this fact together with~\eqref{eqn:basismutation}, one can check that the transformation used to glue $\mathcal{T}_{\mathbf{i}_t}$ to~$\mathcal{T}_{\mathbf{i}_{t'}}$ coincides with the transformation in Proposition~\ref{prop:changecoordinates}. Hence there is a canonical homeomorphism $\mathcal{T}(\mathbb{S},\mathbb{M})\cong\mathcal{T}(Q)$. The equivariance follows from Theorem~\ref{thm:graphgroup}.
\end{proof}

For any vertex $t\in\Exch_\Delta(\mathcal{D})$, let $\mathbf{i}_t=(\Gamma,\{e_i\}_{i\in I}, \langle-,-\rangle)$ be the corresponding seed, well defined up to isomorphism. Let $\tau\in\Tri_{\bowtie}(\mathbb{S},\mathbb{M})$ be the tagged triangulation corresponding to~$t$ under the isomorphism of Theorem~\ref{thm:graphgroup}. From the proof of Theorem~\ref{thm:clusterteich}, one see that for each $i\in I$, the function $X_{\mathbf{i},e_i}:\mathcal{T}_\mathbf{i}\rightarrow\mathbb{R}_{>0}$ is identified with the cluster coordinate $X_{\tau,\alpha_i}:\mathcal{T}^\pm(\mathbb{S},\mathbb{M})\rightarrow\mathbb{R}_{>0}$ associated to a corresponding tagged arc~$\alpha_i$ of~$\tau$. The correspondence between basis elements $e_i$ and tagged arcs~$\alpha_i$ is such that if $\phi$ is a complete, saddle-free differential with tagged WKB~triangulation~$\tau$, then $e_i$ is the class in $K(\mathcal{A})\cong\widehat{H}(\phi)$ of the standard saddle connection corresponding to~$\alpha_i$. See~\cite{BridgelandSmith15}, Section~10.4 for details.

\subsection{From stability conditions to Teichm\"uller space}

If we now write $\widehat{\Psi}:\Sigma(Q,W)\rightarrow\mathcal{T}(Q)$ for the composition of $\Psi^\pm:\mathcal{Q}^\pm(\mathbb{S},\mathbb{M})\rightarrow\mathcal{T}^\pm(\mathbb{S},\mathbb{M})$ with the homeomorphisms of Theorems~\ref{thm:stabquad} and~\ref{thm:clusterteich}, then we have our first main result.

\begin{theorem}
Let $(Q,W)$ be the quiver with potential associated to a tagged triangulation of an amenable marked bordered surface, and let $\mathcal{D}=\mathcal{D}(Q,W)$ be the associated 3-Calabi-Yau triangulated category. Then there is a  $\mathcal{G}_\Delta(\mathcal{D})$-equivariant continuous map 
\[
\widehat{\Psi}:\Sigma(Q,W)\rightarrow\mathcal{T}(Q)
\]
from the space of stability conditions to the enhanced Teichm\"uller space.
\end{theorem}

\begin{proof}
The continuity of $\widehat{\Psi}$ follows from Proposition~\ref{prop:Psipmcontinuous}, while the equivariance follows from the $\MCG^\pm(\mathbb{S},\mathbb{M})$-equivariance of~$\Psi^\pm$ combined with the equivariance properties in Theorems~\ref{thm:stabquad} and~\ref{thm:clusterteich}.
\end{proof}

Next suppose that $\sigma=(\mathcal{A},Z)$ is a stability condition in the interior of~$\Stab(\mathcal{A})$ for some heart $\mathcal{A}\in\Tilt_\Delta(\mathcal{D})$. We will consider the 1-parameter family of stability conditions $\sigma_R=(\mathcal{A},R\cdot Z)$ for~$R>0$. The heart~$\mathcal{A}$ determines a seed~$\mathbf{i}=(\Gamma,\{e_i\}_{i\in I}, \langle-,-\rangle)$ whose underlying lattice is~$\Gamma=K(\mathcal{D})$, and hence we have a function $X_{\mathbf{i},\gamma}:\mathcal{T}_{\mathbf{i}}\rightarrow\mathbb{R}_{>0}$ for every $\gamma\in K(\mathcal{D})$. In the following, we will write $X_{\sigma,\gamma}(R)=X_{\mathbf{i},\gamma}(\widehat{\Psi}(\sigma_R))$.

\begin{theorem}
Take notation as in the last paragraph. Then 
\[
X_{\sigma,\gamma}(R)\cdot\exp(R\cdot\Re Z(\gamma))\rightarrow1 \quad \text{as $R\rightarrow\infty$}.
\]
\end{theorem}

\begin{proof}
By the remarks following Theorem~\ref{thm:stabquad}, the stability condition $\sigma_R$ corresponds to a quadratic differential~$R^2\cdot\phi$ where $\phi$ is complete and saddle-free. There is an isomorphism $K(\mathcal{A})\cong\widehat{H}(\phi)$ that identifies the stability function~$Z$ with the period map~$Z_\phi$ and identifies the basis elements provided by the seed~$\mathbf{i}$ with standard saddle connections. By the remark following Theorem~\ref{thm:clusterteich}, the functions $X_{\mathbf{i},e_i}(R)$ are identified with cluster coordinates on~$\mathcal{T}^\pm(\mathbb{S},\mathbb{M})$ corresponding to these standard saddle-connections. The theorem therefore follows from Theorem~\ref{thm:clusterasymptotics}.
\end{proof}

\bibliographystyle{amsplain}

\begin{thebibliography}{99}

\bibitem{Allegretti18} Allegretti, D.G.L. (2018). Stability conditions and cluster varieties from quivers of type~$A$. \emph{Advances in Mathematics}, \textbf{337}, 260--293.

\bibitem{Allegretti19} Allegretti, D.G.L. (2019). Voros symbols as cluster coordinates. \emph{Journal of Topology}, \textbf{12}(4), 1031--1068.

\bibitem{Allegretti20} Allegretti, D.G.L. (2020). On the wall-crossing formula for quadratic differentials. \texttt{arXiv:2006.08059 [math.GT]}.

\bibitem{Allegretti21} Allegretti, D.G.L. (2021). Stability conditions, cluster varieties, and Riemann-Hilbert problems from surfaces. \emph{Advances in Mathematics}, \textbf{380}.

\bibitem{AllegrettiBridgeland20} Allegretti, D.G.L. and Bridgeland, T. (2020). The monodromy of meromorphic projective structures. \emph{Transactions of the AMS}, \textbf{373}(9), 6321--6367.

\bibitem{BapatDeopurkarLicata20} Bapat, A., Deopurkar, A., and Licata, M. (2020). A Thurston compactification of the space of stability conditions. \texttt{arXiv:2011.07908 [math.RT]}.

\bibitem{BiquardBoalch04} Biquard, O. and Boalch, P. (2004). Wild non-abelian Hodge theory on curves. \emph{Compositio Mathematica}, \textbf{140}(1), 179--204.

\bibitem{BAGG97} Biswas, I., Ar\'es-Gastesi, P. and Govindarajan, S. (1997). Parabolic Higgs bundles and Teichm\"uller spaces for punctured surfaces. \emph{Transactions of the American Mathematical Society}, \textbf{394}(4), 1551--1560.

\bibitem{Bridgeland07} Bridgeland, T. (2007). Stability conditions on triangulated categories. \emph{Annals of Mathematics}, \textbf{166}(2), 317--345.

\bibitem{BridgelandSmith15} Bridgeland, T. and Smith, I. (2015). Quadratic differentials as stability conditions. \emph{Publications Math\'ematiques de l'Institut des Hautes \'Etudes Scientifiques}, \textbf{121}(1), 155--278.

\bibitem{DW07} Daskalopoulos, G. and Wentworth, R.A. (2007). Harmonic maps and Teichm\"uller theory. In Handbook of Teichm\"uller theory I, \emph{IRMA Lectures in Mathematics and Theoretical Physics}, \textbf{11}, 33--109.

\bibitem{DWZ08} Derksen, H., Weyman, J., and Zelevinsky, A. (2008). Quivers with potentials and their representations~I: Mutations. \emph{Selecta Mathematica}, \textbf{14}(1), 59--119.

\bibitem{DHKK14} Dimitrov, G., Haiden, F., Katzarkov, L., and Kontsevich, M. (2017). Dynamical systems and categories. \emph{The influence of Solomon Lefschetz in geometry and topology}, \textbf{621}, 133--170.

\bibitem{DumasNeitzke20} Dumas, D. and Neitzke, A. (2020). Opers and nonabelian Hodge: numerical studies. \texttt{arXiv:2007.00503 [math.DG]}.

\bibitem{EellsLemaire81} Eells, J. and Lemaire, L. (1981). Deformations of metrics and associated harmonic maps. \emph{Geometry and Analysis, Patodi Memorial Volume}, \textbf{90}(1), 33--45.

\bibitem{EellsSampson64} Eells, J. and Sampson, J.H. (1964). Harmonic mappings of Riemannian manifolds. \emph{American Journal of Mathematics}, \textbf{86}(1), 109--160.

\bibitem{Fan21} Fan, Y.W. (2021). Systolic inequalities for K3 surfaces via stability conditions. \emph{Mathematische Zeitschrift}, \textbf{384}, 1--23.

\bibitem{FanFilip20} Fan, Y.W. and Filip, S. (2020). Asymptotic shifting numbers in triangulated categories. \texttt{arXiv:2008.06159 [math.AG]}.

\bibitem{FFHKL21} Fan, Y.W., Filip, S., Haiden, F., Katzarkov, L., and Liu, Y. (2021). On pseudo-Anosov autoequivalences. \emph{Advances in Mathematics}, \textbf{384}.

\bibitem{FanFuOuchi21} Fan, Y.W., Fu, L., and Ouchi, G. (2021). Categorical polynomial entropy. \emph{Advances in Mathematics}, \textbf{383}.

\bibitem{FockGoncharov1} Fock, V.V. and Goncharov, A.B. (2006). Moduli spaces of local systems and higher Teichm\"uller theory. \emph{Publications Math\'ematiques de l'Institut des Hautes \'Etudes Scientifiques}, \textbf{103}(1), 1--211.

\bibitem{FockGoncharov2} Fock, V.V. and Goncharov, A.B. (2007). Dual Teichm\"uller and lamination spaces. In \emph{Handbook of Teichm\"uller theory I, IRMA Lectures in Mathematics and Theoretical Physics}, \textbf{11}, 647--684.

\bibitem{FominShapiroThurston08} Fomin, S., Shapiro, M., and Thurston, D. (2008). Cluster algebras and triangulated surfaces. Part~I: Cluster complexes. \emph{Acta Mathematica}, \textbf{201}(1), 83--146.

\bibitem{GaiottoMooreNeitzke13} Gaiotto, D., Moore, G.W., and Neitzke, A. (2013). Wall-crossing, Hitchin systems, and the WKB approximation. \emph{Advances in Mathematics}, \textbf{234}(2013), 239--403.

\bibitem{Gupta19} Gupta, S. (2019). Harmonic maps and wild Teichm\"uller spaces. \emph{Journal of Topology and Analysis}, 1--45.

\bibitem{Han96} Han, Z.C. (1996). Remarks on the geometric behavior of harmonic maps between surfaces. \emph{Elliptic and parabolic methods in geometry}, 57--66.

\bibitem{Hitchin87} Hitchin, N.J. (1987). The self-duality equations on a Riemann surface. \emph{Proceedings of the London Mathematical Society}, \textbf{3}(1), 59--126.

\bibitem{Haiden21} Haiden, F. (2021). 3-d Calabi-Yau categories for Teichm\"uller theory. \texttt{arXiv:2104.06018 [math.AG]}.

\bibitem{HKK17} Haiden, F., Katzarkov, L., and Kontsevich, M. (2017). Flat surfaces and stability structures. \emph{Publications Math\'ematiques de l'Institut des Hautes \'Etudes Scientifiques}, \textbf{126}(1), 247--318.

\bibitem{Hartman67} Hartman, P. (1967). On homotopic harmonic maps. \emph{Canadian Journal of Mathematics}, \textbf{19}, 673--687.

\bibitem{Jost13} Jost, J. (2013). \emph{Partial Differential Equations}. Graduate Texts in Mathematics. Springer-Verlag.

\bibitem{KellerYang11} Keller, B. and Yang, D. (2011). Derived equivalences from mutations of quivers with potential. \emph{Advances in Mathematics}, \textbf{226}(3), 2118--2168.

\bibitem{KontsevichSoibelman08} Kontsevich, M. and Soibelman, Y. (2008). Stability structures, motivic Donaldson-Thomas invariants, and cluster transformations. \texttt{arXiv:0811.2435 [math.AG]}.

\bibitem{LabardiniFragoso08} Labardini-Fragoso, D. (2008). Quivers with potential associated to triangulated surfaces. \emph{Proceedings of the London Mathematical Society}, \textbf{98}(3), 797--839.

\bibitem{LabardiniFragoso16} Labardini-Fragoso, D. (2016). Quivers with potential associated to triangulated surfaces, part IV: Removing boundary assumptions. \emph{Selecta Mathematica}, \textbf{22}(1), 145--189.

\bibitem{Lemaire82} Lemaire, L. (1982). Boundary value problems for harmonic and minimal maps of surfaces into manifolds. \emph{Annali della Scuola Normale Superiore di Pisa, Classe di Scienze}, \textbf{9}(1), 91--103.

\bibitem{Lohkamp91} Lohkamp, J. (1991). Harmonic diffeomorphisms and Teichm\"uller theory. \emph{Manuscripta Mathematica}, \textbf{71}(1), 339--360.

\bibitem{Penner12} Penner, R.C. (2012). \emph{Decorated Teichm\"uller theory}. European Mathematical Society.

\bibitem{Sabbah99} Sabbah, C. (1999). Harmonic metrics and connections with irregular singularities. \emph{Annales de l'institut Fourier}, \textbf{49}(4), 1265--1291.

\bibitem{Sagman19} Sagman, N. (2019). Infinite energy equivariant harmonic maps, domination, and anti-de Sitter 3-manifolds. \texttt{arXiv:1911.06937 [math.DG]}.

\bibitem{Sampson78} Sampson, J.H. (1978). Some properties and applications of harmonic mappings. \emph{Annales Scientifiques de l'\'Ecole Normale Sup\'erieure}, \textbf{11}(2), 211--228.

\bibitem{SchoenYau76} Schoen, R. and Yau, S.T. (1976). Harmonic maps and the topology of stable hypersurfaces and manifolds with non-negative Ricci curvature. \emph{Commentarii Mathematici Helvetici}, \textbf{51}(1), 333--341.

\bibitem{SchoenYau78} Schoen, R. and Yau, S.T. (1978). On univalent harmonic maps between surfaces. \emph{Inventiones mathematicae}, \textbf{44}(3), 265--278.

\bibitem{SeidelThomas01} Seidel, P. and Thomas, R. (2001). Braid group actions on derived categories. \emph{Duke Mathematical Journal}, \textbf{108}(1), 37--108.

\bibitem{Simpson90} Simpson, C.T. (1990). Harmonic bundles on noncompact curves. \emph{Journal of the American Mathematical Society}, \textbf{3}(3), 713--770.

\bibitem{Strebel84} Strebel, K. (1984). \emph{Quadratic differentials}. Springer-Verlag.

\bibitem{Wolf89} Wolf, M. (1989). The Teichm\"uller theory of harmonic maps. \emph{Journal of Differential Geometry}, \textbf{29}(2), 449--479.

\bibitem{Wolf91} Wolf, M. (1991). Infinite energy harmonic maps and degeneration of hyperbolic surfaces in moduli space. \emph{Journal of Differential Geometry}, \textbf{33}(2), 487--539.

\end{thebibliography}

\end{document}